\documentclass[10pt,reqno,a4paper]{amsart}
\usepackage[dvipsnames, table]{xcolor}
\usepackage{amsthm}
\theoremstyle{plain}
\usepackage{amssymb}
\usepackage{marvosym}
\usepackage{bm}
\usepackage{mathrsfs}
\usepackage{enumerate}  
\usepackage{mathtools}
\usepackage{tikz-cd}
\usepackage{tikz}
\usepackage{graphicx}
\usepackage{float}
\usepackage[titletoc,title]{appendix}
\usepackage{marginnote}
\usepackage{todonotes}
\usepackage{etoolbox}
\usepackage[all]{xy}
\makeatletter
\patchcmd{\Ginclude@eps}{"#1"}{#1}{}{}
\makeatother
\usepackage[outdir=./]{epstopdf}
\usepackage[numbers]{natbib}
\usepackage[utf8]{inputenc}
\usepackage[english]{babel}
\usepackage{changepage}
\usepackage{makecell}
\usepackage{enumitem}
\usepackage{comment}
\numberwithin{equation}{section}

\usepackage[colorlinks,citecolor=blue]{hyperref}

\definecolor{lightblue}{HTML}{1F88CD}
\definecolor{lightgrey}{HTML}{727272}
\definecolor{lightblue2}{HTML}{009EC1}
\definecolor{mypink}{HTML}{FD00B0}
\definecolor{lightred}{HTML}{ff4d4d}

\usepackage{hyperref}
\hypersetup{
	colorlinks=true,
    linkcolor={OliveGreen},
    citecolor={blue},
	urlcolor={black}
}

\newtheorem*{theorem*}{Theorem}
\newtheorem{theorem}{Theorem}[section]
\newtheorem{corollary}[theorem]{Corollary}
\newtheorem{lemma}[theorem]{Lemma}

\newtheorem{proposition}[theorem]{Proposition}
\theoremstyle{definition}

\theoremstyle{definition}
\newtheorem{definition}[theorem]{Definition}
\theoremstyle{definition}
\newtheorem{remark}[theorem]{Remark}
\theoremstyle{definition}

\theoremstyle{definition}

\theoremstyle{definition}

\theoremstyle{definition}

\theoremstyle{definition}

\theoremstyle{definition}
\newtheorem{question!}[theorem]{Question!}
\theoremstyle{definition}

\makeatletter
\newcommand*\sbt{\mathpalette\sbt@{.75}}
\newcommand*\sbt@[2]{\mathbin{\vcenter{\hbox{\scalebox{#2}{$\m@th#1\bullet$}}}}}
\makeatother

\newcommand{\ra}{\rightarrow}
\newcommand{\xra}{\xrightarrow}

\newcommand{\wt}{\widetilde}

\newcommand{\sst}{\subset}

\newcommand{\bR}{\bm{\mathrm{R}}}
\newcommand{\bL}{\bm{\mathrm{L}}}

\newcommand{\D}{\mathrm{D}}

\newcommand{\KK}{\mathrm{K}}

\newcommand{\ZZ}{\mathbb{Z}}

\newcommand{\QQ}{\mathbb{Q}}

\newcommand{\CC}{\mathbb{C}}
\newcommand{\PP}{\mathbb{P}}

\newcommand{\ch}{\mathrm{ch}}

\newcommand{\perf}{\mathrm{perf}}

\newcommand{\pr}{\mathrm{pr}}

\newcommand{\Hilb}{\mathrm{Hilb}}
\newcommand{\Knum}{\mathrm{K}_{\mathrm{num}}}

\renewcommand{\Re}{\operatorname{Re}}
\renewcommand{\Im}{\operatorname{Im}}

\DeclareMathOperator{\identity}{id}
\DeclareMathOperator{\im}{im}

\DeclareMathOperator{\Ext}{Ext}

\DeclareMathOperator{\Hom}{Hom}
\DeclareMathOperator{\RHom}{RHom}

\DeclareMathOperator{\ext}{ext}

\DeclareMathOperator{\Spec}{Spec}
\DeclareMathOperator{\Pic}{Pic}

\DeclareMathOperator{\cone}{cone}
\DeclareMathOperator{\Stab}{Stab}
\DeclareMathOperator{\Gr}{Gr}

\newcommand{\GL}{\widetilde{\mathrm{GL}}^+(2,\mathbb{R})}
\newcommand{\cX}{\mathcal{X}}
\newcommand{\cY}{\mathcal{Y}}
\newcommand{\cC}{\mathcal{C}}
\newcommand{\cA}{\mathcal{A}}

\newcommand{\cU}{\mathcal{U}}
\newcommand{\cH}{\mathcal{H}}

\newcommand{\cS}{\mathcal{S}}

\newcommand{\cI}{\mathcal{I}}
\newcommand{\cT}{\mathcal{T}}
\newcommand{\cQ}{\mathcal{Q}}
\newcommand{\Ku}{\mathcal{K}u}

\newcommand{\cD}{\mathcal{D}}

\newcommand{\cM}{\mathcal{M}}
\newcommand{\cV}{\mathcal{V}}

\DeclareMathOperator{\oh}{\mathcal{O}}

\usetikzlibrary{decorations.pathmorphing}
\begin{document}

\title[Double EPW cubes from twisted cubics on Gushel--Mukai fourfolds]{Double EPW cubes from twisted cubics on Gushel--Mukai fourfolds}


\address{Department of Mathematics, Imperial College, London SW7 2AZ, United Kingdom}
\email{s.feyzbakhsh@imperial.ac.uk}

\address{Shanghai Center for Mathematical Sciences, Fudan University, Jiangwan Campus, 2005 Songhu Road, Shanghai, 200438, China
}
\email{hfguo@fudan.edu.cn}

\address{School of Mathematical Sciences, Zhejiang University, Hangzhou, Zhejiang Province 310058, P. R. China}
\email{jasonlzy0617@gmail.com}

%\address{Simons Laufer Mathematical Sciences Institute,17 Gauss Way, Berkeley, CA 94720}
%\address{Max Planck Institute for Mathematics, Vivatsgasse 7, 53111 Bonn, Germany}
%\address{Institut de Mathématiqes de Toulouse, UMR 5219, Université de Toulouse, Université Paul Sabatier, 118 route de
%Narbonne, 31062 Toulouse Cedex 9, France}
%\email{shizhuozhang@mpim-bonn.mpg.de,shizhuo.zhang@math.univ-toulouse.fr, shizhuozhang@msri.org}

\address{Institut de Mathématiqes de Toulouse, UMR 5219, Université de Toulouse, Université Paul Sabatier, 118 route de
Narbonne, 31062 Toulouse Cedex 9, France}
\address{Center for Geometry and Physics, Institute for Basic Science. 
79, Jigok-ro 127beon-gil, Nam-gu, Pohang-si,
Gyeongsangbuk-do,
Republic of Korea 37673}
%\email{shizhuozhang@mpim-bonn.mpg.de,shizhuo.zhang@math.univ-toulouse.fr}
\email{shizhuozhang@msri.org, zszmath@ibs.re.kr}

\author{Soheyla Feyzbakhsh, Hanfei Guo, Zhiyu Liu, Shizhuo Zhang}
\address{}
\email{}

\begin{abstract}
In this paper, we conduct the first systematic investigation of twisted cubics on Gushel--Mukai (GM) fourfolds. We then study the double EPW cube, a 6-dimensional hyperk\"ahler manifold associated with a general GM fourfold $X$, through the Bridgeland moduli space, and show that it is the maximal rationally connected (MRC) quotient of the Hilbert scheme of twisted cubics on $X$. We also prove that a general double EPW cube admits a covering by Lagrangian subvarieties constructed from the Hilbert schemes of twisted cubics on GM threefolds, which provides a new example for a conjecture of O'Grady.
\end{abstract}

\maketitle

{
\hypersetup{linkcolor=blue}
\setcounter{tocdepth}{1}
\tableofcontents
}

\section{Introduction}

%{\color{black}
Hyperk\"ahler manifolds play a central role in algebraic geometry and complex geometry. However, only a few examples are known, and it is challenging to construct explicit examples of projective hyperk\"ahler manifolds. 

One of the most studied examples comes from K3 surfaces: the Hilbert schemes of points on K3 surfaces \cite{beauville:hilb-on-k3}, or more generally, the moduli spaces of stable sheaves on them \cite{mukai:moduli-K3-I}. More recently, instead of considering the moduli space of stable sheaves on K3 surfaces, people consider Hilbert schemes of rational curves on Fano fourfolds with a K3-type Hodge structure, such as cubic fourfolds and Gushel--Mukai (GM) fourfolds, which are smooth cubic hypersurfaces in $\PP^5$ and smooth quadric sections of linear sections of the projective cone over $\Gr(2,5)$, respectively (see Section \ref{sec-very-general-GM}).

Several families of projective hyperk\"ahler manifolds have been constructed from the Hilbert schemes of low-degree rational curves on cubic fourfolds or GM fourfolds. These examples include
\begin{enumerate}
    \item Fano varieties of lines on cubic fourfolds \cite{beauville:fano-variety-cubic-4fold}, which are hyperk\"ahler fourfolds,
    \item LLSvS eightfolds, which are hyperk\"ahler eightfolds constructed in \cite{LLSvS17} as the maximal rationally connected (MRC) quotient of the Hilbert schemes of twisted cubics on general cubic fourfolds,
    \item LSV tenfolds, which are hyperk\"ahler compactifications of the twisted Jacobian fibrations of cubic fourfolds (cf.~\cite{laza2017hyper,Voisin-twisted}) and the MRC quotient of the Hilbert schemes of rational quartics on cubic fourfolds by \cite{li2020elliptic}, and
    \item double (dual) EPW sextics, which are hyperk\"ahler fourfolds constructed in \cite{o2006irreducible} and can be realized as the MRC quotient of the Hilbert schemes of conics on GM fourfolds, as shown in \cite{iliev2011fano}.
\end{enumerate}

More recently, a hyperk\"ahler sixfold $\widetilde{C}_X$ associated with a general GM fourfold $X$ is constructed in \cite{IKKR19}, referred to as \emph{the double EPW cube associated with $X$}. As in the examples listed above, it is conjectured by Atanas Iliev and Laurent Manivel that double EPW cubes should be the MRC quotient of the Hilbert schemes of twisted cubics on GM fourfolds. Our first main theorem confirms this expectation, thereby enriching the above picture.

\begin{theorem}[{Theorem \ref{cor-cube-as-MRC}}]\label{main_theorem_1}
Let $X$ be a general GM fourfold. Then the double EPW cube $\wt{C}_X$ associated with $X$ is the MRC quotient of the Hilbert scheme $\Hilb_X^{3t+1}$ of twisted cubics on $X$.
\end{theorem}

\subsection*{Twisted cubics and K3 categories}

Our approach to Theorem \ref{main_theorem_1} is based on studying double EPW cubes and twisted cubics via the Bridgeland moduli spaces of certain K3 categories.

More precisely, for a GM fourfold $X$, there is a semi-orthogonal decomposition
\[\D^b(X)=\langle\Ku(X),\oh_X,\cU^{\vee}_X,\oh_X(H),\cU_X^{\vee}(H)\rangle,\]
where $\cU_X$ and $\oh_X(H)$ are the pull-back of the tautological subbundle and the Pl\"ucker line bundle on~$\Gr(2, 5)$, respectively. The residue category $\Ku(X)$ is called the \emph{Kuznetsov component} of $X$, which is a K3 category, i.e.~its Serre functor is $[2]$. The numerical Grothendieck group of $\Ku(X)$ contains a canonical rank two sublattice, generated by the classes $\Lambda_1$ and $\Lambda_2$ (cf.~Section \ref{sec-very-general-GM}). %Moreover, when $X$ is very general, there is a unique stability condition on $\Ku(X)$ up to a natural group action (cf.~Proposition \ref{prop-gm4-unique}).

There is a set of stability conditions $\Stab^{\circ}(\Ku(X))$ on $\Ku(X)$ constructed in \cite[Theorem 4.12]{perry2019stability}. We refer to Section \ref{sec-stab-ku} for more details. For a pair of integers $a,b$, we denote by $M^X_{\sigma_X}(a,b)$ the moduli space that parameterizes S-equivalence classes of $\sigma_X$-semistable objects of class $a\Lambda_1 +b\Lambda_2$ in $\Ku(X)$. Then, by \cite{perry2019stability}, if $a$ and $b$ are coprime and $\sigma_X\in \Stab^{\circ}(\Ku(X))$ is generic, the space $M^X_{\sigma_X}(a,b)$ is a projective hyperk\"ahler manifold.

In \cite[Theorem 1.1]{kapustka2022epw}, it is shown that $\wt{C}_X$ and $M^X_{\sigma_X}(1,-1)$ share the same period point, hence they are birational. This enables us to apply categorical methods to establish the relationship between $\widetilde{C}_X$ and $\Hilb_X^{3t+1}$, as in \cite{LLMS18,li2018twisted,li2020elliptic,GLZ2021conics}. Specifically, a careful study of the projection objects of twisted cubics under the projection functor
\[\pr_X\colon \D^b(X)\to \Ku(X)\]
defined in \eqref{eq-prX} leads to the following result.

%Therefore, to prove Theorem \ref{main_theorem_1}, it suffices to study the relation between the Hilbert scheme $\Hilb^{3t+1}_X$ and the Bridgeland moduli space $M^X_{\sigma_X}(1,-1)$. After a careful study of the projection objects of twisted cubics under the projection functor
%\[\pr_X\colon \D^b(X)\to \Ku(X)\]
%defined in \eqref{eq-prX}, we prove that $\Hilb^{3t+1}_X$ is almost a $\PP^1$-fibration over $M^X_{\sigma_X}(1,-1)$.

\begin{theorem}[{Theorem \ref{thm-pr-induce-map}}]\label{thm-intro-moduli}
Let $X$ be a general GM fourfold and $\sigma_X\in \Stab^{\circ}(\Ku(X))$ be a generic stability condition. Then the projection functor $\pr_X$ induces a dominant rational map
\begin{align*}
   pr \colon \Hilb_X^{3t+1} &\dashrightarrow M_{\sigma_X}^X(1,-1)\\
   [C] & \mapsto [\pr_X(I_C(H))]
\end{align*}
with $\PP^1$ as general fibers.
\end{theorem}

Combining this with \cite[Theorem 1.1]{kapustka2022epw}, we obtain Theorem \ref{main_theorem_1}.

To establish the aforementioned result, we conduct a thorough examination of the geometry of twisted cubics in Section \ref{sec-twisted-cubic-gm4} and Section \ref{sec-projection-twisted_cubics}. In particular, we explore how the geometry of a twisted cubic $C\subset X$ interacts with the categorical properties of $\pr_X(I_C(H))$. As far as we know, this is the first paper to systematically investigate twisted cubics on GM fourfolds. We believe that our study may be of independent interest and could aid in constructing a map $\Hilb_X^{3t+1}\dashrightarrow \widetilde{C}_X$ using classical techniques similar to those in \cite{iliev2011fano}.

\subsection*{Lagrangian covering families} 

A significant advantage of relating hyperk\"ahler manifolds to the Hilbert schemes of curves on Fano varieties is that hyperk\"ahler manifolds can be explicitly studied through the geometry of curves. In this paper, we primarily focus on the problem of the existence of Lagrangian covering families.

Roughly speaking, a Lagrangian family is a flat projective family of Lagrangian subvarieties that covers the hyperk\"ahler manifold, which is a generalization of a Lagrangian fibration (cf.~Definition \ref{def-lag-family}). The existence of a Lagrangian covering family on a hyperk\"ahler manifold implies several well-behaved cohomological properties and is closely related to the Lefschetz standard conjecture, as demonstrated in \cite{voisin:remark-coisotropic,voisin2021lefschetz,Bai22}. It is conjectured by O'Grady that every projective hyperk\"ahler manifold admits a Lagrangian covering family.

In our case, by \cite{FGLZ24} and the description of $\wt{C}_X$ via Bridgeland moduli spaces in \cite{kapustka2022epw}, we know that $\wt{C}_X$ admits a family of Lagrangian subvarieties. Using Theorem \ref{thm-intro-moduli}, we prove that this is a Lagrangian covering family, which provides a new example of O'Grady's conjecture.

\begin{theorem}[{Theorem \ref{thm-second-lag-cover-family}}]\label{thm-epw-covering}
Let $X$ be a general GM fourfold. Then the associated double EPW cube~$\wt{C}_X$ admits a Lagrangian covering family.
\end{theorem}

%In particular, this provides a new example of O'Grady's conjecture.

Theorem \ref{thm-epw-covering} is a consequence of the following more explicit construction. Let $\cY\subset |\oh_X(H)|\times X$ be the universal hyperplane section. Denote by $\cY_V$ the restriction of $\cY$ to an open subscheme $V\subset |\oh_X(H)|$. Then, we have:

\begin{theorem}[{Theorem \ref{thm-covering-moduli}}]\label{main_theorem_2}
Let $X$ be a general GM fourfold and $\sigma_X\in \Stab^{\circ}(\Ku(X))$ be a generic stability condition. Then there exists an open subscheme $V\subset |\oh_X(H)|$ and a dominant morphism $q\colon \Hilb^{3t+1}_{\cY_V/V}\to M_{\sigma_X}^X(1,-1)$ such that 
% https://q.uiver.app/#q=WzAsMyxbMCwwLCJcXEhpbGJeezN0KzF9X3tcXGNIX1YvVn0iXSxbMCwxLCJWXFxzdWJzZXQgXFxQUF44Il0sWzEsMCwiXFx3dHtDfV9YIl0sWzAsMSwicCIsMl0sWzAsMiwicSJdXQ==
\[\begin{tikzcd}
	{\Hilb^{3t+1}_{\cY_V/V}} & {M_{\sigma_X}^X(1,-1)} \\
	{V}
	\arrow["p"', from=1-1, to=2-1]
	\arrow["q", from=1-1, to=1-2]
\end{tikzcd}\]
is a Lagrangian covering family with $p$ smooth and projective and $q|_{p^{-1}(s)}$ is birational onto its image for each $s\in V$. Here $p\colon \Hilb^{3t+1}_{\cY_V/V}\to V$ is the structure map of the relative Hilbert scheme over $V$.
\end{theorem}

In Appendix \ref{appendix-A}, based on \cite{FGLZ24}, we apply the same approach to other classical examples to uniformly reconstruct Lagrangian covering families of Fano varieties of lines on cubic fourfolds, LLSvS eightfolds, and double (dual) EPW sextics.

%The above provides a new example of O'Grady's conjecture asserting that every projective hyperk\"ahler manifold admits a Lagrangian covering family. 

\subsection*{Related work}

Degree one and degree two subvarieties in GM varieties have been systematically studied in \cite{iliev2011fano,debarre2019gushel,Debarre2024quadrics,debarre2012period,Log12}, etc. More recently, \cite{JLLZ2021gushelmukai,GLZ2021conics} use Kuznetsov components to examine conics on GM threefolds and fourfolds.

Numerous studies have been devoted to investigating the moduli spaces of rational curves of low degree on cubic fourfolds and GM fourfolds, demonstrating that their MRC quotients are hyperk\"ahler varieties, as shown in \cite{beauville:fano-variety-cubic-4fold, iliev2011fano, LLMS18, LLSvS17, bayer2017stability, li2018twisted, li2020elliptic, GLZ2021conics}.

The idea of constructing Lagrangian subvarieties of hyperk\"ahler manifolds arising from cubic or GM fourfolds using their hyperplane sections has been realized in several classical results. For instance, the Fano variety of lines of a cubic fourfold by \cite{voisin1992stabilite} and the LLSvS eightfold associated with a general cubic fourfold due to \cite[Proposition 6.9]{shinder2017geometry}. On the other hand, for a GM fourfold $X$, according to \cite[Proposition 5.2, 5.6]{iliev2011fano}, one can construct concrete Lagrangians of the double EPW sextic either from hyperplane sections of $X$ or from fivefolds that contain $X$ as a hyperplane section. In \cite{FGLZ24}, this idea is unified and widely generalized to the setting of Bridgeland moduli spaces of stable objects in Kuznetsov components, which is then used in \cite{guo-liu:atomic} to construct new examples of atomic Lagrangians in hyperk\"ahler manifolds.

%In the recent paper \cite{ppzEnriques2023}, the authors study the fixed locus of the anti-symplectic involution of hyperk\"ahler varieties constructed from the Kuznetsov components of GM fourfolds and show that they are nonempty, thus Lagrangian. The difference is that their method relies on the relation between the Kuznetsov component of a GM threefold and a special GM fourfold branched along it, so the Lagrangian subvariety can not vary in the family of hyperplane sections.

As for the hyperk\"ahler manifolds admitting Lagrangian covering families, 
the work of \cite{voisin:remark-coisotropic,voisin2021lefschetz,Bai22} exploits their geometry and establishes the Lefschetz standard conjecture of degree 
two, along with other well-behaved cohomological properties.

%The notion of Lagrangian constant cycle subvarieties was introduced in \cite{Huy14} for the case of K3 surfaces. As a specific case of \cite[Conjecture 0.4]{voisin:remark-coisotropic}, it is predicted that there exists a Lagrangian constant cycle subvariety in any projective hyperk\"ahler manifold. In \cite{lin:constant-cycle-lag-fibration}, the author constructs Lagrangian constant cycle subvarieties for hyperk\"ahler varieties admitting a Lagrangian fibration. The existence of such subvarieties has also been verified for some cases, such as \cite{Huy14, voisin:remark-coisotropic, Beckmann20, zhang2023one, li2023bloch}.

\subsection*{Plan of the paper} 
In Section~\ref{sec-preliminay}, we recall the semi-orthogonal decomposition for general triangulated categories and stability conditions on Kuznetsov components of GM varieties with related properties. In Section~\ref{sec-twisted-cubic-gm4}, we conduct a thorough study of the geometry of twisted cubics on GM fourfolds. In Section~\ref{sec-projection-twisted_cubics}, we compute the projections of ideal sheaves of twisted cubics into the Kuznetsov component of a general GM fourfold and prove their stability. In Section~\ref{sec-epw-cube}, we first study the geometric properties of the Hilbert scheme of twisted cubics on a general GM fourfold and prove Theorem \ref{cor-cube-as-MRC}. Then we construct a Lagrangian covering family for a general double EPW cube via Hilbert schemes of twisted cubics on GM threefolds, which proves Theorem~\ref{thm-second-lag-cover-family}.  In Appendix \ref{appendix-A}, we use our categorical method to uniformly reconstruct Lagrangian covering families of Fano varieties of lines on cubic fourfolds, LLSvS eightfolds, and double (dual) EPW sextics. In Appendix \ref{appendix-B}, we relate the Hilbert scheme of twisted cubics on a GM threefold to its Bridgeland moduli space and prove the smoothness and irreducibility of the Hilbert scheme.

\subsection*{Notation and conventions} \leavevmode
\begin{itemize}
    %\item Let $\sigma$ be a stability condition. Then the central charge and heart are denoted by $Z_{\sigma}$ and $\cA_{\sigma}$,  respectively.
 
    %\item We denote the phase and slope with respect to a stability condition $\sigma$ by $\phi_{\sigma}$ and $\mu_{\sigma}$,  respectively. The maximal/minimal phase of the Harder--Narasimhan factors of a given object will be denoted by $\phi_{\sigma}^+$ and $\phi_{\sigma}^-$,  respectively.
    
    \item All triangulated categories are assumed to be $\CC$-linear of finite type, i.e.~$\sum_{i\in \ZZ} \dim_{\CC} \Ext^i(E,F)$ is finite for any two objects $E,F$.

    \item We use $\hom$ and $\ext^{i}$ to represent the dimension of the vector spaces $\Hom$ and~$\Ext^{i}$. We denote $\RHom(-,-)=\bigoplus_{i\in \ZZ} \Hom(-,-[i])[-i]$ and $\chi(-,-)=\sum_{i\in \ZZ} (-1)^i \ext^i(-,-).$
    
    \item For a triangulated category $\cT$, its Grothendieck group and numerical Grothendieck group are denoted by $\mathrm{K}(\cT)$ and $\Knum(\cT):=\KK(\cT)/\ker(\chi)$, respectively.

    \item If $X$ is a GM fourfold or a cubic fourfold and $Z$ is a closed subscheme of $X$, we denote by $I_Z$ the ideal sheaf of $Z$ in $X$. If $Z$ is also contained in a closed subscheme $Z'$ of $X$, then we denote by $I_{Z/Z'}$ the ideal sheaf of $Z$ in $Z'$.

    \item We use $V_i$ to denote a complex vector space of dimension $i$.

    \item For any closed subscheme $Z\subset \PP^n$, we denote by $\langle Z\rangle$ the projective space spanned by $Z$.

    \item Let $X\to S$ be a morphism of schemes and $p(t)$ be a polynomial. The corresponding relative Hilbert scheme with respect to a fixed polarization on $X$ over $S$ is denoted by $\Hilb^{p(t)}_{X/S}$. We set $\Hilb^{p(t)}_{X}:=\Hilb^{p(t)}_{X/\Spec(\CC)}$.
\end{itemize}

\subsection*{Acknowledgements}

It is our pleasure to thank Arend Bayer, Sasha Kuznetsov, and Qizheng Yin for very useful discussions on the topics of this project. Special thanks to Sasha Kuznetsov for answering many questions on Gushel--Mukai varieties and his generous help for Section \ref{sec-twisted-cubic-gm4} and \ref{sec-projection-twisted_cubics}. We would like to thank Enrico Arbarello, Marcello Bernardara, Lie Fu, Yong Hu, Grzegorz Kapustka, Chunyi Li, Zhiyuan Li, Laurent Manivel, Kieran O'Grady, Alexander Perry, Laura Pertusi, Richard Thomas, Claire Voisin, Ruxuan Zhang, Yilong Zhang, and Xiaolei Zhao for helpful conversations. SF acknowledges the support of the Royal Society URF/ R1/231191. HG is supported by NKRD Program of China (No. 2020YFA0713200) and LNMS. ZL is partially supported by NSFC Grant 123B2002. SZ is supported by the ERC Consolidator Grant WallCrossAG, no. 819864, ANR project FanoHK, grant ANR-20-CE40-0023 and partially supported by GSSCU2021092. SZ is also supported by the NSF under grant No. DMS-1928930, while he is residence at the Simons Laufer Mathematical Sciences Institute in Berkeley, California. Part of the work was finished during the junior trimester program of SF, ZL, and SZ funded by the Deutsche Forschungsgemeinschaft (DFG, German Research Foundation) under Germany’s Excellence Strategy – EXC-2047/1 – 390685813. We would like to thank the Hausdorff
Research Institute for Mathematics for their hospitality.

\section{Preliminaries}\label{sec-preliminay}

In this section, we recall basic definitions and properties of Gushel--Mukai varieties and Bridgeland stability conditions on their Kuznetsov components. % and the construction of stability conditions on the Kuznetsov component of certain Fano fourfold. 
We mainly follow from \cite{bridgeland:stability} and \cite{bayer2017stability}.

\subsection{Semi-orthogonal decompositions} \label{sod}
%First, we recall some properties of semi-orthogonal decompositions and mutation functors. All triangulated categories in this paper are assumed to be $\CC$-linear.
Let $\cT$ be a triangulated category and $\cD\subset \cT$ a full triangulated subcategory. We define the \emph{right orthogonal complement} of $\cD$ in $\cT$ as the full triangulated subcategory
\[ \cD^\bot = \{ X \in \cT \mid \Hom(Y, X) =0  \text{ for all } Y \in \cD \}.  \]
The \emph{left orthogonal complement} is defined similarly, as 
\[ {}^\bot \cD = \{ X \in \cT \mid \Hom(X, Y) =0  \text{ for all } Y \in \cD \}.  \]
%\begin{definition}
%Let $\cT$ be a triangulated category. 
We say a full triangulated subcategory $\cD \sst \cT$ is \emph{admissible} if the inclusion functor $i \colon \cD \hookrightarrow \cT$ has left adjoint $i^*$ and right adjoint $i^!$.
%\end{definition}
Let $( \cD_1, \dots, \cD_m  )$ be a collection of admissible full subcategories of $\cT$. We say that $$\cT = \langle \cD_1, \dots, \cD_m \rangle$$ is a \emph{semi-orthogonal decomposition} of $\cT$ if $\cD_j \sst \cD_i^\bot $ for all $i > j$, and the subcategories $(\cD_1, \dots, \cD_m )$ generate $\cT$, i.e.~the category resulting from taking all shifts and cones in the categories $(\cD_1, \dots, \cD_m )$ is equivalent to $\cT$.

Let $i\colon \cD \hookrightarrow \cT$ be an admissible full subcategory. Then the \emph{left mutation functor} $\bL_{\cD}$ through $\cD$ is defined as the functor lying in the canonical functorial exact triangle 
\[  i i^! \ra \identity \ra \bL_{\cD}   \]
and the \emph{right mutation functor} $\bR_{\cD}$ through $\cD$ is defined similarly, by the triangle 
\[ \bR_{\cD} \ra \identity \ra i i^*  .  \]
Therefore, $\bL_{\cD}$ is exactly the left adjoint functor of $\cD^\bot\hookrightarrow \cT$. Similarly, $\bR_{\cD}$ is the right adjoint functor of ${}^\bot \cD\hookrightarrow \cT$. When $E \in \cT$ is an exceptional object and $F \in \cT$ is any object, the left mutation $\bL_E F:=\bL_{\langle E\rangle} F$ fits into the triangle 
\[ E \otimes \RHom(E, F) \ra F \ra \bL_E F , \]
and the right mutation $\bR_E F:=\bR_{\langle E\rangle} F$ fits into the triangle
\[ \bR_E F \ra F \ra E \otimes \RHom(F, E)^\vee  . \]

Given a semi-orthogonal decomposition $\cT = \langle \cD_1 , \cD_2 \rangle$. Then $$\cT \simeq \langle S_{\cT}(\cD_2), \cD_1 \rangle \simeq \langle \cD_2, S^{-1}_{\cT}(\cD_1) \rangle$$ are also semi-orthogonal decompositions of $\cT$, where $S_{\cT}$ is the Serre functor of $\cT$. Moreover, \cite[Lemma 2.6]{kuz:fractional-CY} shows that  
\[ S_{\cD_2} = \bR_{\cD_1} \circ S_{\cT}  \, \, \, \text{ and } \, \, \, S_{\cD_1}^{-1} = \bL_{\cD_2} \circ S_{\cT}^{-1} . \]

\subsection{Stability conditions}

Let $\cT$ be a triangulated category and $\KK(\cT)$ be its Grothendieck group. 
%\begin{definition}
%The \emph{heart of a bounded t-structure} on $\cT$ is an abelian subcategory $\cA \sst \cT$ such that the following conditions are satisfied
%\begin{enumerate}
%    \item for any $E, F \in \cA$ and $n <0$, we have $\Hom(E, F[n])=0$,
%    \item for any object $E \in \cT$, there exists a sequence of morphisms 
%    \[ 0=E_0 \xrightarrow{\phi_1} E_1 \xrightarrow{\phi_2} \cdots \xra{\phi_m} E_m=E \]
%    such that $\cone(\phi_i)$ is in the form $A_i[k_i]$, for some sequence $k_1 > k_2 > \cdots > k_m$ of integers and $A_i \in \cA$.
%\end{enumerate}
%\end{definition}
%\begin{definition}
%Let $\cA$ be an abelian category and $Z \colon \KK(\cA) \ra \mathbb{C}$ be a group homomorphism such that for any $E \in \cA$, we have $\Im Z(E) \geq 0$ and if $\Im Z(E) = 0$, $\Re Z(E) < 0$. Then we call $Z$ a \emph{stability function} on $\cA$.
%\end{definition}
Fix a surjective morphism to a finite rank lattice $v \colon \KK(\cT) \ra \Lambda$. 

\begin{definition}
A \emph{stability condition} on $\cT$ is a pair $\sigma = (\cA, Z)$, where $\cA$ is the heart of a bounded t-structure on $\cT$ and $Z \colon \Lambda \ra \CC$ is a group homomorphism such that 
\begin{enumerate}
    \item the composition $Z \circ v : \KK(\cA) \cong \KK(\cT) \ra \CC$ is a stability function on $\cA$, i.e.~for any $E \in \cA$, we have $\Im Z(v(E)) \geq 0$ and if $\Im Z(v(E)) = 0$, $\Re Z(v(E)) < 0$. From now on, we write $Z(E)$ rather than $Z(v(E))$.
\end{enumerate}
For any object $E \in \cA$, we define the slope function $\mu_{\sigma}(-)$ as
\[
\mu_\sigma(E) := \begin{cases}  - \frac{\Re Z(E)}{\Im Z(E)}, & \Im Z(E) > 0 \\
+ \infty , & \text{else}.
\end{cases}
\]
An object $0 \neq E \in \cA$ is called $\sigma$-(semi)stable if for any proper subobject $F \sst E$, we have $\mu_\sigma(F) (\leq) \mu_\sigma(E)$.
\begin{enumerate}[resume]
    \item Any object $E \in \cA$ has a Harder--Narasimhan filtration in terms of $\sigma$-semistability defined above.
    \item There exists a quadratic form $Q$ on $\Lambda \otimes \mathbb{R}$ such that $Q|_{\ker Z}$ is negative definite  and $Q(E) \geq 0$ for all $\sigma$-semistable objects $E \in \cA$. This is known as the \emph{support property}.
\end{enumerate}
%If the composition $Z \circ v$ is a stability function, then $\sigma$ is a \emph{stability condition} on $\cT$.
\end{definition}

The \emph{phase} of a $\sigma$-semistable object $E\in \cA$ is defined as
\[\phi_{\sigma}(E):=\frac{1}{\pi}\mathrm{arg}(Z(E))\in (0,1].\]
For $n\in \ZZ$, we set $\phi_{\sigma}(E[n]):=\phi_{\sigma}(E)+n$. %Given a non-zero object $E\in \cT$, we will denote by $\phi^+_{\sigma}(E)$ (resp.~$\phi^-_{\sigma}(E)$) the biggest (resp.~smallest) phase of a Harder--Narasimhan factor of $E$ with respect to $\sigma$.

%A \emph{slicing} $\mathcal{P}_{\sigma}$ of $\mathcal{T}$ with respect to the stability condition $\sigma$ consists of full additive subcategories $\mathcal{P}_{\sigma}(\phi) \subset \mathcal{T}$ for each $\phi \in \mathbb{R}$ such that the subcategory $\mathcal{P}_{\sigma}(\phi)$ contains the zero object and all $\sigma$-semistable objects whose phase is $\phi$.
%\begin{enumerate}
%\item for $\phi\in (0,1]$, the subcategory $\cP(\phi)$ is given by the zero object and all $\sigma$-semistable objects whose phase is $\phi$;
%\item for $\phi + n$ with $\phi\in (0,1]$ and $n\in \mathbb{Z}$, we set $\cP(\phi + n) := \cP(\phi)[n]$.
%\end{enumerate}
%\end{definition}
%For any interval $I\subset \mathbb{R}$, we denote by $\cP_{\sigma}(I)$ the category given by the extension closure of $\{\cP_{\sigma}(\phi)\}_{\phi\in I}$. We will use both notations $\sigma = (\cA,Z)$ and $\sigma = (\cP_{\sigma},Z)$ for a stability condition $\sigma$ with heart $\cA = \cP_{\sigma}((0,1])$. %where $\cP_{\sigma}$ is the slicing of $\sigma$.

%We finish this section with the following definition:

\begin{definition}\label{def-serre-invariant}
Let $\cT$ be a triangulated category and $\Phi$ be an auto-equivalence of $\cT$. We say a stability condition $\sigma$ on $\cT$ is \emph{$\Phi$-invariant} if
\[\Phi\cdot \sigma=\sigma\cdot\wt{g}\]
for an element $\wt{g}\in \GL$\footnote{The action of $\Phi$ and $\GL$ on $\sigma$ is defined in \cite[Lemma 8.2]{bridgeland:stability}.}. We say $\sigma$ is \emph{Serre-invariant} if it is $S_{\cT}$-invariant, where $S_{\cT}$ is the Serre functor of $\cT$.
\end{definition}

\subsection{Gushel--Mukai varieties}\label{sec-very-general-GM}

Next, we review the basic constructions and properties of Gushel--Mukai varieties and their Kuznetsov components.

Recall that a Gushel--Mukai (GM) variety $X$ of dimension $n$ is a smooth intersection $$X=\mathrm{Cone}(\Gr(2,5))\cap Q,$$ 
where $\mathrm{Cone}(\Gr(2,5))\subset \PP^{10}$ is the projective cone over the \text{Plücker} embedded Grassmannian $\Gr(2,5)\subset \PP^9$, and $Q\subset \PP^{n+4}$ is a quadric hypersurface. Then $n\leq 6$ and we have a natural morphism $\gamma_X\colon X\to \Gr(2,5)$. We say $X$ is \emph{ordinary} if $\gamma_X$ is a closed immersion, and \emph{special} if $\gamma_X$ is a double covering onto its image. 

\begin{definition}
Let $X$ be a GM fourfold. We say $X$ is \emph{Hodge-special} if 
\[\mathrm{H}^{2, 2}(X)\cap \mathrm{H}_{\mathrm{van}}^4(X, \QQ) \neq 0,\]
where $\mathrm{H}_{\mathrm{van}}^4(X, \QQ):=\mathrm{H}_{\mathrm{van}}^4(X, \ZZ)\otimes \QQ$ and $\mathrm{H}_{\mathrm{van}}^4(X, \ZZ)$ is defined as the orthogonal complement of $$\gamma_X^*\mathrm{H}^4(\Gr(2,5), \ZZ)\subset \mathrm{H}^4(X, \ZZ)$$ with respect to the intersection form.
\end{definition}

By \cite[Corollary 4.6]{debarre2015special}, $X$ is non-Hodge-special when $X$ is very general among all ordinary GM fourfolds or very general among all special GM fourfolds.

The semi-orthogonal decomposition of $\D^b(X)$ for a GM variety $X$ of dimension $n\geq 3$ is given by 
\[\D^b(X)=\langle\Ku(X),\oh_X,\cU^{\vee}_X,\cdots,\oh_X((n-3)H),\cU_X^{\vee}((n-3)H)\rangle,\]
where $\cU_X$ and $\oh_X(H)$ are the pull-back of the tautological subbundle and the Pl\"ucker line bundle on $\Gr(2, 5)$ via $\gamma_X$, respectively. We refer to $\Ku(X)$ as the Kuznetsov component of $X$. We define the projection functor \begin{equation}\label{eq-prX}
    \mathrm{pr}_X:=\bR_{\cU_X}\bR_{\oh_X(-H)}\bL_{\oh_X}\bL_{\cU^{\vee}_X}\colon \D^b(X)\to \Ku(X)
\end{equation}
when $n= 4$ and $\pr_X:=\bL_{\oh_X}\bL_{\cU^{\vee}_X}$ when $n=3$.

%By \cite[Proposition 2.6]{kuznetsov2018derived}, the Serre functor $S_{\Ku(X)}=[2]$ when $\dim X$ is even, and $S_{\Ku(X)}=T_3\circ [2]$ for an involution $T_3$ of $\Ku(X)$ when $\dim X$ is odd.  

When $n=3$, according to the proof of \cite[Proposition 3.9]{kuznetsov2009derived}, $\Knum(\Ku(X))$ is a rank two lattice generated by $\lambda_1$ and $\lambda_2$, where
\begin{equation}\label{lambda}
    \ch(\lambda_1)=-1+\frac{1}{5}H^2,\quad \ch(\lambda_2)=2-H+\frac{1}{12}H^3,
\end{equation}
with the Euler pairing
\begin{equation}\label{eq-matrix-odd}
\left[               
\begin{array}{cc}   
-1 & 0 \\  
0 & -1\\
\end{array}
\right].
\end{equation} 

When $n=4$, there is a rank two sublattice $-A_1^{\oplus 2}$ in $\Knum(\Ku(X))$ generated by $\Lambda_1$ and $\Lambda_2$, where
\begin{equation}\label{Lambda}
    \ch(\Lambda_1)=-2+(H^2-\gamma^*_X \sigma_2)-\frac{1}{20}H^4,
\quad \ch(\Lambda_2)=4-2H+\frac{1}{6}H^3,
\end{equation}
whose Euler pairing is 
\begin{equation}\label{eq-matrix-even}
\left[               
\begin{array}{cc}   
-2 & 0 \\  
0 & -2\\
\end{array}
\right],
\end{equation}
where $\gamma^*_X \sigma_2$ is the pull-back of the Schubert cycle $\sigma_2\in \mathrm{H}^4(\Gr(2, 5), \ZZ)$. When $X$ is non-Hodge-special, by \cite[Proposition 2.25]{kuznetsov2018derived}, we have $$\Knum(\Ku(X))=-A_1^{\oplus 2}=\ZZ\Lambda_1\oplus \ZZ\Lambda_2.$$ 

By \cite[Section 3]{GLZ2021conics}, for a GM fourfold $X$, we see
\[\ch(\cU_X)=2-H+(\gamma_X^*\sigma_2-\frac{1}{2}H^2)+\frac{1}{30}H^3-\frac{1}{120}H^4,\]
hence $\cU_X(H)\cong \cU^{\vee}_X$. Similarly, we have
\[\ch(\cQ_X)=3+H+(\gamma_X^*\sigma_{1,1}-\frac{1}{2}H^2)-\frac{1}{30}H^3+\frac{1}{120}H^4.\]

For our later computations, we need the following result.

\begin{proposition}[{\cite[Proposition 4.5]{FGLZ24}}] \label{prop-pushforward}
%Let $X$ be a cubic fourfold or a GM fourfold and $j \colon Y\hookrightarrow X$ be a smooth hyperplane section. 
Let $X$ be a GM fourfold and $j\colon Y\hookrightarrow X$ be a smooth hyperplane section. For any object $E\in \D^b(Y)$, we have
\[\pr_X(j_*E)\cong\pr_X(j_*\pr_Y(E)).\]
\end{proposition}

\subsection{Stability conditions on the Kuznetsov components}\label{sec-stab-ku}

%In this subsection, we recall some basic facts about stability conditions on Kuznetsov components of Fano threefolds and fourfolds.

In \cite{bayer2017stability}, the authors provide a way to construct stability conditions on a semi-orthogonal component from weak stability conditions on a larger category. For GM varieties, we have the following. 

\begin{theorem}[{\cite{bayer2017stability,perry2019stability}}]\label{blms-induce}
Let $X$ be %a cubic threefold/fourfold or 
a GM variety of dimension $n \geq 3$. Then there exists a family of stability conditions on $\Ku(X)$.
\end{theorem}

When $X$ is a GM fourfold, we denote by $\Stab^{\circ}(\Ku(X))$ the family of stability conditions on $\Ku(X)$ constructed in \cite[Theorem 4.12]{perry2019stability}.

When $Y$ is a GM threefold, it is proved in \cite{pertusiGM3fold} that stability conditions on $\Ku(Y)$ constructed in \cite{bayer2017stability} are \emph{Serre-invariant}. Furthermore, they all belong to the same $\GL$-orbit.

\begin{theorem}[{\cite{JLLZ2021gushelmukai,FeyzbakhshPertusi2021stab}}]\label{thm-unique-threefold}
Let $Y$ be a GM threefold, then all Serre-invariant stability conditions on $\Ku(Y)$ are in the same $\GL$-orbit.
\end{theorem}

For GM fourfolds, all stability conditions on $\Ku(X)$ are Serre-invariant as $S_{\Ku(X)}=[2]$ and we have the following analog:

\begin{proposition}[{\cite[Proposition 4.13]{FGLZ24}}]\label{prop-gm4-unique}
Let $X$ be a non-Hodge-special GM fourfold. Then all stability conditions on $\Ku(X)$ are in the same $\GL$-orbit.
\end{proposition}

%The final step in proving Theorem \ref{thm-very-general-GM4} is to explore further stability conditions on $\Ku(X)$ and $\Ku(Y)$. We begin by examining the stability of specific objects in these categories. % with minimal $\ext^1$-group.

%The following lemma for the stability of objects with small $\ext^1$ will be used in later proofs.

%\begin{lemma}[{\cite[Lemma 4.12]{FGLZ24}}]\label{lem-mukai-4fold}
%Let $X$ be a GM variety of dimension $n$. 
%    \begin{enumerate}
 %       \item If $n=3$, any object $E \in \Ku(X)$ with $\ext^1_X(E, E) \leq 3$ is stable with respect to every Serre-invariant stability condition on $\Ku(X)$.  
%        \item If $n=4$ and $X$ is non-Hodge-special, any object $E\in\Ku(X)$ with $\mathrm{ext}_X^1(E,E)<8$ is stable with respect to every stability condition on $\Ku(X)$.
%    \end{enumerate}
%\end{lemma}

In this paper, we denote by $M^X_{\sigma_X}(a,b)$ (resp.~$M^Y_{\sigma_Y}(a,b)$) the moduli space that parameterizes S-equivalence classes of $\sigma_X$-semistable (resp.~$\sigma_Y$-semistable) objects of class $a\Lambda_1 +b\Lambda_2$ (resp.~$a\lambda_1 +b\lambda_2$) in $\Ku(X)$ (resp.~$\Ku(Y)$). Then by \cite{perry2019stability}, for any pair of coprime integers $a,b$ and a generic $\sigma_X$, the space $M^X_{\sigma_X}(a,b)$ is a projective hyperk\"ahler manifold. The corresponding moduli stacks are denoted by $\cM^X_{\sigma_X}(a,b)$ and $\cM^Y_{\sigma_Y}(a,b)$, respectively.

\section{Twisted cubics on Gushel--Mukai fourfolds}\label{sec-twisted-cubic-gm4}

In this section, we explore the fundamental properties of twisted cubics on a GM fourfold.

We fix a $5$-dimensional vector space $V_5$ and an ordinary GM fourfold $X$ in $\Gr(2, V_5)$. Recall that an ordinary GM fourfold $X$ is given by
\[X=\Gr(2, V_5)\cap \PP(W)\cap Q,\]
where $\PP(W)\cong \PP^8\subset \PP(\wedge^2 V_5)\cong \PP^9$ is a hyperplane and $Q\subset \PP(\wedge^2 V_5)$ is a quadric hypersurface. We will always assume that $X$ is \emph{general}. In particular, $X$ does not contain a plane (cf.~\cite[Lemma 3.6]{iliev2011fano}). A one-dimensional closed subscheme~$C$ of $X$ is a \emph{twisted cubic} if the Hilbert polynomial of $\oh_C$ is $3t+1$ with respect to $\oh_X(H)$.

\subsection{Basic properties}

We first investigate some basic properties of twisted cubics in $X$.

\begin{lemma}\label{lem-pure-cubic}
We have

\begin{enumerate}
    \item $\Hilb^{3t+m}_X=\varnothing$ for $m\leq 0$, and

    \item $\oh_C$ is a pure sheaf for any $[C]\in \Hilb^{3t+1}_X$, or in other words, $C$ is Cohen--Macaulay.
\end{enumerate}
\end{lemma}

\begin{proof}
By \cite[Corollary 1.38]{sa14}, we have $\Hilb_X^{3t+m}=\varnothing$ for $m<0$. Thus to prove (1), we only need to show $\Hilb_X^{3t}=\varnothing$. From \cite[Corollary 1.38]{sa14}, $\langle C\rangle\cong \PP^2$ for any $[C]\in \Hilb_X^{3t}$. Since $X$ is an intersection of quadrics and does not contain any plane, such $C$ is not contained in $X$. Hence, $\Hilb_X^{3t}=\varnothing$ and (1) follows. Now (2) follows from (1), \cite[Lemma 4.3]{liu-ruan:cast-bound}, and \cite[Tag 0BXG]{stacks-project}.
\end{proof}

\begin{lemma}\label{lem-cubic-cohomology-1}
%Let $X$ be a general GM fourfold and 
Let $C\subset X$ be a twisted cubic. Then we have
\begin{enumerate}
    \item $\RHom_X(\oh_X, I_C)=H^1(\oh_C)=0$,

    \item $\RHom_X(\oh_X, I_C(H))=\CC^5$, and

    \item $\RHom_X(\oh_X, I_C(-H))=\RHom_X(I_C, \oh_X(-H))=\CC^2[-2]$.

    %\item $\RHom_X(\cU_X, I_C)=0$ or $\CC\oplus \CC[-1]$, and

    %\item $\RHom_X(\cQ^{\vee}_X, I_C)=\CC\oplus \CC[-1]$.
\end{enumerate}
\end{lemma}

\begin{proof}
By the standard exact sequence
\begin{equation}\label{eq-standard-C}
    0\to I_C\to \oh_X\to \oh_C\to 0,
\end{equation}
associated with $\oh_C$ and $\chi(I_C)=0$, to prove (1), we only need to show $H^1(\oh_C)=0$. This follows from Lemma \ref{lem-pure-cubic}(2) and \cite[Corollary 1.38(3)]{sa14}. 

For (2), note that $\hom_X(\oh_X, I_C(H))\leq 5$, otherwise $C$ is contained in $\PP^2$, which is impossible. Since we have~$\chi(\oh_C(H))=4$ and $\RHom_X(\oh_X, \oh_X(H))=\CC^9$, to prove (2), it is left to show $H^1(\oh_C(H))=0$, which follows from Lemma \ref{lem-pure-cubic}(2) and \cite[Lemma 2.10]{heinrich:twisted-cubic}. 

To prove (3), by Serre duality and \eqref{eq-standard-C}, it is sufficient to prove $H^0(\oh_C(-H))=0$. If this is not true, we have a non-zero morphism $\oh_X\to \oh_C(-H)$. After composing with $\oh_C(-H)\hookrightarrow \oh_C$, we obtain a non-zero non-surjective map~$\oh_X\to \oh_C$. However, we have  $H^0(\oh_C)=\CC$ from (1), which makes a contradiction.
\end{proof}

Using the geometry of the Grassmannian $\Gr(2, V_5)$, we can classify all twisted cubics on $X$ into the following three types, similar to the classification of conics in \cite{iliev2011fano}:

\begin{definition}
Let %$X$ be a general GM fourfold and
$C\subset X$ be a twisted cubic. 

\begin{itemize}
    \item We say $C$ is a \emph{$\tau$-cubic} if $\langle C \rangle $ is not in $\Gr(2, V_5)$ and $C$ is not contained in $\Gr(2, V_4)$ for any $4$-dimensional subspace $V_4\subset V_5$.

    \item We say $C$ is a \emph{$\rho$-cubic} if $\langle C \rangle $ is not in $\Gr(2, V_5)$ and $C$ is contained in $\Gr(2, V_4)$ for a $4$-dimensional subspace $V_4\subset V_5$.

    \item We say $C$ is a \emph{$\sigma$-cubic} if $\langle C \rangle $ is contained in $\Gr(2, V_5)$.
\end{itemize}
\end{definition}

Recall that a global section $s$ of $\cU^{\vee}_{\Gr(2,V_5)}$ is a vector in $V_5^{\vee}=H^0(\cU^{\vee}_{\Gr(2, V_5)})$. The zero locus $Z(s)$ parameterizes $2$-dimensional subspaces $U\subset V_5$ such that $U\subset \ker(s)$ when we regard $s$ as a linear function $s\colon V_5\to \CC$. Therefore, we see $$Z(s)=\Gr(2, V_4)=\PP(\wedge^2 V_4)\cap \Gr(2, V_5)$$ for a $4$-dimensional subspace $V_4=\ker(s)\subset V_5$. Similarly, a global section $s$ of $\cQ_{\Gr(2,V_5)}$ is a vector in $V_5=H^0(\cQ_{\Gr(2, V_5)})$. The zero locus $Z(s)$ parameterizes $2$-dimensional subspaces $U\subset V_5$ such that the image of $s$ under $V_5\twoheadrightarrow V_5/U$ is zero, i.e.~$s\in U$. Thus in this case, we have $$Z(s)=\PP(V_1\wedge V_5)=\Gr(1,V_5/V_1)\cong \PP^3\subset \Gr(2, V_5)$$ for the $1$-dimensional subspace $V_1$ spanned by $s$.

Relying on the descriptions above, we derive the following homological characterization of types of twisted cubics, which is similar to \cite[Lemma 5.2]{GLZ2021conics} for conics.

\begin{lemma}\label{lem-homo-cubic}
Let %$X$ be a general GM fourfold and 
$C\subset X$ be a twisted cubic.

\begin{enumerate}
    \item If $C$ is a $\tau$-cubic, then $\RHom_X(\cU_X, I_C)=0$ and $\RHom_X(\cQ_X^{\vee}, I_C)=\CC[-1]$.

    \item If $C$ is a $\rho$-cubic, then $\RHom_X(\cU_X, I_C)=\CC\oplus \CC[-1]$ and $\RHom_X(\cQ_X^{\vee}, I_C)=\CC[-1]$.

    \item If $C$ is a $\sigma$-cubic, then $\RHom_X(\cU_X, I_C)=0$ and $\RHom_X(\cQ_X^{\vee}, I_C)=\CC\oplus \CC^2[-1]$.
\end{enumerate}

\end{lemma}

\begin{proof}
As $\cU^{\vee}_X$ is generated by global sections and $H^1(\oh_C)=0$, we see $H^1(\cU^{\vee}_X|_C)=0$. Similarly, we have $H^1(\cQ_X|_C)=0$. Therefore, by applying $\Hom_X(\cU_X, -)$ and $\Hom_X(\cQ^{\vee}_X, -)$ to \eqref{eq-standard-C} and using
\[\RHom_X(\oh_X, \cQ_X)=\RHom_X(\oh_X, \cU^{\vee}_X)=\CC^5,\]
to determine $\RHom_X(\cU_X, I_C)$ and $\RHom_X(\cQ^{\vee}_X, I_C)$, we only need to compute $$\Hom_X(\cU_X, I_C) ~ \text{ and }~\Hom_X(\cQ^{\vee}_X, I_C).$$

To this end, note that $\hom_X(\cU_X, I_C)\leq 1$, otherwise $C\subset \Gr(2,3)\cong \PP^2$, which is impossible. Similarly, we have $\hom_X(\cQ^{\vee}_X, I_C)\leq 1$, otherwise $C\subset \Gr(0,3)\cong \Spec \CC$. Then by definition and the discussion above, it is clear that $\Hom_X(\cQ^{\vee}_X, I_C)=0$ when $C$ is a $\tau$-cubic or $\rho$-cubic, and $\Hom_X(\cQ^{\vee}_X, I_C)=\CC$ when $C$ is a $\sigma$-cubic. Similarly, we obtain $\Hom_X(\cU_X, I_C)=0$ when $C$ is a $\tau$-cubic and $\Hom_X(\cU_X, I_C)=\CC$ when~$C$ is a $\rho$-cubic. 

It remains to show $\Hom_X(\cU_X, I_C)=0$ when $C$ is a $\sigma$-cubic. By $\langle C\rangle \subset \Gr(2, V_5)$ and \cite[Section 4.1]{debarre2019gushel}, we know that $\langle C\rangle=\PP(V_1\wedge V_5)$ for a $1$-dimensional subspace $V_1\subset V_5$. If $\Hom_X(\cU_X, I_C)\neq 0$, then $C\subset \PP(\wedge^2 V_4)$ for a $4$-dimensional subspace $V_4\subset V_5$. When $V_1\nsubseteq V_4$, we have $V_1\cap V_4=0$ and hence $\PP(V_1\wedge V_5)\cap \PP(\wedge^2 V_4)=\varnothing$, which is impossible. When $V_1\subset V_4$, we see 
\[C\subset \PP(V_1\wedge V_5)\cap \PP(\wedge^2 V_4)\subset \PP(V_1\wedge V_4)\cong \PP(V_4/V_1)\cong \PP^2,\]
which also makes a contradiction. Thus the result follows.
\end{proof}

%We close the discussion on the basic properties by the following easy lemma, which will be used later.

%\begin{lemma}\label{lem-H1-bundle}
%Let $C\subset X$ be a twisted cubic. If $E$ is a globally generated vector bundle on $X$, then $H^1(E|_C)=0$.
%\end{lemma}

%\begin{proof}
%From the assumption, we have an exact sequence $0\to K\to \oh_X\otimes H^0(E)\to E\to 0$, where $K$ is a vector bundle. After restricting to $C$, we have an exact sequence $$0\to K|_C\to \oh_C \otimes H^0(E)\to E|_C\to 0.$$
%As $H^2(K|_C)=0$ and $H^1(\oh_C)=0$ from Lemma \ref{lem-cubic-cohomology-1}(1), we obtain $H^1(E|_C)=0$.
%\end{proof}

\subsection{Interlude: surfaces in GM fourfolds}

Before we continue to discuss twisted cubics, let us establish some useful properties of surfaces contained in $X$. Recall that a closed subscheme of $X$ is called a quadric surface if it has the Hilbert polynomial $(t+1)^2$ with respect to $\oh_X(H)$.

The following lemma classifies quadric surfaces in $\Gr(2, V_5)$, which is known to experts (cf.~\cite{Debarre2024quadrics}). 

\begin{lemma}\label{lem-type-of-quadric-surface}
Let $S\subset \Gr(2, V_5)$ be a quadric surface. Then either $S\subset \Gr(2, V_4)$ for a $4$-dimensional subspace $V_4\subset V_5$, or $S\subset \PP(V_1\wedge V_5)\cong \PP^3$ for a $1$-dimensional subspace $V_1\subset V_5$.
\end{lemma}

%For a general GM fourfold $X$, 
It is known that there is a unique quadric surface $q$ contained in $X$ (cf.~\cite[Section 3]{debarre2015special}). It is smooth since $X$ is general and $\langle q \rangle \cong \PP^3\subset \Gr(2, V_5)$. We call $q$ \emph{the $\sigma$-quadric of $X$}. According to \cite[Proposition 5.6]{GLZ2021conics}, there is an exact sequence
\begin{equation}\label{eq-def-q}
    0\to \cU_X\to \cQ^{\vee}_X\to I_q\to 0.
\end{equation}
Using this, it is easy to compute that
\begin{equation}\label{eq-class-q}
    [q]=\gamma_X^*\sigma_{2}-\gamma_X^*\sigma_{1,1}\in \mathrm{H}^4(X, \ZZ).
\end{equation}
By standard Schubert calculus, the sublattice $\gamma^*\mathrm{H}^4(\Gr(2, V_5), \ZZ)\subset \mathrm{H}^4(X, \ZZ)$ has the intersection matrix
\begin{equation}\label{eq-matrix-surface}
\left[               
\begin{array}{cc}   
2 & 2 \\  
2 & 4\\
\end{array}
\right]
\end{equation}
in basis $(\gamma_X^*\sigma_{1,1}, \gamma_X^*\sigma_{2})$. Hence we see $[q]^2=2$.

\begin{lemma}\label{lem-normal-bundle-q}
Let %$X$ be a general GM fourfold and 
$q\subset X$ be the $\sigma$-quadric.

\begin{enumerate}
    \item We have $H^*(\cQ^{\vee}_q)=0$, $H^*(\cQ_q)=\CC^4$, and an exact sequence
    \[0\to \cQ^{\vee}_q\to \oh_q^{\oplus 4}\to \oh_q(H)\to 0.\]

    \item We have an exact sequence
    \begin{equation}\label{eq-lem-3.7}
        0\to \oh_q(-H)\to \cQ^{\vee}_q\to N_{q/X}^{\vee} \to 0.
    \end{equation}

    \item We have $H^*(N_{q/X}^{\vee})=H^*(N_{q/X})=0$.
\end{enumerate}
\end{lemma}

\begin{proof}
%Since $X$ is general, the $\sigma$-quadric $q$ is smooth. Hence, $q\cong \PP^1\times \PP^1$, and $K_q\cong \oh_q(-2H)\cong \oh_q(-2,-2)$. 

By \cite[Lemma 3.7]{debarre2019gushel}, we have $\cU^{\vee}_q\cong \oh_q\oplus \oh_q(H)$. Then (1) follows from the tautological exact sequence
\[0\to \cQ^{\vee}_q\to \oh_q^{\oplus 5}\to  \cU^{\vee}_q\cong \oh_q\oplus \oh_q(H)\to 0.\]

For (2), note that $\cH^{-1}((I_q)|_q)\cong \oh_q$. Then after pulling back \eqref{eq-def-q} to $q$, we have an exact sequence $$0\to \oh_q\to \oh_q\oplus \oh_q(-H)\to \cQ^{\vee}_q\to N_{q/X}^{\vee}\to 0,$$
which gives \eqref{eq-lem-3.7}. 

Since $H^*(\cQ^{\vee}_q)=0$, \eqref{eq-lem-3.7} implies $H^*(N_{q/X}^{\vee})=0$. As $N_{q/X}$ is a bundle of rank $2$ and $c_1(N_{q/X})=0$, we see $N_{q/X}\cong N_{q/X}^{\vee}$. Thus, we have $H^*(N_{q/X})=0$, which proves (3).

%Since $H^*(\cQ^{\vee}_q)=0$, we have $\oh_q\subset \ker(s|_q)$. Since $N_{q/X}^{\vee}$ is rank $2$ vector bundle, we deduce that $\im(s|_q)$ is a line bundle. Therefore, the only possibility is $\ker(s|_q)=\oh_q$ and $\im(s|_q)=\oh_q(-H)$. In other words, we have an exact sequence
%\begin{equation}
%    0\to \oh_q(-H)\xra{t} \cQ^{\vee}_q\to N_{q/X}^{\vee} \to 0,
%\end{equation}
%which implies $H^*(N_{q/X}^{\vee})=0$. As $N_{q/X}$ is a bundle of rank $2$ and $c_1(N_{q/X})=0$, we see $N_{q/X}\cong N_{q/X}^{\vee}$. Thus, we have $H^*(N_{q/X})=0$.
\end{proof}

The following result gives useful geometric properties of surfaces with cohomology class $\gamma^*_X\sigma_{1,1}$ or $\gamma^*_X\sigma_{2}$.

\begin{lemma}\label{lem-surface-in-GM4}
Let %$X$ be a general GM fourfold and 
$S\subset X$ is a $2$-dimensional closed subscheme such that $\oh_S$ is a pure sheaf.

\begin{enumerate}
    \item If $S.H^2=2$, then $S$ is the unique smooth $\sigma$-quadric surface $q$.

    \item If $[S]=\gamma^*_X\sigma_{1,1}$, then $S$ is an integral surface.

    \item If $[S]=\gamma^*_X\sigma_2$, then $S$ is reduced. Furthermore, if $S$ is reducible, then it is the union of $q$ and an integral surface $S'$ of degree $4$.
\end{enumerate}

\end{lemma}

\begin{proof}
(1): We know that $X$ does not contain any plane. Thus, if $S.H^2=2$, then $S$ is irreducible. Moreover, if $S$ is non-reduced, then $S_{red}.H^2=2$ as well by the same reason.  This implies that the kernel of $\oh_{S}\twoheadrightarrow \oh_{S_{red}}$ is supported in dimension $\leq 1$, which contradicts the purity of $\oh_S$. Therefore, $S$ is integral and $h^0(\oh_S)=1$. 

Now, let $C$ be a general hyperplane section of $S$, which is integral by Bertini's theorem. As $$\chi(\oh_C)=1-h^1(\oh_C)\leq 1,$$
according to \cite[Corollary 1.38]{sa14}, we have $\chi(\oh_C)=1$, implying that $C$ is a conic. Using the exact sequence $$0\to \oh_S(-H)\to \oh_S\to \oh_C\to 0,$$ we see $h^0(\oh_S(H))\leq 4$, which means that $S$ is a quadric surface in $\PP^3$. As $X$ is general, we know that $S$ is the unique smooth $\sigma$-quadric. This establishes (1).

(2): It is clear that $S.H^2=4$. If $S$ is reducible, then as $X$ does not contain any plane, we know that each component of $S$ is of degree two, which is the $\sigma$-quadric by (1) and contradicts the uniqueness part of (1). Thus $S$ is irreducible. If $S$ is non-reduced, then $S_{red}.H^2=2$, which means $S_{red}$ is the unique $\sigma$-quadric $q$ in $X$. Therefore, we have $$[S]=\gamma^*_X\sigma_{1,1}=2[q]\in \mathrm{H}^4(X,\ZZ).$$ However, this is impossible according to \eqref{eq-matrix-surface} since $[q]^2=2$.

%\footnote{This is a well-known fact. If such cubic surface $Z$ exists, then the same argument in (1) shows that $Z\subset \PP^4$ is integral with general hyperplane sections being smooth twisted cubics. By \cite{min-deg}, $Z$ is either a cubic scroll or a cone over a smooth twisted cubic. This implies all such $X$ fall into a divisor in the moduli stack of GM fourfolds (cf.~\cite[Section 7.4]{debarre2015special}).}

(3): It is clear that $S.H^2=6$. Since $X$ is general, we know that $S$ does not contain any plane and surface of degree $3$\footnote{If such cubic surface $Z$ exists, without loss of generality we can assume that $\oh_Z$ is pure. Then the same argument in (1) shows that $Z\subset \PP^4$ is integral with general hyperplane sections being smooth twisted cubics. By \cite{min-deg}, $Z$ is either a cubic scroll or a cone over a smooth twisted cubic. In each case, $Z$ is a codimension $2$ linear section of $\PP(V_2\wedge V_5)\cap \Gr(2, V_5)$ for a $2$-dimensional $V_2\subset V_5$. This implies $[Z]=\sigma_1^2.\sigma_2$ in the Chow ring of $\Gr(2, V_5)$ by \cite[Lemma 2.2]{debarre:GM-jacobian}, and all $X$ containing such $Z$ fall into a codimension one locus in the moduli stack of GM fourfolds (cf.~\cite[Section 7.4]{debarre2015special}).}. If $S$ is non-reduced, then $S_{red}.H^2=2$ or $4$ as $\oh_S$ is pure-dimensional. If $S_{red}.H^2=4$, then the support of $\ker(\oh_{S}\twoheadrightarrow \oh_{S_{red}})$ is a degree $2$ surface in $S$, which is the $\sigma$-quadric $q$ by (1). This implies that $q\subset S_{red}$ is an irreducible component of $S_{red}$. However, this means $S_{red}$ is the union of two surfaces of degree $2$, which contradicts the uniqueness part of (1). Thus, the only possible case is $S_{red}.H^2=2$, and from (1), we deduce $S_{red}=q$. However, this implies $$[S]=\gamma_X^*\sigma_2=3[q]\in \mathrm{H}^4(X,\ZZ),$$
which is impossible by \eqref{eq-matrix-surface} and $[q]^2=2$. Hence, $S$ is reduced.

When $S$ is reducible, as $S.H^2=6$ and $X$ is general, we know that either $S$ has three irreducible components of degree $2$, or it has two components of degree $2$ and $4$, respectively. By the uniqueness part of (1), we conclude that the former case cannot occur. Thus $S=q\cup S'$, where $S'$ is an irreducible component of $S$ with $S'.H^2=4$. Since $S$ is reduced, $S'$ is also reduced and the last statement of (3) follows.
\end{proof}

\subsection{$\rho$-cubics}

%In this subsection, we focus on 
Now, we continue to study $\rho$-cubics in detail. Their geometry is closely related to surfaces of the class $\gamma^*_X\sigma_{1,1}\in \mathrm{H}^4(X, \ZZ)$.

\begin{lemma}\label{lem-rho-cubic}
Let %$X$ be a general GM fourfold and 
$C\subset X$ be a $\rho$-cubic.

\begin{enumerate}
    \item $C$ is contained in $\Gr(2, V_4)=\PP(\wedge^2 V_4)\cap \Gr(2, V_5)$ for a unique $V_4\subset V_5$.

    \item $S_{1,1}:=\Gr(2, V_4)\cap X$ is a degree $4$ integral Gorenstein surface with $[S_{1,1}]=\gamma_X^*\sigma_{1,1}$ and we have an exact sequence
    \begin{equation}\label{eq-def-S11}
        0\to \oh_X(-H)\to \cU_X\to I_{S_{1,1}}\to 0.
    \end{equation}

    \item There is a unique hyperplane section $D$ of $S_{1,1}$ containing $C$. Moreover, $D$ is a degree $4$ elliptic curve and the residue curve of $C$ in $D$ is a line $L$ with an exact sequence
    \[0\to \oh_L(-2)\to \oh_D\to \oh_C\to 0.\]
\end{enumerate}

\end{lemma}

\begin{proof}
The existence of $V_4$ in (1) follows directly from the definition of a $\rho$-cubic, and its uniqueness is derived from $\Hom_X(\cU_X, I_C)=\CC$ by Lemma \ref{lem-homo-cubic}(2). 

For (2), note that $$S_{1,1}=\Gr(2, V_4)\cap X=\Gr(2,V_4)\cap \PP(W)\cap Q\subset \PP^9,$$ we see $\dim S_{1,1}\geq 2$. Since $\Pic(X)=\ZZ \oh_X(H)$, $X$ does not contain any threefold of degree less than $10$. Consequently, $S_{1,1}$ can only have dimension $2$. Thus $S_{1,1}$ is a degree $4$ Gorenstein surface. Given that $S_{1,1}$ is the zero locus of a section of $\cU^{\vee}_X$, we have a surjection $s\colon \cU_X\twoheadrightarrow I_{S_{1,1}}$. By $c_1(I_{S_{1,1}})=0$, we obtain $c_1(\ker(s))=-H$. Since $\cU^{\vee}_X$ is locally free and $I_{S_{1,1}}$ is torsion-free, $\ker(s)$ is a rank one reflexive sheaf, which is a line bundle by the smoothness of $X$. Therefore, $\ker(s)=\oh_X(-H)$ and this proves \eqref{eq-def-S11}. Then a direct computation using \eqref{eq-def-S11} gives $[S_{1,1}]=\gamma_X^*\sigma_{1,1}$, which implies the integrality of $S_{1,1}$ by Lemma \ref{lem-surface-in-GM4}(2).

From \eqref{eq-def-S11}, we have $\RHom_X(\oh_X, I_{S_{1,1}}(H))=\CC^4$. Then by Lemma \ref{lem-cubic-cohomology-1}(2), we get 
\begin{equation}\label{eq-ICSH}
\RHom_X(\oh_X, I_{C/S_{1,1}}(H))=\CC,
\end{equation}
which implies that there is a unique hyperplane section $D$ of $S_{1,1}$ containing $C$. Since $S_{1,1}$ is integral, it is clear that $D$ is a Gorenstein curve of degree $4$ and genus $1$. Moreover, $H.\ch_3(\ker(\oh_D\twoheadrightarrow \oh_C))=1$, hence, $\ker(\oh_D\twoheadrightarrow \oh_C)$ is supported on the residue line $L$ of $C$ in $D$. A calculation of Euler characteristic then shows that $\ker(\oh_D\twoheadrightarrow \oh_C)\cong \oh_L(-2)$ and (3) follows.
\end{proof}

\subsection{$\tau$-cubics}

Now we examine the geometry of $\tau$-cubics and the surfaces containing them.

%Next, we investigate the zero loci of sections of $\wedge^2 \cQ_X$ containing $\tau$-cubics.

\begin{lemma}\label{lem-in-zero-loci}
%Let $X$ be a general GM fourfold. Then 
We have $\Hom_X(\cQ_X(-H), I_C)\neq 0$ for any twisted cubic $C\subset X$. In other words, $C$ is contained in the zero locus of a section of $\wedge^2\cQ_X\cong \cQ^{\vee}_X(H)$.
\end{lemma}

\begin{proof}
Note that $\cQ^{\vee}_X(H)\cong \wedge^2\cQ_X$. Then using the Koszul resolution and the
Borel--Weil--Bott theorem as in \cite[Lemma 5.4]{GLZ2021conics}, it is straightforward to compute that $$H^0(\cQ^{\vee}_X(H))=H^0(\cQ_{\Gr(2, V_5)}^{\vee}(H))=\wedge^2 V_5.$$ Moreover, by taking wedge product to $V_5\otimes \oh_X\twoheadrightarrow \cQ_X$, we have a surjection $(\wedge^2 V_5)\otimes \oh_X\twoheadrightarrow \cQ^{\vee}_X(H)$, which means $\cQ^{\vee}_X(H)$ is generated by global sections. Then combined with $H^1(\oh_C)=0$, we have $H^1(\cQ^{\vee}_X(H)|_C)=0$. From $\chi(\cQ^{\vee}_X(H)|_C)=9$, we see $h^0(\cQ^{\vee}_X(H)|_C)=9$. Then the result follows from $H^0(\cQ^{\vee}_X(H))=\wedge^2 V_5\cong \CC^{10}$ and applying $\Hom_X(\cQ_X(-H), -)$ to the standard exact sequence of $\oh_C$.
\end{proof}

\begin{lemma}\label{lem-tau-in-sextic}
Let %$X$ be a general GM fourfold and 
$C\subset X$ be a $\tau$-cubic. 

\begin{enumerate}
    \item There exists a unique $2$-dimensional subspace $V_2\subset V_5$ such that $C\subset \PP(V_2\wedge V_5)$.

    \item $S_{2,0}:=\PP(V_2\wedge V_5)\cap X$ is a Cohen--Macaulay reduced sextic surface, which is the zero locus of a section of~$\wedge^2 \cQ_X$.

    \item We have an exact sequence
\begin{equation}\label{eq-def-S20}
0\to V_2\otimes \oh_X(-H)\to \cQ_X(-H)\to I_{S_{2,0}}\to 0
\end{equation}
and $[S_{2,0}]=\gamma^*_X\sigma_{2}\in \mathrm{H}^4(X, \ZZ)$.
\end{enumerate}

\end{lemma}

\begin{proof}
First of all, we describe the zero locus of a section of $\wedge^2 \cQ_{\Gr(2, V_5)}$. Let $s\in \wedge^2 V_5=H^0(\wedge^2 \cQ_{\Gr(2,V_5)})$. Then $Z(s)$ parameterizes all 2-dimensional subspaces $U\subset V_5$ such that the image of $s$ under the quotient map $p_U\colon \wedge^2V_5\twoheadrightarrow \wedge^2(V_5/U)$ is zero.

If $s$ has rank $2$ as a skew-symmetric bilinear form, then up to normalization, i.e.~action of $\mathrm{GL}(V_5)$ on $\wedge^2 V_5$, we can assume that $s=s_1\wedge s_2$, where $s_1,s_2\in V_5$ are linearly independent vectors. Denote by $V_2$ the $2$-dimensional space spanned by $s_1$ and $s_2$. Note that $\dim_{\CC} \wedge^2 U=1$, then it is easy to see~$p_U(s)=0\in \wedge^2(V_5/U)$ if and only if $V_2\cap U\neq \{0\}$, which is also equivalent to $\wedge^2 U\subset V_2\wedge V_5$. In this case, we have $$Z(s)=\PP(V_2\wedge V_5)\cap \Gr(2, V_5),$$ which is the cone over $\PP(V_2)\times \PP(V_5/V_2)\cong \PP^1\times \PP^2$ with~$\PP(\wedge^2 V_2)$ as the cone point by \cite[Lemma 2.2]{debarre:GM-jacobian}. In this case, $Z(s)$ is normal and Cohen--Macaulay.

If $s$ has rank $4$ as a skew-symmetric bilinear form, then up to normalization, we can assume that $s=s_1\wedge s_2+s_3 \wedge s_4$ such that $s_1,s_2,s_3,s_4$ span a $4$-dimensional subspace $V_4\subset V_5$. Note that the restriction $s|_{V_4}$ of $s$ is a symplectic form on $V_4$. Then in this case, a basic computation of linear algebra shows that $p_U(s)=0\in \wedge^2(V_5/U)$ if and only if $U\subset V_4$ is a Lagrangian subspace with respect to $s|_{V_4}$. Therefore, $Z(s)$ is the Lagrangian Grassmannian $\mathrm{LGr}(2, V_4)$, which is a linear section of $\Gr(2, V_4)$ and is a quadric threefold.

Now by Lemma \ref{lem-in-zero-loci}, $C$ is contained in the zero locus of a section $s\in H^0(\wedge^2 \cQ_X)=\wedge^2 V_5$. From the definition of $\tau$-cubics, we see $C\nsubseteq \mathrm{LGr}(2,4)$, i.e.~$s$ has rank $2$ as a skew-symmetric bilinear form. Therefore, the zero locus $S_{2,0}$ of $s$ is
\[S_{2,0}=\PP(V_2\wedge V_5)\cap X=\PP(V_2\wedge V_5)\cap \Gr(2, V_5)\cap \PP(W)\cap Q.\]
Note that such $V_2$ is unique. Otherwise, $C$ is contained in $\PP(V_1\wedge V_5)\cong \Gr(1, 4)$, which is the zero locus of a section of $\cQ_{\Gr(2, V_5)}$ and is impossible since $C$ is a $\tau$-cubic. This proves (1).

Now we prove (2). By \cite[Lemma 2.2]{debarre:GM-jacobian}, we know that $\PP(V_2\wedge V_5)\cap \Gr(2, V_5)$ is a degree three fourfold, hence each component of $S_{2,0}$ has dimension at least $2$. It is evident that $\dim S_{2,0}\leq 3$, and as $X$ does not contain any threefold of degree $<10$, we conclude that $\dim S_{2,0}\neq 3$. Consequently, $S_{2,0}$ is pure of dimension $2$ and is the complete intersection of $\PP(V_2\wedge V_5)\cap \Gr(2, V_5)$ with $\PP(W)$ and $Q$. This implies that $S_{2,0}$ is a sextic Cohen--Macaulay surface. This proves (2) except the reducedness of $S_{2,0}$.

Next, we verify \eqref{eq-def-S20}. As $S_{2,0}=\PP(V_2\wedge V_5)\cap X$, we deduce that  $S_{2,0}$ is exactly the degeneracy locus of the natural map $\pi\colon V_2\otimes \oh_X\hookrightarrow \cQ_X$. Since $\mathrm{cok}(\pi)$ is a quotient of locally free sheaves, if $\mathrm{cok}(\pi)$ is not torsion-free, then the torsion part has pure codimension one. Then the existence of a surjection $\cQ_X\to \mathrm{cok}(\pi)/(\mathrm{cok}(\pi))_{tor}$ contradicts the slope stability of $\cQ_X$. Hence, $\mathrm{cok}(\pi)$ is a rank one torsion-free sheaf with $c_1(\mathrm{cok}(\pi))=H$, and we have $\mathrm{cok}(\pi)\cong I_{S_{2,0}}(H)$ and \eqref{eq-def-S20} follows. Finally, from \eqref{eq-def-S20}, we see $[S_{2,0}]=\gamma^*_X\sigma_{2}$, which implies the reducedness of $S_{2,0}$ by Lemma \ref{lem-surface-in-GM4}(3). This completes the proof.
\end{proof}

Using Lemma \ref{lem-tau-in-sextic}, we have the following two useful results.

\begin{lemma}\label{lem-ev-surjective-tau}
Let %$X$ be a general GM fourfold and 
$C\subset X$ be a $\tau$-cubic, then $\langle C \rangle\cap \Gr(2,V_5)=\langle C \rangle\cap X=C$.
\end{lemma}

\begin{proof}
By definition, we have $\langle C \rangle \nsubseteq \Gr(2,V_5)$. Moreover, $\dim \langle C \rangle\cap \Gr(2,V_5)\neq 2$, as otherwise the intersection $\langle C \rangle\cap \Gr(2,V_5)$ would be a quadric surface containing $C$ because $\Gr(2,V_5)$ is the intersection of quadrics, which contradicts Lemma \ref{lem-type-of-quadric-surface} since $C$ is a $\tau$-cubic. Hence, $\dim \langle C \rangle\cap \Gr(2, V_5)=1$. Now, considering $V_2\subset V_5$ as in Lemma \ref{lem-tau-in-sextic}, we have $\langle C \rangle \subset \PP(V_2\wedge V_5)$. Thus we obtain 
$$\langle C \rangle \cap \Gr(2, V_5)=\langle C \rangle \cap \PP(V_2\wedge V_5) \cap \Gr(2, V_5).$$
As $\PP(V_2\wedge V_5) \cap  \Gr(2, V_5)$ is a degree $3$ normal Cohen--Macaulay fourfold in $\PP(V_2\wedge V_5)$ by \cite[Lemma 2.2]{debarre:GM-jacobian} and since we have already established that $$\dim (\langle C\rangle \cap \PP(V_2\wedge V_5) \cap \Gr(2, V_5))=1,$$ it follows that $\langle C \rangle\cong \PP^3$ intersects $\PP(V_2\wedge V_5) \cap \Gr(2, V_5)$ properly in $\PP(V_2\wedge V_5)\cong \PP^6$. This implies that $\langle C \rangle \cap \Gr(2, V_5)$ is a degree $3$ Cohen--Macaulay curve containing $C$. Thus we have $\langle C \rangle \cap \Gr(2, V_5)=C$ as desired.
\end{proof}

\begin{lemma}\label{lem-conic-in-tau-cubic}
Let %$X$ be a general GM fourfold and 
$C\subset X$ be a $\tau$-cubic. If $Z\subset C$ is a closed subscheme such that $\dim \langle Z\rangle>1$, then $\langle Z\rangle$ is not contained in $\Gr(2, V_5)$.
\end{lemma}

\begin{proof}
By Lemma \ref{lem-ev-surjective-tau}, we know that $\langle Z \rangle \cap \Gr(2, V_5)\subset \langle C \rangle \cap \Gr(2, V_5)=C$. Since $\dim \langle Z\rangle>1$, it is evident that $\langle Z\rangle\nsubseteq C$, thus $\langle Z\rangle\nsubseteq \Gr(2, V_5)$.
\end{proof}

Next, we study some cohomological properties of $C$ and $S_{2,0}$.

\begin{lemma}\label{lem-cohomology-ICS}
Let %$X$ be a general GM fourfold and 
$C\subset X$ be a $\tau$-cubic, and let $S_{2,0}$ be the sextic surface in Lemma \ref{lem-tau-in-sextic} containing $C$. Then we have

\begin{enumerate}
    \item $\RHom_X(\oh_X, I_{C/S_{2,0}}(H))=\CC^2$, and

    \item $\RHom_X(I_{C/S_{2,0}}(2H), \cU_X)=\RHom_X(I_{C/S_{2,0}}(2H), \oh_X)=0$.
\end{enumerate}

\end{lemma}

\begin{proof}
By \eqref{eq-def-S20}, we have
\[\RHom_X(\oh_X, I_{S_{2,0}}(H))=V_5/V_2\cong\CC^3\]
and $\RHom_X(\oh_X, I_{S_{2,0}})=0.$ As $$\RHom_X(\cU_X, \cQ_X(-H))=\RHom_X(\cU^{\vee}_X, \cQ_X)=0$$ by \cite[Lemma 5.4(2)]{GLZ2021conics}, using \eqref{eq-def-S20} we also have $\RHom_X(\cU, I_{S_{2,0}})=0$.

Then (1) follows from Lemma \ref{lem-cubic-cohomology-1} and applying $\Hom_X(\oh_X, -)$ to
\[0\to I_{S_{2,0}}(H)\to I_C(H)\to I_{C/S_{2,0}}(H)\to 0.\]
Similarly, (2) follows from Serre duality, Lemma \ref{lem-cubic-cohomology-1} and Lemma \ref{lem-homo-cubic}(2).
\end{proof}

The proof of the following proposition is lengthy, but it is crucial for computing projection objects in the next section and understanding the geometry of $\tau$-cubics.

\begin{proposition}\label{prop-residue-tau-cubic}
Let %$X$ be a general GM fourfold and 
$C\subset X$ be a $\tau$-cubic, and let $S_{2,0}$ be the sextic surface in Lemma \ref{lem-tau-in-sextic} containing $C$. 

\begin{enumerate}
    \item Given any two different hyperplane sections $D_1$ and $D_2$ of $S_{2,0}$ containing $C$, we have $D_1\cap D_2=C$, or in other words, the evaluation map
    \[\oh_X^{\oplus 2}\to I_{C/S_{2,0}}(H)\]
    is surjective.

    \item For any hyperplane section $D$ of $S_{2,0}$ containing $C$, denote by $C'$ the residue curve of $C$ in $D$. Then we have an exact sequence
    \begin{equation}\label{eq-residue}
        0\to \oh_{S_{2,0}}\to I_{C/S_{2,0}}(H)\to \oh_{C'}\to 0
    \end{equation}
and $C'$ is a $\tau$-cubic.
\end{enumerate}

\end{proposition}

Before proving this proposition, we introduce the following definition:

\begin{definition}
Let $C\subset X$ be a $\tau$-cubic. Then we say a twisted cubic $C'\subset X$ is a \emph{residue cubic} of $C$ if $C'$ is a residue curve in the sense of Proposition \ref{prop-residue-tau-cubic}(2).
\end{definition}

\begin{remark}\label{rmk-P1}
By Proposition \ref{prop-residue-tau-cubic}, once we have proved that every $C'$ in Proposition \ref{prop-residue-tau-cubic}(2) is a twisted cubics, it is clear that every $\tau$-cubic $C$ has a $\PP^1$-family of residue cubics, parameterized by $\PP(H^0(I_{C/S_{2,0}}(H)))\cong \PP^1$.    
\end{remark}

\begin{proof}[{Proof of Proposition \ref{prop-residue-tau-cubic}}]
By Lemma \ref{lem-cohomology-ICS}(1), we have $H^0(I_{C/S_{2,0}}(H))=\CC^2$. Since the natural map $H^0(I_C(H))\to H^0(I_{C/S_{2,0}}(H))$ is surjective, to prove the surjectivity of $$H^0(I_{C/S_{2,0}}(H))\otimes \oh_X=\oh_X^{\oplus 2}\to I_{C/S_{2,0}}(H),$$ it suffices to show that the natural evaluation map $H^0(I_C(H))\otimes \oh_X\to I_C(H)$ is surjective, which follows from Lemma \ref{lem-ev-surjective-tau}.

Now we prove (2). We divide the proof into several steps.

\medskip

\textbf{Step I}. \emph{$C'$ is a twisted cubic.}

By (1), for a hyperplane section $D'$ of $S_{2,0}$ different from $D$, we have $D\cap D'=C$. As $C'$ is the residue curve of $C=D\cap D'$ in $D$, from \cite[Equation (1)]{hartshorne:double-plane} we have an exact sequence
\[0\to \oh_{C'}(-H)\to \oh_D\to \oh_{C}\to 0.\]
Thus $I_{C/D}=\oh_{C'}(-H)$ and the exact sequence
\[0\to I_{D/S_{2,0}}\to I_{C/S_{2,0}}\to I_{C/D}\to 0\]
becomes
\[0\to \oh_{S_{2,0}}(-H)\to I_{C/S_{2,0}}\to \oh_{C'}(-H)\to 0.\]
Hence, we get \eqref{eq-residue} by tensoring with $\oh_X(H)$. Now a computation of $\ch_3$ and Euler characteristic using \eqref{eq-residue} shows that $C'$ has degree three and $\chi(\oh_{C'})=1$, i.e.~$C'$ is a twisted cubic.

\medskip

\textbf{Step II}. \emph{$\langle C\rangle\cap \langle C' \rangle$ is a plane.}

%It remains to show that $C'$ can not be a $\sigma$-cubic or $\rho$-cubic. 

According to Lemma \ref{lem-tau-in-sextic}, it follows that $S_{2,0}=\PP(V_2\wedge V_5)\cap X$ for a $2$-dimensional subspace $V_2\subset V_5$. Since $\langle S_{2,0} \rangle\cong \PP^5$, we know that $\langle D \rangle \cong \PP^4$. Moreover, as $C,C'\subset D$, we have $\langle C\rangle, \langle C' \rangle \subset \PP^4.$ Therefore, either $\langle C\rangle= \langle C' \rangle$ or $\langle C\rangle\cap \langle C' \rangle$ is a plane. As $\langle C\rangle\cap X=C$ by Lemma \ref{lem-ev-surjective-tau}, if $\langle C\rangle= \langle C' \rangle$, then we have $C'\subset \langle C' \rangle\cap X=C$, which implies $C=C'$ and the result follows. Thus in the following, we may assume that $\langle C\rangle\cap \langle C' \rangle$ is a plane. 

\medskip

\textbf{Step III}. \emph{$C'$ can only be of type $\tau$ or $\rho$.}

If $C'$ is a $\sigma$-cubic, then we have $\langle C' \rangle\subset \Gr(2, V_5)$. Moreover, $\langle C' \rangle\subset \PP(V_2\wedge V_5)=\langle S_{2,0} \rangle$. Thus
\[\langle C' \rangle\subset \PP(V_2\wedge V_5)\cap \Gr(2, V_5).\]
However, as we have already proved in Lemma \ref{lem-ev-surjective-tau} that
\[\langle C \rangle\cap  \PP(V_2\wedge V_5)\cap \Gr(2, V_5)=C,\]
hence by Step II, we obtain 
\[\PP^2\cong \langle C \rangle\cap \langle C' \rangle\subset \langle C \rangle\cap \PP(V_2\wedge V_5)\cap \Gr(2, V_5)=C,\]
which is a contradiction. Therefore, a residue cubic of $C$ can only be of type $\tau$ or $\rho$.

\medskip

\textbf{Step IV}. \emph{If there exists a residue $\tau$-cubic $C''$ of $C$, then all residue cubics of $C$ are of type $\tau$.} 

Indeed, from \eqref{eq-residue}, we have an exact sequence
\[0\to I_{S_{2,0}}\to \bL_{\oh_X}(I_{C/S_{2,0}}(H))[-1]\to I_{C''}\to 0.\]
As $C''$ is a $\tau$-cubic, applying $\Hom_X(\cU_X, -)$ to the exact sequence above and using Lemma \ref{lem-homo-cubic}(1), we have
\[\RHom_X(\cU_X, \bL_{\oh_X}(I_{C/S_{2,0}}(H)))=0.\]
On the other hand, if $C'$ is another residue cubic of $C$, by \eqref{eq-residue} we also have an exact sequence 
\[0\to I_{S_{2,0}}\to \bL_{\oh_X}(I_{C/S_{2,0}}(H))[-1]\to I_{C'}\to 0.\]
As $\RHom_X(\cU_X, \bL_{\oh_X}(I_{C/S_{2,0}}(H)))=\RHom_X(\cU_X, I_{S_{2,0}})=0$, we get $\RHom_X(\cU_X, I_{C'})=0$. Then by Lemma \ref{lem-homo-cubic}(2), $C'$ cannot be a $\rho$-cubic. As $C'$ is not a $\sigma$-cubic by Step III, we deduce that~$C'$ is a $\tau$-cubic and our claim follows.

\medskip

\textbf{Step V}. \emph{If $C'$ is a residue $\rho$-cubic of $C$, then $\langle Z\rangle =\langle C\rangle \cap \langle C'\rangle$ for a conic $Z=C\cap \langle C' \rangle$.}

%By Step IV, to prove (2), it suffices to prove that there exists a residue $\tau$-cubic of $C$. 

Let $C'$ be a residue $\rho$-cubic of $C$. Then by definition, $\langle C'\rangle \subset \PP(\wedge^2 V_4)$ for a $4$-dimensional subspace $V_4\subset V_5$. If $V_2\nsubseteq V_4$, then $V_2\wedge V_5 \cap \wedge^2 V_4=V_1\wedge V_4$, where $V_1=V_2\cap V_4$. Since $$C'\subset \langle S_{2,0} \rangle= \PP(V_2\wedge V_5)\cap \PP(W)\cong \PP^5,$$ we obtain $C'\subset \PP(V_1\wedge V_4)\cong \PP^2$, which is impossible. Thus $V_2\subset V_4$, and in this case we have
\[\PP^3\cong\langle C'\rangle\subset \PP(V_2\wedge V_5)\cap\PP(\wedge^2 V_4)=\PP(V_2\wedge V_4)\cong \PP^4.\]
As we also have $$\langle C'\rangle \subset \langle D\rangle\cong \PP^4\subset \PP(V_2\wedge V_5)\cong \PP^6,$$ we see either $\langle D \rangle =\PP(V_2\wedge V_4)$ or $\langle D \rangle \cap \PP(V_2\wedge V_4)=\langle C'\rangle$. But the former case is not possible since $C\subset D$ is a $\tau$-cubic. Thus when $C'$ is a $\rho$-cubic such that $C'\subset \PP(\wedge^2 V_4)$, we always have
$$\langle D \rangle \cap \PP(V_2\wedge V_4)=\langle C'\rangle.$$
Note that $\PP(V_2\wedge V_4)\cap \Gr(2, V_5)$ is a quadric threefold in $\PP(V_2\wedge V_4)\cong \PP^4$ since it is a hyperplane section of $\Gr(2, V_4)\subset \PP(\wedge^2 V_4)$. Hence $\langle C'\rangle\cap \Gr(2,V_5)$ is a quadric surface. Then $\langle C'\rangle\cap X=\langle C'\rangle\cap Q$ is a degree $4$ elliptic curve containing $C'$. By Lemma \ref{lem-ev-surjective-tau}, we have
\[C\cap \langle C' \rangle=\langle C\rangle \cap \langle C'\rangle \cap \Gr(2, V_5),\]
then $C\cap \langle C' \rangle$ is a hyperplane section of a quadric surface $\langle C'\rangle \cap \Gr(2, V_5)$ contained in $C$, which can only be a conic and we denote it by $Z$. So it is evident that $\langle Z\rangle\nsubseteq \Gr(2, V_5)$ by Lemma \ref{lem-conic-in-tau-cubic} and
\[\langle Z\rangle =\langle C\rangle \cap \langle C'\rangle.\]

\medskip

\textbf{Step VI}. \emph{Any residue cubic of $C$ is a $\tau$-cubic.}

By Step IV, to prove (2), it suffices to prove that there exists a residue $\tau$-cubic of $C$. Since there is a $\PP^1$-family of residue cubics of $C$ parameterized by $\PP(H^0(I_{C/S_{2,0}}(H)))$, we only need to prove that there are at most finitely many residue $\rho$-cubics of $C$.

%Finally, we prove that there are at most finitely many residue $\rho$-cubics of $C$. Since there is a $\PP^1$-family of residue cubics of $C$ parameterized by $\PP(H^0(I_{C/S_{2,0}}(H)))$, we deduce that there exists a residue $\tau$-cubic of $C$, which completes the proof of (2) by the claim above. 

To this end, as shown in Step V, $\langle C'\rangle\cap X$ is a degree $4$ curve for any residue $\rho$-cubic $C'$ of $C$, so it contains at most finitely many twisted cubics, which implies that $\langle C'\rangle$ contains finitely many twisted cubics in $X$. Consequently, if $\{C_i\}_{i\in I}$ is a set of residue $\rho$-cubics of $C$ such that the index set $I$ is infinite, then $\{\langle C_i\rangle \}_{i\in I}$ is a infinite set as well. According to Step V, each $C\cap \langle C_i\rangle$ is a conic in $C$ and $C$ can only contain finitely many conics, so if there exist infinitely many residue $\rho$-cubics, then we can find two residue $\rho$-cubics $C_1$ and $C_2$ of $C$ such that $$\langle C_1\rangle\neq \langle C_2\rangle$$
and 
\[Z=C\cap \langle C_1\rangle=C\cap \langle C_2\rangle.\]
In this case, we know that $C_1\subset \PP(\wedge^2 V_4)$ and $C_2\subset \PP(\wedge^2 V'_4)$ for two different $4$-dimensional subspaces $V_4, V_4'$ of $V_5$ such that $V_2\subset V_4'\cap V_4\subset V_5$. Then we have
\[\langle Z \rangle\subset \PP(V_2\wedge V_4)\cap \PP(V_2\wedge V_4'),\]
which implies $\langle Z \rangle =\PP(V_2\wedge V_3)=\PP(\wedge^2 V_3)$, where $V_3:=V_4\cap V_4'$. However, this implies $\langle Z\rangle \subset \Gr(2, V_5)$, which contradicts $\langle C\rangle \cap \langle C_1\rangle=\langle Z\rangle$ and 
$Z=\langle C\rangle \cap \langle C_1\rangle \cap \Gr(2, V_5)$. Thus we finish the proof.
\end{proof}

We finish this section by proving that being a residue cubic is a symmetric relation. Thus, every $\tau$-cubic is a residue cubic of another $\tau$-cubic.

%\begin{lemma}
%Let $F$ be a coherent sheaf on $\PP^n$. We denote by the hyperplane class of $\PP^n$ by $H$.

%\begin{enumerate}
%    \item If the Hilbert polynomial of $F$ is $t+1+k$ for an integer $k$, then $F\cong \oh_L(kH)$ for a line $L\subset \PP^n$.

%    \item If the Hilbert polynomial of $F$ is $2t+1+2k$ for an integer $k$, then $F\cong \oh_C(kH)$ for a conic $C\subset \PP^n$.
%\end{enumerate}
%\end{lemma}

\begin{lemma}\label{lem-residue-symmetry}
Let %$X$ be a general GM fourfold and 
$C\subset X$ be a $\tau$-cubic. If $C'$ is a residue cubic of $C$, then~$C$ is a residue cubic of $C'$.
\end{lemma}

\begin{proof}
This follows from a direct computation of commutative algebra, here we present the proof for completeness. Let $S_{2,0}$ be the sextic surface in Lemma \ref{lem-tau-in-sextic} containing $C$ and $C'$, and $D$ be the hyperplane section of $S_{2,0}$ containing $C$ and $C'$. Assume that $D_1$ is another hyperplane section of $S_{2,0}$ containing $C$ and $D_2$ is another hyperplane section containing $C'$. Let $C''$ be the residue curve of $C'$ in $D$. By Proposition \ref{prop-residue-tau-cubic}, we have $C=D\cap D_1$ and $C'=D\cap D_2$. Hence, we obtain 
\[I_{C}=I_D+I_{D_1},\quad I_{C'}=I_D+I_{D_2}.\]
Recall that by the definition of a residue scheme (cf.~\cite[Section 2]{hartshorne:double-plane}), we have
\[I_{C'}=(I_D:I_C), \quad I_{C''}=(I_D: I_{C'}),\]
where $(-:-)$ denotes the ideal quotient. From \cite[Exercise 1.12(v)]{atiyah:commutative-algebra},  since $(I_D: I_D)=\oh_X$, we have 
\[I_{C'}=(I_D:I_C)=(I_D: I_D+I_{D_1})=(I_D: I_D)\cap (I_D: I_{D_1})=(I_D: I_{D_1}).\]
In order to establish $C=C''$, we only need to show $I_C\subset I_{C''}$, which is equivalent to prove
\[I_D+I_{D_1}\subset I_{C''}=(I_D: I_{C'})=(I_D: (I_D: I_{D_1})).\]
To this end, by \cite[Exercise 1.12(i)]{atiyah:commutative-algebra} we have $I_D\subset (I_D: I_{C'})$. Combined with \cite[Exercise 1.12(ii)]{atiyah:commutative-algebra}, we get $I_{D_1}\cdot (I_D: I_{D_1})\subset I_D$. Hence, from the definition of the ideal quotient, we obtain
\[I_{D_1}\subset (I_D: (I_D: I_{D_1})).\]
This proves $I_D, I_{D_1}\subset I_{C''}=(I_D: I_{C'})$, which implies $I_C=I_D+I_{D_1}\subset I_{C''}$. Thus we complete the proof.
\end{proof}

\section{Projection of twisted cubics}\label{sec-projection-twisted_cubics}

In this section, we investigate the projection of objects associated with twisted cubics to the Kuznetsov components of GM fourfolds. We first describe the projection objects in Section \ref{subsec-proj} and then obtain their stability in Theorem \ref{thm-stability-proj-general}. We always assume $X$ to be a \emph{general} GM fourfold. %The main results involve the construction and description of the morphism in Theorem \ref{thm-pr-induce-map}, as well as the construction of Lagrangian covering families of double EPW cubes (Theorem \ref{thm-second-lag-cover-family}).

Recall that we have the projection functor
$$\mathrm{pr}_X:=\bR_{\cU_X}\bR_{\oh_X(-H)}\bL_{\oh_X}\bL_{\cU^{\vee}_X}\colon \D^b(X)\to \Ku(X).$$
Similar to $\Ku(X)$, we can define the alternative Kuznetsov component $\cA_X$ by
\[\D^b(X)=\langle\cA_X,\cU_X,\oh_X, \cU_X^{\vee},\oh_X(H)\rangle\]
and the corresponding projection functor is
$$\mathrm{pr}'_X:=\bR_{\oh_X(-H)}\bR_{\cU_X(-H)}\bL_{\cU_X}\bL_{\oh_X}\colon \D^b(X)\to \cA_X.$$
Moreover, there is an equivalence
\[\Xi:=\bL_{\oh_X}\circ (-\otimes \oh_X(H))\colon \cA_X\to \Ku(X).\]
A direct computation gives the following.

\begin{lemma}\label{lem-Xi}
Take an object $E\in \D^b(X)$ with $\RHom_X(\oh_X, E)=0$. Then we have
\[\pr_X(E(H))=\Xi(\pr'_X(E)).\]
\end{lemma}

\subsection{Projection objects}\label{subsec-proj}

First, we compute $\pr_X(I_C(H))$ and $\pr_X'(I_C)$ for a twisted cubic $C\subset X$ of type $\rho$ or $\tau$ and their self-Ext groups. We divide the computation by types of cubics.

\subsubsection{$\rho$-cubics}

\begin{proposition}\label{prop-proj-rho}
Let %$X$ be a general GM fourfold and 
$C\subset X$ be a $\rho$-cubic. Then 
\[\pr_X(I_C(H))=\pr_X(\oh_L(-1))\]
is the unique object fitting into the non-split exact triangle
\begin{equation}\label{eq-triangle-pr-rho}
\bR_{\cU_X}\bR_{\oh_X(-H)}\oh_L(-1)\to \pr_X(I_C(H))\to \cQ^{\vee}_X[1],
\end{equation}
where $L$ is the residue line of $C$ given in Lemma \ref{lem-rho-cubic}(3).
\end{proposition}

\begin{proof}
By \eqref{eq-def-S11}, we have  $\bL_{\oh_X}\bL_{\cU_X^{\vee}}(I_{S_{1,1}}(H))=0$. Then we see
\[\pr_X(I_C(H))=\pr_X(I_{C/S_{1,1}}(H)).\]
And from \eqref{eq-ICSH}, we obtain an exact triangle
\[\oh_X\xra{b} I_{C/S_{1,1}}(H) \to \bL_{\oh_X}(I_{C/S_{1,1}}(H)),\]
where $b$ corresponds to the unique hyperplane section of $S_{1,1}$ containing $C$ in Lemma \ref{lem-rho-cubic}. Hence, we obtain 
\[\bL_{\oh_X}\bL_{\cU_X^{\vee}}(I_{C/S_{1,1}}(H))=\bL_{\oh_X}\bL_{\cU_X^{\vee}}\bL_{\oh_X}(I_{C/S_{1,1}}(H)).\]
Note that $\im(b)=\oh_{S_{1,1}}$ since $S_{1,1}$ is integral by Lemma \ref{lem-rho-cubic}, hence we get $$\cH^{-1}(\bL_{\oh_X}(I_{C/S_{1,1}}(H)))=\ker(b)=I_{S_{1,1}}$$ and 
$$\cH^0(\bL_{\oh_X}(I_{C/S_{1,1}}(H)))=\mathrm{cok}(b)=\oh_L(-1),$$ where $L$ is the residue line in Lemma \ref{lem-rho-cubic}. Since $\pr_X(I_{S_{1,1}})=0$, we have
\[\pr_X(I_{C/S_{1,1}}(H))=\pr_X(\bL_{\oh_X}(I_{C/S_{1,1}}(H)))=\pr_X(\oh_L(-1)).\]
Finally, a direct computation of $\pr_X(\oh_L(-1))$ verifies \eqref{eq-triangle-pr-rho}. At the same time, the uniqueness result follows from
\[\RHom_X(\cQ_X^{\vee},\bR_{\cU_X}\bR_{\oh_X(-H)}\oh_L(-1))=\RHom_X(\cQ_X^{\vee},\oh_L(-1))=\CC,\]
which is due to $\RHom_X(\cQ^{\vee}_X, \cU_X)=\RHom_X(\cQ^{\vee}_X, \oh_X(-H))=0$ by \cite[Lemma 5.4(2)]{GLZ2021conics} and the fact that $\cQ_X|_L\cong \oh_L\oplus \oh_L\oplus \oh_L(1)$.
\end{proof}

Using the above calculation, we can determine the self-Ext groups of $\pr_X(I_C(H))$.

\begin{proposition}\label{prop-rhom-rho}
Let %$X$ be a general GM fourfold and 
$C \subset X$ be a $\rho$-cubic, then 
\[\RHom_X(\pr_X(I_C(H)), \pr_X(I_C(H)))=\CC\oplus \CC^6[-1]\oplus \CC[-2].\]
\end{proposition}

\begin{proof}
From the fact that $\chi(\pr_X(I_C(H)), \pr_X(I_C(H)))=-4$ and $S_{\Ku(X)}=[2]$, it suffices to show $\Hom_X(\pr_X(I_C(H)), \pr_X(I_C(H)))=\CC$ and $\Hom_X(\pr_X(I_C(H)), \pr_X(I_C(H))[k])=0$ for any $k<0$.

By adjunction of functors, we have
\[\RHom_X(\pr_X(\oh_L(-1)), \pr_X(\oh_L(-1)))=\RHom_X(\pr_X(\oh_L(-1)), \bL_{\oh_X}\bL_{\cU^{\vee}_X}(\oh_L(-1))).\]
And it is easy to see $\bL_{\oh_X}\bL_{\cU^{\vee}_X}(\oh_L(-1))=K_L[1]$, where $K_L$ is the kernel of the unique non-zero map $\cQ^{\vee}_X\twoheadrightarrow \oh_L(-1)$. Therefore,  we see $\RHom_X(\cQ_X^{\vee}, K_L)=0$, and from \eqref{eq-triangle-pr-rho} we have
\[\RHom_X(\pr_X(\oh_L(-1)), \pr_X(\oh_L(-1)))=\RHom_X(\bR_{\cU_X}\bR_{\oh_X(-H)}\oh_L(-1), K_L[1]).\]
By definition, we have an exact triangle
\[\bR_{\cU_X}\bR_{\oh_X(-H)}(\oh_L(-1))\to \bR_{\cU_X}(\oh_L(-1))\to F[2],\]
where
\[0\to \oh_X(-H)\to \cU_X^{\oplus 5}\to F=\bR_{\cU_X}\oh_X(-H)[1]\to 0,\]
and
\[\bR_{\cU_X}(\oh_L(-1))\to \oh_L(-1)\to \cU^{\oplus 3}_X[3].\]
Thus 
\[\Hom_X(\bR_{\cU_X}(\oh_L(-1)), K_L[1])=\Hom(\oh_L(-1), K_L[1])=\CC\]
and $\Hom_X(\bR_{\cU_X}(\oh_L(-1)), K_L[i])=0$ for $i\leq 0$. Then we see
\[\Hom_X(\bR_{\cU_X}\bR_{\oh_X(-H)}\oh_L(-1), K_L[i])=\Hom_X(\pr_X(I_C(H)), \pr_X(I_C(H))[i-1])=0\]
for $i\leq 0$ and we have an exact sequence
\[0\to \Hom_X(\bR_{\cU_X}(\oh_L(-1)), K_L[1])=\CC\to \Hom_X(\bR_{\cU_X}\bR_{\oh_X(-H)}\oh_L(-1), K_L[1])\to \Hom_X(F, K_L).\]
Therefore, to show $\Hom_X(\pr_X(I_C(H)), \pr_X(I_C(H)))=\CC$, it is enough to prove $\Hom_X(F, K_L)=0$. If $\Hom_X(\cU_X, K_L)=0$, as we have a surjective $\cU_X^{\oplus 5}\twoheadrightarrow F$, we see $\Hom_X(F, K_L)=0$ as desired. If there is a non-zero map $s\colon \cU_X\to K_L$, then from the exact sequence
\[0\to K_L\to \cQ^{\vee}_X\to \oh_L(-1)\to 0\]
and $\Hom_X(\cU_X, \cQ_X^{\vee})=\CC$, we see $\Hom_X(\cU_X, K_L)=\CC$. Hence, we get a commutative diagram
% https://q.uiver.app/#q=WzAsMTYsWzEsMSwiXFxjVV9YIl0sWzIsMSwiXFxjVV9YIl0sWzAsMiwiMCJdLFsxLDIsIktfTCJdLFsyLDIsIlxcY1Fee1xcdmVlfV9YIl0sWzMsMiwiXFxvaF9MKC0xKSJdLFs0LDIsIjAiXSxbMSwwLCIwIl0sWzIsMCwiMCJdLFsyLDMsIklfcSJdLFszLDMsIlxcb2hfTCgtMSkiXSxbNCwzLCIwIl0sWzAsMywiMCJdLFsxLDQsIjAiXSxbMiw0LCIwIl0sWzEsMywiXFxtYXRocm17Y29rfShzKSJdLFswLDMsInMiXSxbMSw0XSxbMCwxLCIiLDEseyJzdHlsZSI6eyJoZWFkIjp7Im5hbWUiOiJub25lIn19fV0sWzAsMSwiIiwxLHsib2Zmc2V0IjoxLCJzdHlsZSI6eyJoZWFkIjp7Im5hbWUiOiJub25lIn19fV0sWzcsMF0sWzgsMV0sWzQsOV0sWzksMTRdLFsyLDNdLFszLDRdLFs0LDVdLFs1LDZdLFs5LDEwXSxbMTAsMTFdLFs1LDEwLCIiLDEseyJzdHlsZSI6eyJoZWFkIjp7Im5hbWUiOiJub25lIn19fV0sWzUsMTAsIiIsMSx7Im9mZnNldCI6LTEsInN0eWxlIjp7ImhlYWQiOnsibmFtZSI6Im5vbmUifX19XSxbMTUsMTNdLFszLDE1XSxbMTIsMTVdLFsxNSw5XV0=
\[\begin{tikzcd}
	& 0 & 0 \\
	& {\cU_X} & {\cU_X} \\
	0 & {K_L} & {\cQ^{\vee}_X} & {\oh_L(-1)} & 0 \\
	0 & {\mathrm{cok}(s)} & {I_{q}} & {\oh_L(-1)} & 0 \\
	& 0 & 0
	\arrow["s", from=2-2, to=3-2]
	\arrow[from=2-3, to=3-3]
	\arrow[no head, from=2-2, to=2-3]
	\arrow[shift right, no head, from=2-2, to=2-3]
	\arrow[from=1-2, to=2-2]
	\arrow[from=1-3, to=2-3]
	\arrow[from=3-3, to=4-3]
	\arrow[from=4-3, to=5-3]
	\arrow[from=3-1, to=3-2]
	\arrow[from=3-2, to=3-3]
	\arrow[from=3-3, to=3-4]
	\arrow[from=3-4, to=3-5]
	\arrow[from=4-3, to=4-4]
	\arrow[from=4-4, to=4-5]
	\arrow[no head, from=3-4, to=4-4]
	\arrow[shift left, no head, from=3-4, to=4-4]
	\arrow[from=4-2, to=5-2]
	\arrow[from=3-2, to=4-2]
	\arrow[from=4-1, to=4-2]
	\arrow[from=4-2, to=4-3]
\end{tikzcd}\]
with all rows and columns exact, where $q$ is the unique $\sigma$-quadric contained in $X$. From the definition $F=\bR_{\cU_X}\oh_X(-H)[1]$, we have
\[\Hom_X(F,K_L)=\Hom_X(F, \mathrm{cok}(s)).\]
Since $\RHom_X(\cU_X, I_{q})=0$, we obtain $\Hom_X(F, I_{q})=\Hom_X(F, \mathrm{cok}(s))=0$ and the result follows.
\end{proof}

\subsubsection{$\tau$-cubics}

Next, we focus on $\tau$-cubics. We define a functor
\begin{equation}\label{eq-involution}
    T':=\bL_{\cU_X}\bL_{\oh_X}\circ (-\otimes \oh_X(H))[-1]\colon \cA_X\to \cA_X,
\end{equation}
which is an involution on $\cA_X$ (cf.~\cite[Theorem 4,15(2)]{bayer2022kuznetsov}). By Lemma \ref{lem-Xi}, for convenience, we will first work on the category $\cA_X$ and compute $\pr_X'(I_C)$.
%\[T':=\bR_{\oh_X(-H)}\bR_{\cU_X(-H)}\circ (-\otimes \oh_X(-H))[1]\colon \cA_X\to \cA_X.\]

We start with a computation for $\pr_X'(I_C(H))$ instead of $\pr_X'(I_C)$.

\begin{proposition}\label{prop-proj-tau-1}
Let %$X$ be a general GM fourfold and 
$C\subset X$ be a $\tau$-cubic. Then $\pr_X'(I_C(H))[-1]$ is the kernel of the evaluation map $\oh_X^{\oplus2}\twoheadrightarrow I_{C/S_{2,0}}(H)$. We have the following exact sequence
\begin{equation}\label{eq-tau-pr'-2}
  0\to I_{S_{2,0}}\to \pr_X'(I_C(H))[-1]\to I_{C'}\to 0,      
\end{equation}
where $S_{2,0}$ is the sextic surface in Lemma \ref{lem-tau-in-sextic} containing $C$ and $C'$ is any residue cubic of $C$.
\end{proposition}

\begin{proof}
By Proposition \ref{prop-residue-tau-cubic}(1), we know that the evaluation map of $I_{C/S_{2,0}}(H)$ is surjective. Then using Lemma \ref{lem-cohomology-ICS}(1), we have an exact sequence
\[0\to \bL_{\oh_X}(I_{C/S_{2,0}}(H))[-1]\to \oh_X^{\oplus 2}\to I_{C/S_{2,0}}(H) \to 0.\]
Moreover, by \eqref{eq-residue}, we have an exact sequence
\[0\to I_{S_{2,0}}\to \bL_{\oh_X}(I_{C/S_{2,0}}(H))[-1]\to I_{C'}\to 0.\]
As $C'$ is a $\tau$-cubic, applying $\Hom_X(\cU_X, -)$ to the above exact sequence, we have 
\[\RHom_X(\cU_X, \bL_{\oh_X}(I_{C/S_{2,0}}(H)))=0.\]
Therefore, combined with Lemma \ref{lem-cohomology-ICS}(2), we see  $\pr_X'(I_C(H))=\bL_{\oh_X}(I_{C/S_{2,0}}(H))$ and the result follows. 
\end{proof}

The following lemma explains the relation between different projection objects $\pr_X'(I_C(H))$ and $\pr_X'(I_C)$, and shows that taking residue cubics of a $\tau$-cubic corresponds to the action of the categorical involution $T'$ on its projection object.

\begin{lemma}\label{lem-T'}
Let %$X$ be a general GM fourfold and 
$C\subset X$ be a $\tau$-cubic, and let $S_{2,0}$ be the sextic surface in Lemma \ref{lem-tau-in-sextic} containing $C$. Let $C'$ be a residue cubic of $C$. Then we have
\[\pr_X'(I_C(H))=\pr_X'(I_{C/S_{2,0}}(H))=\pr_X'(I_{C'})[1]\]
and
\[T'(\pr_X'(I_{C'}))=\pr_X'(I_C).\]
\end{lemma}

\begin{proof}
By \eqref{eq-def-S20}, we have $\pr'_X(I_{S_{2,0}}(H))=0$. From the standard exact sequence $$0\to I_{S_{2,0}}(H) \to I_{C}(H)\to I_{C/S_{2,0}}(H)\to 0,$$ 
we obtain
\[\pr_X'(I_C(H))=\pr_X'(I_{C/S_{2,0}}(H)).\]
Using $\bR_{\cU_X(-H)}\cQ_X(-H)=\oh_X^{\oplus 5}(-H)$, we get $\pr_X'(I_{S_{2,0}})=0$. Then applying $\pr_X'$ to \eqref{eq-tau-pr'-2}, we have
\[\pr_X'(I_{C/S_{2,0}}(H))=\pr_X'(I_{C'})[1].\]
Since $T'$ is an involution and $\pr_X'(I_{C'})[1]=\pr_X'(I_C(H))$, to prove $T'(\pr_X'(I_{C'}))=\pr_X'(I_C)$, we only need to show
\begin{equation}\label{eq-lem-6.7}
T'(\pr_X'(I_C))=\pr_X'(I_C(H))[-1].
\end{equation}
Note that by Lemma \ref{lem-cubic-cohomology-1} and \ref{lem-homo-cubic}, we have
\[\pr_X'(I_C)=\bR_{\oh_X(-H)}\bR_{\cU_X(-H)}(I_C),\]
then we see
\[T(\pr_X'(I_C))=\bL_{\cU_X}\bL_{\oh_X}(I_C(H))[-1].\]
Using Lemma \ref{lem-cubic-cohomology-1} and \ref{lem-homo-cubic}(1), it is straightforward to check that $\bL_{\cU_X}\bL_{\oh_X}(I_C(H))=\pr_X'(I_{C}(H))$ and the result follows.
\end{proof}

Now we can fully determine $\pr'_X(I_C)$.

\begin{proposition}\label{prop-proj-tau}
Let %$X$ be a general GM fourfold and 
$C\subset X$ be a $\tau$-cubic. Then we have a short exact sequence
\begin{equation}\label{eq-seq-pr'IC}
    0\to I_{S_{2,0}} \to \pr_X'(I_C)\to I_C \to 0,
\end{equation}
where $S_{2,0}$ is the sextic surface in Lemma \ref{lem-tau-in-sextic} containing $C$.
\end{proposition}

\begin{proof}
Let $C'\subset S_{2,0}$ be a $\tau$-cubic such that $C'$ is a residue cubic of $C$. Then by Lemma \ref{lem-residue-symmetry}, $C$ is also a residue cubic of $C'$. Thus from Proposition \ref{prop-proj-tau-1}, we have an exact sequence
\[0\to I_{S_{2,0}}\to \pr_X'(I_{C'}(H))[-1]\to I_{C}\to 0.\]
Now the result follows from Lemma \ref{lem-T'} as $\pr_X'(I_{C'}(H))[-1]=\pr_X'(I_C)$.
\end{proof}

Next, we are going to compute self-$\Ext$ groups of $\pr_X'(I_C)$. %We start with a geometric lemma.

\begin{lemma}\label{lem-tau-homIC=1}
Let %$X$ be a general GM fourfold and 
$C\subset X$ be a $\tau$-cubic, and let $S_{2,0}$ be the sextic surface containing $C$ defined in Lemma \ref{lem-tau-in-sextic}. Then we have

\begin{enumerate}
    \item $\Hom_X(I_{S_{2,0}}, I_{C/S_{2,0}})=0$,

    \item $\Hom_X(I_C, I_{C/S_{2,0}})=\CC$, and

    \item $\Hom_X(I_{C/S_{2,0}}, I_{C/S_{2,0}})=\CC.$
\end{enumerate}

\end{lemma}

\begin{proof}
By \eqref{eq-def-S20}, we know $\RHom_X(\cQ_X(-H), I_{S_{2,0}})=\CC$, which is generated by the natural surjection $\cQ_X(-H)\twoheadrightarrow I_{S_{2,0}}$. Moreover, from the uniqueness part in Lemma \ref{lem-tau-in-sextic}, we have $\Hom_X(\cQ_X(-H), I_C)=\CC$ as well. This implies $\Hom_X(\cQ_X(-H), I_{C/S_{2,0}})=0$. If there is a non-zero map $I_{S_{2,0}}\to I_{C/S_{2,0}}$, composing it with $\cQ_X(-H)\twoheadrightarrow I_{S_{2,0}}$ will yield a non-zero map $\cQ_X(-H)\to I_{C/S_{2,0}}$, contradicting $\Hom_X(\cQ_X(-H), I_{C/S_{2,0}})=0$. This proves (1).

For (2), if $a\colon I_C\to I_{C/S_{2,0}}$ is a non-zero map, then we know that $\ker(a)=I_Z$, where $Z$ is a closed subscheme of $X$ containing $C$. Since $\Hom_X(I_{S_{2,0}}, I_C)=\CC$, to prove (2), we only need to show $Z=S_{2,0}$.

To this end, note that we have an inclusion 
\begin{equation}\label{eq-lem-tau-inclusion}
\mathrm{im}(a)=I_{C/Z}\hookrightarrow I_{C/S_{2,0}}.
\end{equation}
Moreover, by (1), we see $I_{S_{2,0}}\subset \ker(a)$, then $Z\subset S_{2,0}$ and we get a commutative diagram
% https://q.uiver.app/#q=WzAsMTAsWzEsMSwiSV9aPVxca2VyKGEpIl0sWzIsMSwiSV9DIl0sWzIsMCwiSV9DIl0sWzEsMCwiSV97U197MiwwfX0iXSxbMywwLCJJX3tDL1NfezIsMH19Il0sWzMsMSwiSV97Qy9afT1cXG1hdGhybXtpbX0oYSkiXSxbNCwwLCIwIl0sWzAsMCwiMCJdLFswLDEsIjAiXSxbNCwxLCIwIl0sWzMsMiwiIiwwLHsic3R5bGUiOnsidGFpbCI6eyJuYW1lIjoiaG9vayIsInNpZGUiOiJ0b3AifX19XSxbMiwxLCIiLDAseyJvZmZzZXQiOjEsInN0eWxlIjp7ImhlYWQiOnsibmFtZSI6Im5vbmUifX19XSxbMCwxXSxbMywwLCIiLDIseyJzdHlsZSI6eyJ0YWlsIjp7Im5hbWUiOiJob29rIiwic2lkZSI6InRvcCJ9fX1dLFsyLDEsIiIsMSx7InN0eWxlIjp7ImhlYWQiOnsibmFtZSI6Im5vbmUifX19XSxbMiw0XSxbNCw1LCJiIl0sWzEsNV0sWzgsMF0sWzUsOV0sWzQsNl0sWzcsM11d
\[\begin{tikzcd}
	0 & {I_{S_{2,0}}} & {I_C} & {I_{C/S_{2,0}}} & 0 \\
	0 & {I_Z=\ker(a)} & {I_C} & {I_{C/Z}=\mathrm{im}(a)} & 0
	\arrow[hook, from=1-2, to=1-3]
	\arrow[shift right, no head, from=1-3, to=2-3]
	\arrow[from=2-2, to=2-3]
	\arrow[hook, from=1-2, to=2-2]
	\arrow[no head, from=1-3, to=2-3]
	\arrow[from=1-3, to=1-4]
	\arrow["b", from=1-4, to=2-4]
	\arrow[from=2-3, to=2-4]
	\arrow[from=2-1, to=2-2]
	\arrow[from=2-4, to=2-5]
	\arrow[from=1-4, to=1-5]
	\arrow[from=1-1, to=1-2]
\end{tikzcd}\]
with all rows are exact and $b$ is the induced map. By the Snake Lemma, we see $b$ is surjective and $\ker(b)=I_{Z/S_{2,0}}$.

If $S_{2,0}$ is integral, then from \eqref{eq-lem-tau-inclusion}, we have $\mathrm{Supp}(I_{C/Z})=\mathrm{Supp}(I_{C/S_{2,0}})=S_{2,0}$. However, this implies $\ker(b)$ is supported in dimension $\leq 1$, and the only possible case is $\ker(b)=0$. Thus, combined with the subjectivity, $b$ is an isomorphism. Hence, $Z=S_{2,0}$ and (2) follows. 

If $S_{2,0}$ is not integral, then from Lemma \ref{lem-surface-in-GM4}(3), we know that $S_{2,0}=q\cup S'$, where $S'$ is an integral surface of degree $4$. In the following, we show that $Z=S_{2,0}$ holds as well. If $Z\neq S_{2,0}$, as $S_{2,0}=q\cup S'$, we know that $\ch_2(I_{C/Z})=\ch_2(\oh_q)$ or $\ch_2(\oh_{S'})$. In other words, $Z_1:=\mathrm{Supp}(I_{C/Z})=q$ or $S'$. Since $q$ and $S'$ are integral, and $I_{C/Z}$ is a rank one torsion-free sheaf on $Z_1$, we observe that $\Hom_X(I_{C/Z}, I_{C/Z})=\CC$. If we set $Z_2:=\mathrm{Supp}(I_{Z/S_{2,0}})$, then $\{Z_1, Z_2\}=\{q, S'\}$. Since $q$ and $S'$ have no common components, $\im(a)$ is not contained in $\ker(b)$. Hence, we obtain a non-zero map
\[\mathrm{im}(a)=I_{C/Z}\hookrightarrow I_{C/S_{2,0}}\xra{b}I_{C/Z}=\im(a),\]
which is an isomorphism by $\Hom_X(I_{C/Z}, I_{C/Z})=\CC$. Thus, we have the splitting 
\[I_{C/S_{2,0}}= I_{C/Z}\oplus I_{Z/S_{2,0}}.\]
Since $I_{C/Z}\hookrightarrow I_{C/S_{2,0}} \subset \oh_{S_{2,0}}$, we may assume that $I_{C/Z}= I_{Z_3/S_{2,0}}$ for a closed subscheme $Z_3\subset S_{2,0}$, hence $C= Z\cap Z_3$ and $S_{2,0}=Z\cup Z_3$. We claim that $q=Z_3$ or $Z$, which contradicts $C= Z\cap Z_3$ since $C$ is a $\tau$-cubic. Indeed, without loss of generality, we may assume that $Z_1=q$ and $Z_2=S'$, then $q\subset Z$. Since $X$ does not contain any plane, we see $Z.H^2=2$ or $4$. If $Z.H^2=4$, then by $S_{2,0}=Z\cup Z_3$ and $C= Z\cap Z_3$, we get $Z_3.H^2=2$, which contradicts the uniqueness of quadric surfaces in $X$. Therefore, we have $q.H^2=Z.H^2=2$. Since $Z$ is pure and $q\subset Z$, we can conclude that $q=Z$. When $Z_2=q$ and $Z_1=S'$, we have $q\subset Z_3$ and the rest of the argument is similar to the previous case. This completes the proof of (2).

%The Hilbert polynomial of $I_{Z/S_{2,0}}$ is of form $2t^2+xt+y$ for $x,y\in \mathbb{Q}$. Since $$\RHom_X(\oh_X, I_{C/S_{2,0}})=\RHom_X(\cU_X, I_{C/S_{2,0}})=0,$$ we have
%\[\RHom_X(\oh_X, I_{Z/S_{2,0}})=\RHom_X(\cU_X, I_{Z/S_{2,0}})=0,\]
%which implies $x=y=0$ using the Hirzebruch--Riemann--Roch formula. This shows that $\oh_Z$ has the same Hilbert polynomial as $\oh_q$. Thus, $Z=q$ follows from $q\subset Z$.

Now (3) can be directly deduced from (2) because we have a natural surjection $I_C\twoheadrightarrow I_{C/S_{2,0}}$.
\end{proof}

\begin{proposition}\label{prop-rhom-tau}
Let %$X$ be a general GM fourfold and 
$C\subset X$ be a $\tau$-cubic. Then we have
\[\RHom_X(\pr'_X(I_C), \pr'_X(I_C))=\RHom_X(\pr_X(I_C(H)), \pr_X(I_C(H)))=\CC\oplus \CC^6[-1]\oplus \CC[-2].\]
\end{proposition}

\begin{proof}
From the fact that $\chi(\pr_X(I_C(H)), \pr_X(I_C(H)))=-4$ and $S_{\Ku(X)}=[2]$, to prove the statement, we only need to show $\Hom_X(\pr_X(I_C(H)), \pr_X(I_C(H)))=\CC$ and $\Hom_X(\pr_X(I_C(H)), \pr_X(I_C(H))[k])=0$ for any $k<0$. Since $\pr_X(I_C(H))=\Xi(\pr_X'(I_C))$, in the following, we compute $\RHom_X(\pr'_X(I_C), \pr_X'(I_C))$ instead. 

From Proposition \ref{prop-proj-tau-1}, we have an exact triangle
\[\oh_X^{\oplus 2}\to I_{C/S_{2,0}}(H)\to \bL_{\oh_X}(I_{C/S_{2,0}}(H))=\pr_X'(I_C(H)).\]
Using Lemma \ref{lem-tau-homIC=1}(3) and the adjunction
\[\RHom_X(\pr_X'(I_C(H)), \pr_X'(I_C(H)))=\RHom_X(I_{C/S_{2,0}}(H), \pr_X'(I_C(H))),\]
we obtain
\[\Hom_X(\pr_X'(I_C(H)), \pr_X'(I_C(H)))=\CC,\quad \text{and }\Hom_X(\pr_X'(I_C(H)), \pr_X'(I_C(H))[k])=0 \text{ for any }k< 0.\]
Then the result follows from applying $T'$ and Lemma \ref{lem-T'}.
\end{proof}

\subsection{Stability of projection objects}

Using the above results, we can easily obtain the stability of projection objects when $X$ is very general:

\begin{theorem}\label{thm-stability-proj}
If $X$ is very general and 
%be a very general GM fourfold and 
$C\subset X$ is a twisted cubic %. If $C$ is not a 
which is not $\sigma$-cubic, then $\pr_X(I_C(H))$ is stable with respect to any stability condition on $\Ku(X)$.
\end{theorem}

\begin{proof}
The result follows from the following Proposition \ref{prop-rhom-rho}, \ref{prop-rhom-tau} and \cite[Lemma 4.12(2)]{FGLZ24}.
\end{proof}

\begin{remark}
According to our computation, $\hom_X(\pr_X(I_C(H)), \pr_X(I_C(H)))= 3$ for any $\sigma$-cubic $C$, thus $\pr_X(I_C(H))$ is not stable.
\end{remark}

Although the description of self-Ext groups in Proposition \ref{prop-rhom-rho} and Proposition \ref{prop-rhom-tau} holds for general $X$, we can not directly deduce the stability of $\pr_X(I_C(H))$ from them because of the lack of \cite[Lemma 4.12(2)]{FGLZ24}. However, by applying a similar deformation argument as in \cite[Theorem 4.16]{FGLZ24}, we can also extend Theorem \ref{thm-stability-proj} from the very general case to the general case. 

For simplicity, we denote by $\mathrm{H}_X\subset \Hilb^{3t+1}_X$ the open subscheme parameterizing twisted cubics that are not of type $\sigma$. For a smooth projective morphism $\cX\to S$ over a smooth scheme $S$ over $\CC$ and a $S$-linear semi-orthogonal component $\cD\subset \D^b(\cX)$, we denote by $\cM_{\mathrm{pug}}(\cD/S)$ the moduli stack of universally gluable objects in $\cD$ over $S$, defined in \cite[Definition 9.1]{BLMNPS21}. According to \cite[Lemma 21.12]{BLMNPS21}, the moduli stack $\cM^X_{\sigma_X}(1,-1)$ of $\sigma_X$-semistable objects is an open substack of $\cM_{\mathrm{pug}}(\Ku(X)/\CC)$.

\begin{theorem}\label{thm-stability-proj-general}
If $X$ is a general GM fourfold and 
$C\subset X$ is a twisted cubic %. If $C$ is not a 
which is not a $\sigma$-cubic, then $\pr_X(I_C(H))$ is stable with respect to a generic $\sigma_X\in \Stab^{\circ}(\Ku(X))$.
\end{theorem}

\begin{proof}
By Proposition \ref{prop-rhom-rho} and Proposition \ref{prop-rhom-tau}, for a general GM fourfold $X$, we have a morphism
\[\varphi_X\colon \mathrm{H}_X\to \cM_{\mathrm{pug}}(\Ku(X)/\CC)\]
induced by the projection functor $\pr_X$ given by $[C]\mapsto [\pr_X(I_C(H))]$ at the level of $\CC$-point. Therefore, given a $\CC$-point $s\colon \Spec(\CC)\to \mathrm{H}_X$ corresponding to a twisted cubic $[C]\in \mathrm{H}_X$, to show the $\sigma_X$-semistability of $\pr_X(I_C(H))$, we need to show the composition
\[\Spec(\CC)\xra{s}\mathrm{H}_X\xra{\varphi_X}\cM_{\mathrm{pug}}(\Ku(X)/\CC)\]
factor through the natural open immersion $\cM^X_{\sigma_X}(1,-1)\subset \cM_{\mathrm{pug}}(\Ku(X)/\CC)$. By \cite[Theorem 12.17(3)]{BLMNPS21}, this is equivalent to finding an extension $\kappa$ of $\CC$ such that the composition
\[\Spec(\kappa)\xra{s_{\kappa}}\mathrm{H}_X\xra{\varphi_X}\cM_{\mathrm{pug}}(\Ku(X)/\CC)\]
factors through $\cM^X_{\sigma_X}(1,-1)\subset \cM_{\mathrm{pug}}(\Ku(X)/\CC)$. We denote by $|\cM|$ the associated topological space of an
algebraic stack $\cM$.  From the definition of topological spaces of algebraic stacks (cf.~\cite[\href{https://stacks.math.columbia.edu/tag/04XG}{Tag 04XG}]{stacks-project}), it is also equivalent to say 
$$
s_{\kappa}\in |\cM^X_{\sigma_X}(1,-1)|\cap \varphi_X(|\mathrm{H}_X|)\subset |\cM_{\mathrm{pug}}(\Ku(X)/\CC)|
$$  
Thus, to prove that $\pr_X(I_C(H))$ is $\sigma_X$-semistable for any $[C]\in \mathrm{H}_X$, we only need to show the inclusion
\begin{equation}\label{eq-deform}
    \varphi_X(\mathrm{H}_X)\subset |\cM^X_{\sigma_X}(1,-1)|
\end{equation}
in $|\cM_{\mathrm{pug}}(\Ku(X)/\CC)|$. Then the result can be deduced from Theorem \ref{thm-stability-proj} as follows.

Let $\cM^{\mathrm{GM}}_4$ be the moduli stack of smooth ordinary GM fourfolds, which is a smooth irreducible Deligne--Mumford stack of finite type and separated over $\CC$ (cf.~\cite[Proposition A.2]{kuznetsov2018derived} and \cite[Corollary 5.12]{debarre:gm-moduli}). We take a connected component $W$ of an \'etale atlas of $\cM^{\mathrm{GM}}_4$ that dominants $\cM^{\mathrm{GM}}_4$ and denote by $\pi\colon \cX\to W$ be the corresponding family of ordinary GM fourfolds. We set $\cX_b:=\pi^{-1}(b)$. By replacing $W$ with an open dense subscheme, we can assume that $\cX_b$ is general for each $\CC$-point of $W$ and $\cX\to W$ factor through a closed embedding $\cX\hookrightarrow \Gr_W(2,\cV_5)$ over $W$ for a rank $5$ vector bundle $\cV_5$ on $W$. We denote by $\cQ_{\cX}$ the pull-back of the tautological quotient bundle via the natural morphism $\cX\to \Gr_W(2, \cV_5)$.

By \cite[Lemma 5.9]{bayer2022kuznetsov}, there is an $W$-linear semi-orthogonal component $\Ku(\cX)\subset \mathrm{D}_{\perf}(\cX)$ such that $\Ku(\cX)_b\simeq \Ku(\cX_b)$ for any $b\in W$. Therefore, using the construction in \cite[Section 4]{perry2019stability}, there is a relative stability condition $\underline{\sigma}$ on $\Ku(\cX)$ over $W$ such that $\underline{\sigma}|_b\in \Stab^{\circ}(\Ku(\cX_b))$ for each $b\in W$.

Now, let $\cI$ be the universal ideal sheaf on $\cX\times_W \Hilb^{3t+1}_{\cX/W}$, where $\Hilb^{3t+1}_{\cX/W}$ is the relative Hilbert scheme of twisted cubics on $\cX$ over $W$ parameterizing twisted cubics in the fibers of $\pi$. Then we set 
\[\cH_{\cX/W}:=\Hilb^{3t+1}_{\cX/W} \setminus \mathrm{Supp}(R^0p_{2_*}\cH om_{\cX\times_W \Hilb^{3t+1}_{\cX/W}}(p_1^*\cQ^{\vee}_{\cX}, \cI)),\]
where $p_1\colon \cX\times_W \Hilb^{3t+1}_{\cX/W}\to \cX$ and $p_2\colon \cX\times_W \Hilb^{3t+1}_{\cX/W}\to \Hilb^{3t+1}_{\cX/W}$ are projections. Then $\cH_{\cX/W}$ is an open subscheme of $\Hilb^{3t+1}_{\cX/W}$. Moreover, by Lemma \ref{lem-homo-cubic}(3), $\cH_{\cX/W}$ parameterises twisted cubics in the fibers of $\pi$ that are not $\sigma$-cubics and $(\cH_{\cX/W})_b = \mathrm{H}_{\cX_b}$ for any $\CC$-point $b\in W$.

By Proposition \ref{prop-rhom-rho} and Proposition \ref{prop-rhom-tau}, we have a morphism 
\[\varphi_{\cX}\colon \cH_{\cX/W}\to \cM_{\mathrm{pug}}(\Ku(\cX)/W)\]
induced by the relative projection functor $\mathrm{D}_{\mathrm{perf}}(\cX)\to \Ku(\cX)$ via projecting the universal idea sheaf on $\cX\times_W \cH_{\cX/W}$. In particular, $\varphi_{\cX}$ satisfies $(\varphi_{\cX})_b=\varphi_{\cX_b}$ for any $\CC$-point $b\in W$. Now, we set $$\cU:=\varphi_{\cX}^{-1}(\cM^{\cX}_{\underline{\sigma}}(1,-1)),$$ which is an open subscheme of $\cH_{\cX/W}$ by \cite[Lemma 21.12]{BLMNPS21}. Let $\mathcal{Z}$ be the complement of $|\cU|$ in $|\cH_{\cX/W}|$ and $p\colon \cH_{\cX/W}\to W$ be the natural morphism, which is flat and of finite type. We define $V:=W\setminus p(\mathcal{Z})$, which is a constructible subset by \cite[Corollaire (5.9.2), Th\'eor\`eme (5.9.4)]{laumon:stack-book}. According to Theorem \ref{thm-stability-proj}, $V$ contains very general points of $W$, hence is dense in $W$. As $V$ is constructible and dense in $W$, it contains an open dense subscheme $S$ of $W$ by \cite[\href{https://stacks.math.columbia.edu/tag/005K}{Tag 005K}]{stacks-project}. Then from the construction, we see that GM fourfolds represented by points in $S$ satisfy \eqref{eq-deform}, and the result follows.
\end{proof}

\section{Double EPW cubes and moduli spaces}\label{sec-epw-cube}

Building on Theorem \ref{thm-stability-proj-general}, we continue to study the relation between $\Hilb^{3t+1}_X$ and $M^X_{\sigma_X}(1,-1)$ in this section. Our final goal is to prove Theorem \ref{cor-cube-as-MRC} and Theorem \ref{thm-second-lag-cover-family}. As in the previous sections, we still assume the GM fourfold $X$ to be \emph{general}.

\subsection{Limits of $\tau$-cubics}

We begin with determining the limits of $\tau$-cubics on $X$. Recall that $\Hilb_X^{3t+1}$ can be stratified as
\[\Hilb_X^{3t+1}=\mathrm{H}^{\tau}\bigsqcup \mathrm{H}^{\rho}\bigsqcup \mathrm{H}^{\sigma},\]
where $\mathrm{H}^{\tau}, \mathrm{H}^{\rho}$, and $\mathrm{H}^{\sigma}$ denote the loci parameterizing $\tau$-cubics, $\rho$-cubics, and $\sigma$-cubics, respectively. According to Lemma \ref{lem-homo-cubic}, $\mathrm{H}^{\rho}$ and $\mathrm{H}^{\sigma}$ are degeneracy loci of certain sheaves on $\Hilb_X^{3t+1}$, which are equipped with natural closed subschemes structures. Therefore, $\mathrm{H}^{\tau}$ is an open subscheme of $\Hilb_X^{3t+1}$.

%Firstly, we demonstrate some basic properties of the $\sigma$-quadric of $X$.

The following lemma shows that $\Hilb_X^{3t+1}$ is smooth of dimension $7$ along $\mathrm{H}^{\sigma}$.

\begin{lemma}\label{lem-local-dim-sigma}
Let %$X$ be a general GM fourfold and 
$C\subset X$ be a $\sigma$-cubic, then 

\begin{enumerate}
    \item $\dim \mathrm{H}^{\sigma}=5$,

    \item $H^0(N_{C/X})=\Ext^1_X(I_C, I_C)=\CC^7$, and

    \item $H^1(N_{C/X})=0$ and $\dim_{[C]} \Hilb_X^{3t+1}=7$.
\end{enumerate}

%\[\dim \mathrm{H}^{\sigma}=5,\quad H^0(N_{C/X})=\Ext^1_X(I_C, I_C)=\CC^7, \quad H^1(N_{C/X})=0,\quad \text{and }\dim_{[C]} \Hilb_X^{3t+1}=7.\]
\end{lemma}

\begin{proof}
As $X$ is general, the $\sigma$-quadric $q\subset X$ is smooth. Hence, $q\cong \PP^1\times \PP^1$, and $\oh_q(C)\cong \oh_q(2,1)$ or $\oh_q(1,2)$. Therefore, we have $$\dim \mathrm{H}^{\sigma}=\dim (|\oh_q(2,1)|\sqcup |\oh_q(1,2)|)=5.$$ Without loss of generality, we may assume that $\oh_q(C)\cong \oh_q(2,1)$. Thus $N_{C/q}\cong \oh_q(2,1)|_C$. Using the exact sequence 
\[0\to \oh_q\to \oh_q(2,1)\to \oh_q(2,1)|_C\to 0,\]
we get $H^0(N_{C/q})=\CC^5$ and $H^1(N_{C/q})=0$.

From $\oh_q(C)\cong \oh_q(2,1)$, we have an exact sequence
\[0\to N_{q/X}(-2,-1)\to N_{q/X}\to N_{q/X}|_C\to 0.\]
By Serre duality, we obtain $H^2(N_{q/X}(-2,-1))=H^0(N_{q/X}(0,-1))^{\vee}$. As $H^0(N_{q/X})=0$ by Lemma \ref{lem-normal-bundle-q}, we see $H^0(N_{q/X}(0,-1))=0$ as well. Thus, from the exact sequence above and $H^*(N_{q/X})=0$, we deduce that  $H^1(N_{q/X}|_C)=0$. Finally, combined with $H^1(N_{C/q})=0$ and the exact sequence
\[0\to N_{C/q}\to N_{C/X}\to N_{q/X}|_C\to 0,\]
we obtain $H^1(N_{C/X})=0$ and $H^0(N_{C/X})=\CC^7$. This means that $[C]\in \Hilb^{3t+1}_X$ is a smooth point. Thus, we have the local dimension $\dim_{[C]} \Hilb_X^{3t+1}=7$.
\end{proof}

For $\rho$-cubics, we can also prove that they represent smooth points of $\Hilb_X^{3t+1}$. However, the subsequent lemma, which is simpler, suffices for our purposes.

\begin{lemma}\label{lem-local-dim-rho}
Let %$X$ be a general GM fourfold and 
$C\subset X$ be a $\rho$-cubic, then $\dim \mathrm{H}^{\rho}=6$ and $\dim_{[C]} \Hilb_X^{3t+1}\geq 7$.
\end{lemma}

\begin{proof}
By Proposition \ref{prop-proj-rho}, we have a morphism $\pi\colon \mathrm{H}^{\rho}\to \Hilb^{t+1}_X$, mapping a $\rho$-cubic to its residue line. Moreover, as every $\rho$-cubic is contained in the zero locus of a unique section of $\cU^{\vee}_X$, we have a morphism $\theta\colon \mathrm{H}^{\rho}\to \PP(H^0(\cU^{\vee}_X))\cong \PP^4$. Note that $\pi$ and $\theta$ are both surjective.

Firstly, we show that $\mathrm{H}^{\rho}$ is irreducible. For any line $[L]\in \Hilb^{t+1}_X$, the image of the morphism 
$$\theta|_{\pi^{-1}([L])}\colon \pi^{-1}([L]) \to \PP(H^0(\cU^{\vee}_X))$$ is $\PP(H^0(\cU^{\vee}_X\otimes I_L))\cong \PP^1$, i.e.~the section $s\in H^0(\cU^{\vee}_X)$ such that $L\subset Z(s)$. Hence, $$(\theta|_{\pi^{-1}([L])})^{-1}([s])\cong \PP(H^0(I_{L/Z(s)}(H)))\cong \PP^2,$$ which is irreducible. As the morphism $$\theta|_{\pi^{-1}([L])}\colon \pi^{-1}([L]) \to \PP^1$$ is projective and surjective, we deduce that $\pi^{-1}([L])$ is irreducible of dimension $3$. Therefore, from the fact that $\Hilb^{t+1}_X$ is irreducible of dimension $3$ by \cite[Proposition 5.3]{debarre2019gushel}, the locus $\mathrm{H}^{\rho}$ is irreducible of dimension $6$.

Since $\mathrm{H}^{\rho}$ is irreducible, to prove $\dim_{[C]} \Hilb_X^{3t+1}\geq 7$ for any $[C]\in \mathrm{H}^{\rho}$, we only need to establish this for general $\rho$-cubics. As a general $\rho$-cubic is contained in $Z(s)$ for $s\in H^0(\cU^{\vee}_X)$ such that $Z(s)$ is smooth, the claim $\dim_{[C]} \Hilb_X^{3t+1}\geq 7$ follows from $\chi(N_{C/X})=7$ and \cite[(2.15.3)]{kollar:book-rational-curve}.
\end{proof}

Now we can show that every twisted cubic on $X$ is a limit of $\tau$-cubics.

\begin{proposition}\label{prop-closure-tau}
%Let $X$ be a general GM fourfold. Then 
We have $\overline{\mathrm{H}^{\tau}}=\Hilb^{3t+1}_X$, i.e.~$\mathrm{H}^{\tau}$ is an open dense subscheme of $\Hilb^{3t+1}_X$. %Hence $\dim \Hilb^{3t+1}_X\geq 7$.
\end{proposition}

\begin{proof}
As $\dim \mathrm{H}^{\sigma}=\dim \Hilb_q^{3t+1}=5$ and $\dim \mathrm{H}^{\rho}=6$, from the computation of local dimension in Lemma \ref{lem-local-dim-sigma} and \ref{lem-local-dim-rho}, we know that $\mathrm{H}^{\sigma}$ and $\mathrm{H}^{\rho}$ cannot be irreducible components of $\Hilb^{3t+1}_X$. This implies $\overline{\mathrm{H}^{\tau}}=\Hilb^{3t+1}_X$.
\end{proof}

\subsection{Double EPW cubes as MRC quotient of Hilbert schemes}

In this subsection, we aim to verify the relation between Hilbert schemes of twisted cubics and double EPW cubes. We start with a homological criterion of residue cubics. Recall that the involution $T'$ of $\cA_X$ is defined in \eqref{eq-involution}.

\begin{lemma}\label{lem-criterion-residue-cubic}
Let %$X$ be a general GM fourfold and 
$C, C'\subset X$ be two $\tau$-cubics. Then 
\[T'(\pr_X'(I_C))\cong \pr_X'(I_{C'})\]
if and only if $C'$ is a residue cubic of $C$.
\end{lemma}

\begin{proof}
If $C'$ is a residue cubic of $C$, by Lemma \ref{lem-T'}, we have $T'(\pr_X'(I_C))\cong \pr_X'(I_{C'})$. Conversely, assume that $T'(\pr_X'(I_C))\cong \pr_X'(I_{C'})$. By \eqref{eq-lem-6.7}, we obtain
\[\pr_X'(I_C(H))[-1]\cong \pr_X'(I_{C'}).\]
From Proposition \ref{prop-proj-tau-1}, we have  $\pr_X'(I_C(H))=\bL_{\oh_X}I_{C/S_{2,0}}(H)$. Hence,  using Lemma \ref{lem-cohomology-ICS}(1) and Proposition \ref{prop-residue-tau-cubic}(1), we see $\Hom_X(\pr_X'(I_C(H))[-1], \oh_X)=\CC^2$, and the cone of the natural map 
\[\mathrm{coev}\colon \pr_X'(I_C(H))[-1]\to \Hom_X(\pr_X'(I_C(H))[-1], \oh_X)\otimes \oh_X\]
is $I_{C/S_{2,0}}(H)$. On the other hand, by $\pr_X'(I_C(H))[-1]\cong \pr_X'(I_{C'})$ and \eqref{eq-seq-pr'IC}, we obtain an exact sequence 
\[0\to \oh_{S_{2,0}}\to \mathrm{cone}(\mathrm{coev})=I_{C/S_{2,0}}(H)\to \oh_{C'}\to 0.\]
This implies the exact sequence
\[0\to \oh_{C'}(-H)\to \oh_{S_{2,0}\cap H}\to \oh_C\to 0.\]
Hence, $C'$ is a residue cubic of $C$.
\end{proof}

Now we can prove one of our main theorems.

\begin{theorem}\label{thm-pr-induce-map}
Let $X$ be a general GM fourfold.    %$X$ be a very general GM fourfold. 
Then for a generic stability condition $\sigma_X\in \Stab^{\circ}(\Ku(X))$, the projection functor $\pr_X$ induces a dominant rational map
\begin{align*}
   pr \colon \Hilb_X^{3t+1} &\dashrightarrow M_{\sigma_X}^X(1,-1)\\
   [C] & \mapsto [\pr_X(I_C(H))].
\end{align*}
The map $pr$ is defined on $\Hilb_X^{3t+1}\setminus \mathrm{H}^{\sigma}$ and its general fibers are $\PP^1$.
\end{theorem}

\begin{proof}
By Theorem \ref{thm-stability-proj-general}, $\pr_X$ induces a rational map
\[pr\colon \Hilb_X^{3t+1}\dashrightarrow M_{\sigma_X}^X(1,-1)\]
defined on $\Hilb_X^{3t+1}\setminus \mathrm{H}^{\sigma}$, which is an open dense subset of $\Hilb^{3t+1}_X$ by Proposition \ref{prop-closure-tau}. 

To determine general fibers of $pr$, we fix a $\tau$-cubic $C$. By Lemma \ref{lem-Xi}, $pr^{-1}([\pr_X(I_C(H))])$ parameterizes $\tau$-cubics $C'$ such that $\pr'_X(I_{C'})\cong \pr'_X(I_C)$. Let $C''$ be a residue cubic of $C$, using Lemma \ref{lem-T'}, $\pr'_X(I_{C'})\cong \pr'_X(I_C)$ is equivalent to $T'(\pr'_X(I_{C'}))\cong \pr'_X(I_{C''})$. In other words, according to Lemma \ref{lem-criterion-residue-cubic}, $pr^{-1}([\pr_X(I_C(H))])$ parameterizes residue cubics of $C''$, which is $\PP^1\cong \PP(H^0(I_{C''/S_{2,0}}(H)))$ by Remark \ref{rmk-P1}.

As $\dim \Hilb^{3t+1}_X\geq 7$ and $pr$ only contracts curves in $\mathrm{H}^{\tau}$, we deduce that $\dim pr(\Hilb^{3t+1}_X\setminus \mathrm{H}^{\sigma})\geq 6$. Then $pr$ is dominant since $M_{\sigma_X}^X(1,-1)$ is irreducible of dimension $6$ and the description of the general fibers follows.
\end{proof}

Since hyperk\"ahler manifolds are not uniruled (cf.~\cite[Theorem 1]{mori:criterion-uniruled}), we immediately deduce the following.

\begin{corollary}
Let $X$ be a general GM fourfold. Then for a generic $\sigma_X\in \Stab^{\circ}(\Ku(X))$, the hyperk\"ahler sixfold $M_{\sigma_X}^X(1,-1)$ is the MRC quotient of $\Hilb^{3t+1}_X$.
\end{corollary}

Recall that there is a double EPW cube $\widetilde{C}_X$ associated with %a general GM fourfold 
$X$ which is a hyperk\"ahler sixfold. By \cite[Theorem 1.1]{kapustka2022epw} and \cite[Theorem 5.1]{debarre2019gushel}, $\widetilde{C}_X$ and $X$ have the same period point when $X$ is general. Therefore, $\widetilde{C}_X$ and $M_{\sigma_X}^X(1,-1)$ also have the same period point and we get a birational isomorphism between $\widetilde{C}_X$ and $M_{\sigma_X}^X(1,-1)$. Together with Theorem \ref{thm-pr-induce-map}, we obtain a dominant rational map
\[\Hilb_X^{3t+1}\dashrightarrow \widetilde{C}_X,\]
whose general fibers are $\PP^1$. In conclusion, we get:

\begin{theorem}\label{cor-cube-as-MRC}
Let $X$ be a general GM fourfold. Then the double EPW cube $\widetilde{C}_X$ associated with $X$ is the MRC quotient of $\Hilb_X^{3t+1}$.
\end{theorem}

\subsection{Lagrangian covering families of double EPW cubes}\label{subsec-covering-cube}

We end this section by discussing the construction of Lagrangian covering families on $\widetilde{C}_X$. We first recall a general definition, which can be viewed as a generalization of Lagrangian fibrations.

\begin{definition}\label{def-lag-family}
A \emph{Lagrangian family} of a projective hyperk\"ahler manifold $M$ is a diagram 
% https://q.uiver.app/#q=WzAsMyxbMCwwLCJNIl0sWzEsMCwiWCJdLFswLDEsIkIiXSxbMCwyLCJwIiwyXSxbMCwxLCJxIl1d
\[\begin{tikzcd}
	L & M \\
	B
	\arrow["p"', from=1-1, to=2-1]
	\arrow["q", from=1-1, to=1-2]
\end{tikzcd}\]
where $p$ is flat and projective, $L$ and $B$ are quasi-projective manifolds, such that for a general point $b\in B$, $q|_{p^{-1}(b)}$ is generically finite to its image, which is a Lagrangian subvariety of $M$. The manifold $B$ is called the \emph{base} of a Lagrangian family. A Lagrangian family is called a \emph{Lagrangian covering family} if $q$ is dominant.
\end{definition}

\begin{remark}
%There are other different definitions of Lagrangian families in literature. However, Lemma \ref{lem-lag-family-equiv} allows us to construct those families from the one defined in Definition \ref{def-lag-family}.
There is another definition of Lagrangian covering families in \cite[Definition 0.4]{voisin2021lefschetz} which does not assume $p$ to be flat but requires $B$ to be projective. It is easy to see the existence of a Lagrangian family in the sense of Definition \ref{def-lag-family} implies the existence of that in the sense of \cite[Definition 0.4]{voisin2021lefschetz} by compactifying $p$ and resolving singularities and indeterminacy locus.
\end{remark}

It is not hard to see that the existence of a Lagrangian covering family is invariant under birational isomorphisms.

\begin{lemma}\label{lem-covering-invariant}
Let $M$ be a projective hyperk\"ahler manifold admitting a Lagrangian covering family. If $M'$ is a projective hyperk\"ahler manifold birational to $M$, then $M'$ also admits a Lagrangian covering family.
\end{lemma}

\begin{proof}
Let 
\[\begin{tikzcd}
	L & M \\
	B
	\arrow["p"', from=1-1, to=2-1]
	\arrow["q", from=1-1, to=1-2]
\end{tikzcd}\]
be the corresponding family as in Definition \ref{def-lag-family} and $M\dashrightarrow M'$ be a birational map. Let $\wt{q}\colon \wt{L}\to M$ be a resolution of the indeterminacy locus of $L\dashrightarrow M'$ with $\wt{L}$ smooth. Then we can take an open subscheme $B'\subset B$ such that the restriction of $$\wt{p}\colon \wt{L}\to L\to B$$ over $B'$ is flat. Thus we define $L':=\wt{p}^{-1}(B')$,  $p'=\wt{p}|_{\wt{p}^{-1}(B')}\colon L'\to B'$, and $q':=\wt{q}|_{L'}\colon L'\to M'$. Then it is clear from the construction that
\[\begin{tikzcd}
	L' & M' \\
	B'
	\arrow["p'"', from=1-1, to=2-1]
	\arrow["q'", from=1-1, to=1-2]
\end{tikzcd}\]
is a Lagrangian covering family of $M'$.
\end{proof}

Therefore, to construct a Lagrangian covering family of $\widetilde{C}_X$, we only need to focus on $M_{\sigma_X}^X(1,-1)$. A Lagrangian family of $M_{\sigma_X}^X(1,-1)$ is constructed in \cite[Theorem 5.8]{FGLZ24} from the Bridgeland moduli spaces on hyperplane sections of $X$. In the following, we will show that it is actually a Lagrangian covering family using Theorem \ref{thm-pr-induce-map}. We first state two preliminary results.

\begin{proposition}\label{prop-lag-cover-family-cube}
Let $X$ be a general GM fourfold. Then for a generic $\sigma_X\in \Stab^{\circ}(\Ku(X))$, a general point of $M_{\sigma_X}^X(1, -1)$ is of the form $[\pr_X(I_C(H))]$ for a $\tau$-cubic $C\subset Y \subset X$ where $Y$ is a general smooth hyperplane section. 
\end{proposition}

\begin{proof}
By Theorem \ref{thm-pr-induce-map}, a general point of $M_{\sigma_X}^X(1,-1)$ is in the form $[\pr_X(I_C(H)]$ for a $\tau$-cubic $C\subset X$. According to Theorem \ref{thm-pr-induce-map}, the fibers of $pr$ over $pr(\mathrm{H}^{\tau})$ are $\PP^1$. Moreover, $M_{\sigma_X}^X(1,-1)$ is irreducible of dimension $6$. Thus, it suffices to find a sublocus $W\subset \mathrm{H}^{\tau}$ of dimension $\dim W\geq 7$ such that every $[C]\in W$ is contained in a smooth hyperplane section of $X$.

To this end, let $\cY\subset |\oh_X(H)|\times X$ be the universal hyperplane section. We denote by $\Hilb_{\cY/|\oh_X(H)|}^{3t+1}$ the relative Hilbert scheme of twisted cubics of $\cY$ over $|\oh_X(H)|$. By Corollary \ref{cor-hilb-smooth}, there is an open subscheme $S\subset |\oh_X(H)|$ parameterizing general smooth hyperplane sections such that the fiber of the natural map $p\colon \Hilb_{\cY/|\oh_X(H)|}^{3t+1}\to |\oh_X(H)|$ over $S$ is smooth irreducible of dimension $3$. As $p$ is projective, $p^{-1}(S)$ is smooth irreducible of dimension $11$. Now we consider another natural map $$q\colon \Hilb_{\cY/|\oh_X(H)|}^{3t+1}\to \Hilb^{3t+1}_X.$$ It is clear that for any point $[C]\in \Hilb^{3t+1}_X$,
$$q^{-1}([C])\cong \PP(H^0(I_C(H)))\cong \PP^4.$$  Then $q(p^{-1}(S))$ is irreducible of dimension $\geq 7$. We define $W:=\mathrm{H}^{\tau}\cap q(p^{-1}(S))$. As $\dim \Hilb^{3t+1}_X\setminus \mathrm{H}^{\tau}<7$ by Lemma \ref{lem-local-dim-sigma} and \ref{lem-local-dim-rho}, we have $W\neq \varnothing$. Since $\mathrm{H}^{\tau}$ is dense in $\Hilb^{3t+1}_X$ by Proposition \ref{prop-closure-tau}, we obtain $$\dim W=\dim q(p^{-1}(S))\geq 7$$ and the result follows.
\end{proof}

\begin{lemma}\label{lem-cube-injective}
Let $X$ be a general GM fourfold and $j\colon Y\hookrightarrow X$ be a smooth hyperplane section. Then for any two twisted cubics $C, C'\subset Y$ such that $\pr_Y(I_{C/Y}(H))\neq \pr_Y(I_{C'/Y}(H))$, we have
\[\pr_X(I_C(H))\neq \pr_X(I_{C'}(H)).\]
\end{lemma}

\begin{proof}
Assume otherwise that $\pr_X(I_C(H))= \pr_X(I_{C'}(H))$. First, we assume that $C$ and $C'$ are both $\rho$-cubics. Then by Proposition \ref{prop-proj-rho}, $C$ and $C'$ have the same residue line, which gives $$\pr_Y(I_{C/Y}(H))= \pr_Y(I_{C'/Y}(H))$$ by Lemma \ref{lem-cubic-mutation}(2) and makes a contradiction.

Now we assume that $C$ and $C'$ are both $\tau$-cubics. By \cite[Lemma 5.2(b)]{FGLZ24}, we have $$T_Y(\pr_Y(I_{C/Y}(H)))=\pr_Y(I_{C'/Y}(H)).$$ Then from \cite[Lemma 4.4(1)]{FGLZ24} and Proposition \ref{prop-pushforward}, we obtain $$T_X(\pr_X(I_C(H)))=\pr_X(I_C(H)).$$ In other words, $C$ is a residue cubic of $C$ itself by Lemma \ref{lem-criterion-residue-cubic}. Similarly, as $\pr_X(I_C(H))= \pr_X(I_{C'}(H))$, we know that $$T_X(\pr_X(I_C(H)))=\pr_X(I_{C'}(H)),$$ which implies that $C'$ is a also residue cubic of $C$. However, as $C,C'\subset Y$ and they are both residue cubics of $C$, this implies that $C=C'$ by Proposition \ref{prop-residue-tau-cubic} and makes a contradiction.
\end{proof}

Let $\cY\subset |\oh_X(H)|\times X$ be the universal hyperplane section. We denote by $\cY_V$ the restriction of $\cY$ to an open subscheme $V\subset |\oh_X(H)|$.

\begin{theorem}\label{thm-covering-moduli}
Let $X$ be a general GM fourfold. Then for a generic $\sigma_X\in \Stab^{\circ}(\Ku(X))$, there is an open subscheme $V\subset |\oh_X(H)|$ and a dominant morphism $q\colon \Hilb^{3t+1}_{\cY_V/V}\to M_{\sigma_X}^X(1,-1)$ such that 
% https://q.uiver.app/#q=WzAsMyxbMCwwLCJcXEhpbGJeezN0KzF9X3tcXGNIX1YvVn0iXSxbMCwxLCJWXFxzdWJzZXQgXFxQUF44Il0sWzEsMCwiXFx3dHtDfV9YIl0sWzAsMSwicCIsMl0sWzAsMiwicSJdXQ==
\[\begin{tikzcd}
	{\Hilb^{3t+1}_{\cY_V/V}} & {M_{\sigma_X}^X(1,-1)} \\
	{V\subset \PP^8}
	\arrow["p"', from=1-1, to=2-1]
	\arrow["q", from=1-1, to=1-2]
\end{tikzcd}\]
is a Lagrangian covering family with $p$ smooth and projective and $q|_{p^{-1}(s)}$ is birational onto its image for each $s\in V$. Here $p\colon \Hilb^{3t+1}_{\cY_V/V}\to V$ is the structure map of the relative Hilbert scheme over $V$.
\end{theorem}

\begin{proof}
By Corollary \ref{cor-hilb-smooth}, we take an open dense subset $V\subset |\oh_X(H)|$ such that each fiber of $$p\colon \Hilb^{3t+1}_{\cY_V/V}\to V$$ is smooth irreducible of dimension $3$. Hence, $\Hilb^{3t+1}_{\cY_V/V}$ is smooth irreducible of dimension $11$ and $p$ is smooth and projective.

Note that a smooth GM threefold $Y$ does not contain any $\sigma$-cubic. Otherwise, a cubic is contained in the zero locus of a section of $\cQ_Y$, which is contained in a conic and we get a contradiction. Hence, the natural map  $\Hilb^{3t+1}_{\cY_V/V}\to \Hilb_X^{3t+1}$ can be factored as $$\Hilb^{3t+1}_{\cY_V/V}\xra{h} \Hilb_X^{3t+1}\setminus \mathrm{H}^{\sigma}\hookrightarrow \Hilb_X^{3t+1}.$$ We define $q$ as the composition
\[q\colon \Hilb^{3t+1}_{\cY_V/V}\xra{h} \Hilb_X^{3t+1}\setminus \mathrm{H}^{\sigma} \xra{pr} M_{\sigma_X}^X(1,-1).\]
By Proposition \ref{prop-lag-cover-family-cube}, $q$ is dominant. And from the construction, we have $p^{-1}(s)\cong \Hilb^{3t+1}_Y$ for any $s=[Y]\in V$ and $q|_{p^{-1}(s)}$ can be factored as
\[q|_{p^{-1}(s)}\colon \Hilb^{3t+1}_Y\xra{p_Y} M^Y_{\sigma_Y}(1,-1)\hookrightarrow M^X_{\sigma_X}(1,-1),\]
where $p_Y$ is defined in Theorem \ref{thm-cubic-3fold}. Hence, $q|_{p^{-1}(s)}$ is birational onto its image by Lemma \ref{lem-cube-injective} and Theorem \ref{thm-cubic-3fold}, which is a Lagrangian submanifold of $M_{\sigma_X}^X(1,-1)$, and the result follows.
\end{proof}

Finally, the existence of a Lagrangian covering family on a general double EPW cube is a direct corollary of the results above.

\begin{theorem}\label{thm-second-lag-cover-family}
Let $X$ be a general GM fourfold. Then the associated double EPW cube $\wt{C}_X$ admits a Lagrangian covering family.
\end{theorem}

\begin{proof}
As we have already mentioned above Theorem \ref{cor-cube-as-MRC}, we have a birational isomorphism between $\wt{C}_X$ and $M^X_{\sigma_X}(1,-1)$. Then the statement follows from Lemma \ref{lem-covering-invariant} and Theorem \ref{thm-covering-moduli}.
\end{proof}

\begin{appendix}

\section{Recovering classical examples}\label{appendix-A}

In the classical literature, there are several hyperk\"ahler manifolds that can be constructed from cubic or GM fourfolds: Fano varieties of lines \cite{beauville:fano-variety-cubic-4fold}, LLSvS eightfolds \cite{LLSvS17}, double (dual) EPW sextics \cite{o2006irreducible}, and double EPW cubes \cite{IKKR19}. The construction of Lagrangian covering families of double EPW cubes is done in Section \ref{subsec-covering-cube}. In this appendix, we aim to construct Lagrangian covering families for the remaining examples via our categorical method and recover known classical constructions.

Recall that for a cubic fourfold $X$, its semi-orthogonal decomposition is given by $$\D^b(X)=\langle\Ku(X),\oh_X,\oh_X(H),\oh_X(2H)\rangle=\langle\oh_X(-H),\Ku(X),\oh_X,\oh_X(H)\rangle,$$
where $H$ is the hyperplane class of $X$, satisfying $S_{\Ku(X)}=[2]$. We define the projection functor  $$\pr_X:=\bR_{\oh_X(-H)}\bL_{\oh_X}\bL_{\oh_X(H)}.$$
There is a rank two lattice in the numerical Grothendieck group $\Knum(\Ku(X))$ generated by $\Lambda_1$ and $\Lambda_2$ with $$\ch(\Lambda_1)=3-H-\frac{1}{2}H^2+\frac{1}{6}H^3+\frac{1}{8}H^4,\quad \ch(\Lambda_2)=-3+2H-\frac{1}{3}H^3,$$ 
over which the Euler pairing is of the form
\begin{equation}
\left[               
\begin{array}{cc}   
-2 & 1 \\  
1 & -2\\
\end{array}
\right].
\end{equation}
We denote by $\Stab^{\circ}(\Ku(X))$ the family of stability conditions constructed in \cite{bayer2017stability}.

For a cubic threefold $Y$, we have a semi-orthogonal decomposition
\[\D^b(Y)=\langle \Ku(Y), \oh_Y, \oh_Y(H)\rangle.\]
In this case, $S_{\Ku(Y)}=\bL_{\oh_Y}\circ (-\otimes \oh_Y(H))[1]$ and $S^3_{\Ku(Y)}=[5]$. Moreover, $\Knum(\Ku(Y))$ is a rank two lattice generated by $\lambda_1$ and $\lambda_2$ with 
$$\ch(\lambda_1)=2-H-\frac{1}{6}H^2+\frac{1}{6}H^3 ,\quad \ch(\lambda_2)=-1+H-\frac{1}{6}H^2-\frac{1}{6}H^3$$
and the Euler pairing is 
\begin{equation}
\left[               
\begin{array}{cc}   
-1 & 1 \\  
0 & -1\\
\end{array}
\right].
\end{equation}
We define the projection functor by $\pr_Y:=\bL_{\oh_Y(H)}\bL_{\oh_Y}$. A family of Serre-invariant stability conditions on $\Ku(Y)$ is constructed in \cite{bayer2017stability}.

We denote by $M^X_{\sigma_X}(a,b)$ (resp.~$M^Y_{\sigma_Y}(a,b)$) the moduli space that parameterizes S-equivalence classes of $\sigma_X$-semistable (resp.~$\sigma_Y$-semistable) objects of class $a\Lambda_1 +b\Lambda_2$ (resp.~$a\lambda_1 +b\lambda_2$) in $\Ku(X)$ (resp.~$\Ku(Y)$). Then by \cite{BLMNPS21}, for any pair of coprime integers $a,b$ and generic $\sigma_X\in \Stab^{\circ}(\Ku(X))$, the moduli space $M^X_{\sigma_X}(a,b)$ is a projective hyperk\"ahler manifold.

%Similarly, Let $\sigma_Y$ be a Serre-invariant stability condition on $\Ku(Y)$. We denote by $M^Y_{\sigma_Y}(a,b)$ the moduli space of S-equivalence classes of $\sigma_Y$-semistable objects in $\Ku(Y)$ with class $a\lambda_1+b\lambda_2$.

\subsection{Fano variety of lines}\label{subsec-appendix-line}
Let $X$ be a cubic fourfold and $\sigma_X\in \Stab^{\circ}(\Ku(X))$. By \cite[Theorem 1.1]{li2018twisted}, the Fano variety of lines $F(X)$ is isomorphic to the moduli space $M_{\sigma_X}^X(1,1)$, given by the projection functor $\pr_X$
\[F(X)\xra{\cong} M_{\sigma_X}^X(1,1), \quad [L]\mapsto [\pr_X(I_{L})].\]
On the other hand, for any smooth cubic threefold $Y$ and Serre-invariant stability condition $\sigma_Y$ on $\Ku(Y)$, we also have an isomorphism
\[F(Y)\xra{\cong} M_{\sigma_Y}^Y(1,1), \quad [L]\mapsto [I_{L/Y}]\]
by \cite[Theorem 1.1]{PY20} and \cite{bernardara2012categorical}.

Let $[I_{L/Y}]\in M_{\sigma_Y}^Y(1,1)$. According to \cite[Proposition 4.5(3)]{li2020elliptic}, we have $\mathrm{pr}_X(j_*I_{L/Y})\cong P_L$, where
$$P_L:=\mathrm{cone}(I_L[-1]\xrightarrow{ev}\oh_X(-H)[1]).$$
Moreover, $P_L$ is $\sigma_X$-stable by \cite[Theorem 1.1]{li2018twisted}. Using \cite[Theorem A.4]{FGLZ24}, we have the following result.

\begin{theorem}
\label{Fano_variety_lines}
Let $X$ be a cubic fourfold and $j \colon Y \hookrightarrow X$ be a smooth hyperplane section. Then the functor $\pr_X\circ j_*$ induces a Lagrangian embedding
\[f \colon M_{\sigma_Y}^Y(1,1)\hookrightarrow M_{\sigma_X}^X(1,1).\]
\end{theorem}

There is a natural embedding $i:F(Y)\hookrightarrow F(X)$, mapping $[L\subset Y]$ to $[L\subset X]$. From the construction above, we deduce that the embedding $f$ is compatible with $i$.

\begin{corollary}\label{commupatible_Fano_var_of_lines}
The embedding in Theorem~\ref{Fano_variety_lines} is compatible with the natural one, which means that we have a commutative diagram
% https://q.uiver.app/#q=WzAsOCxbMCwwLCJGKFkpIl0sWzAsMSwiRihYKSJdLFsxLDEsIlxcbWF0aGNhbHtNfV97XFxzaWdtYX1eWCgxLDEpIl0sWzEsMCwiXFxtYXRoY2Fse019X3tcXHRhdX1eWSgxLDEpIl0sWzMsMCwibFxcc3Vic2V0IFkiXSxbMywxLCJsXFxzdWJzZXQgWCJdLFs0LDEsIlBfbCJdLFs0LDAsIklfe2wvWX0iXSxbMywyLCJmIiwwLHsic3R5bGUiOnsidGFpbCI6eyJuYW1lIjoiaG9vayIsInNpZGUiOiJ0b3AifX19XSxbMCwxLCJpIiwyLHsic3R5bGUiOnsidGFpbCI6eyJuYW1lIjoiaG9vayIsInNpZGUiOiJ0b3AifX19XSxbMSwyLCJcXGNvbmciLDJdLFswLDMsIlxcY29uZyIsMl0sWzQsNSwial8qIiwyLHsic3R5bGUiOnsidGFpbCI6eyJuYW1lIjoibWFwcyB0byJ9fX1dLFs3LDYsIlxccHJfWFxcY2lyYyBqXyoiLDAseyJzdHlsZSI6eyJ0YWlsIjp7Im5hbWUiOiJtYXBzIHRvIn19fV0sWzQsNywiXFxwcl9ZIiwwLHsic3R5bGUiOnsidGFpbCI6eyJuYW1lIjoibWFwcyB0byJ9fX1dLFs1LDYsIlxccHJfWCIsMix7InN0eWxlIjp7InRhaWwiOnsibmFtZSI6Im1hcHMgdG8ifX19XV0=
\[\begin{tikzcd}
	{F(Y)} & {M_{\sigma_Y}^Y(1,1)} && {[L\subset Y]} & {[I_{L/Y}]} \\
	{F(X)} & {M_{\sigma_X}^X(1,1)} && {[L\subset X]} & {[P_L]}
	\arrow["f", hook, from=1-2, to=2-2]
	\arrow["i"', hook, from=1-1, to=2-1]
	\arrow["\cong"', from=2-1, to=2-2]
	\arrow["\cong"', from=1-1, to=1-2]
	\arrow["{j_*}"', maps to, from=1-4, to=2-4]
	\arrow["{\pr_X\circ j_*}", maps to, from=1-5, to=2-5]
	\arrow["{\pr_Y}", maps to, from=1-4, to=1-5]
	\arrow["{\pr_X}"', maps to, from=2-4, to=2-5]
\end{tikzcd}\]
\end{corollary}

\subsection{LLSvS eightfold}\label{subsec-appendix-llsvs}
Let $X$ be a cubic fourfold not containing a plane and $\sigma_X\in \Stab^{\circ}(\Ku(X))$. Let $M_3(X)$ be the irreducible component of the Hilbert scheme $\mathrm{Hilb}^{3t+1}_X$, which contains the smooth cubics on $X$. We refer to curves in $M_3(X)$ as  \emph{(generalized) twisted cubics} on $X$. By \cite[Theorem A]{LLSvS17}, $M_3(X)$ is a smooth and irreducible projective variety of dimension $10$. In \cite[Theorem B]{LLSvS17}, the authors construct a two-step contraction 
$$\alpha \colon  M_3(X)\xrightarrow{\alpha_1}Z'\xrightarrow{\alpha_2}Z,$$
where $\alpha_1 \colon M_3(X)\rightarrow Z'$ is a $\mathbb{P}^2$-bundle and $\alpha_2 \colon Z'\rightarrow Z$ is blowing up the image of the embedding $\mu \colon X\hookrightarrow Z$. Moreover, $Z$ is a $8$-dimensional hyperk\"ahler manifold. We usually call $Z$ the \emph{LLSvS eightfold} associated with $X$. 

The fibers of $\alpha_1$ and $\alpha_2$ are studied in details (cf.~\cite[Section 1]{AL17}). For $[C], [C']\in M_3(X)$, we have $[C']\in \alpha_1^{-1}(\alpha_1([C]))$ if and only if $S=\langle C \rangle \cap X=\langle C' \rangle \cap X$ and $I_{C/S}\cong I_{C'/S}$. Moreover, $\alpha_1$ contracts the image of the locus of twisted cubics with an embedded point. For $[C]\in M_3(X)$ with an embedded point $p\in X$, $\alpha([C])=\alpha([C'])$ if and only if $p$ is also an embedded point of $C'$.

On the other hand, we consider the subscheme of $M_3(X)$ parameterizing twisted cubics contained in a smooth hyperplane section $Y$. We denote the image of $M_3(Y)$ under $\alpha$ by $Z_Y$, then $Z_Y$ is a Lagrangian subvariety of $Z$ due to \cite[Proposition 2.9]{shinder2017geometry}. 

The LLSvS eightfold $Z$ is reconstructed as the Bridgeland moduli space $M_{\sigma_X}^X(2,1)$ in \cite{li2018twisted}. 
We consider the moduli space $M_{\sigma_Y}^Y(2,1)$, which is a smooth projective variety of dimension $4$. The geometry of this moduli space is intensively studied in \cite{bayer2020desingularization} and \cite{altavilla2019moduli}. First of all, we identify the moduli space $M_{\sigma_Y}^Y(2,1)$ with $Z_Y$.

Recall that there are two types of twisted cubics in $X$: arithmetically Cohen--Macaulay (aCM) or non-Cohen--Macaulay (non-CM). We refer to \cite[Section 1]{LLMS18} for more details.

\begin{lemma}
\label{projection_object_cubics}
Let $Y$ be a smooth cubic threefold and $[C]\in M_3(Y)$. We define the complex $$E_C:=\pr_Y(I_{C/Y}(2H))[-1].$$
\begin{enumerate}
    \item If $C$ is aCM, then $E_C\cong\ker(\oh_Y^{\oplus 3}\xra{ev} \oh_S(D))$ is a slope-stable bundle such that $S:=\langle C \rangle\cap Y$ is the cubic surface containing $C$ and $I_{C/S}(2H)\cong\oh_S(D)$, where $D$ is a Weil divisor of $S$. 
    \item If $C$ is non-CM, then 
    $E_C\cong\pr_Y(\CC_p)[-2]\cong \ker(\oh_Y^{\oplus 4}\xra{ev} I_{p/Y}(H))$ is a slope-stable reflexive sheaf, 
    where $p$ is the embedded point of $C$.
\end{enumerate}
\end{lemma}

\begin{proof}\leavevmode
By \cite[(1.2.2)]{LLMS18}, we see $\bL_{\oh_Y(H)}I_{C/Y}(2H)$ sits in the triangle
$$\oh_Y(H)\rightarrow I_{C/Y}(2H)\rightarrow\bL_{\oh_Y(H)}I_{C/Y}(2H)\cong I_{C/S}(2H),$$
where $S:=\langle C \rangle\cap Y$ is the cubic surface containing $C$.
Thus $E_C\cong\bL_{\oh_Y}(I_{C/S}(2H))[-1]$. Note that $\RHom_Y(\oh_Y, I_{C/S}(2H))=\CC^3$ by \cite[(1.2.2)]{LLMS18}, then we have a triangle
$$E_C\to \oh_Y^{\oplus 3}\xra{ev} I_{C/S}(2H),$$
where $ev$ is the evaluation map. Using \cite[(1.2.2)]{LLMS18}, it is straightforward to check that $I_{C/S}(2H)$ is $0$-regular in the sense of Castelnuovo--Mumford. In particular, $ev$ is surjective. Thus we obtain that $E_C\cong\ker(\oh_Y^{\oplus 3}\xra{ev} I_{C/S}(2H))$.

(1): If $C$ is aCM, by \cite[Proposition 3.1]{bayer2020desingularization}, we know that $S$ is normal and integral. By \cite[Proposition 3.2]{bayer2020desingularization}, if we set $D:=2H-C$, we have $I_{C/S}(2H)\cong \oh_S(D)$. The stability and locally freeness follow from \cite[Lemma 5.1]{bayer2020desingularization}.

%Hence $E_C$ is isomorphic to the kernel of the evaluation map 
%$$\mathrm{RHom}(\oh_Y,I_{C/S}(2H))\otimes\oh_Y\xrightarrow{ev} I_{C/S}(2H).$$
(2): If $C$ is non-CM, there are two short exact sequences
\begin{equation} \label{seq_kp}
    0\rightarrow I_{C/Y}(2H)\rightarrow I_{C_0/Y}(2H)\rightarrow \CC_p\rightarrow 0
\end{equation}
and 
\begin{equation}\label{SEQ_C0}
    0\rightarrow\oh_Y\rightarrow\oh_Y^{\oplus 2}(H)\rightarrow I_{C_0/Y}(2H)\rightarrow 0,
\end{equation}
where $C_0$ is a plane cubic curve and $p$ is the embedded point. Furthermore,   
\eqref{SEQ_C0} is the Koszul resolution of $I_{C_0/Y}$. Applying $\bL_{\oh_Y(H)}$ to (\ref{SEQ_C0}), we get $$\bL_{\oh_Y(H)}(I_{C_0/Y}(2H))\cong \bL_{\oh_Y(H)}\oh_Y[1]\cong\oh_Y[1].$$
Then applying $\pr_Y$ to \eqref{seq_kp}, we obtain that $\mathrm{pr}_Y(I_{C_0/Y})\cong 0$.  Thus, from (\ref{seq_kp}), we have $$\mathrm{pr}_Y(\CC_p)[-1]\cong\mathrm{pr}_Y(I_{C/Y}(2H)).$$ As $\bL_{\oh_Y(H)}\CC_p[-1]\cong I_{p/Y}(H)$, $E_C$ is the kernel of the evaluation map
\[E_C\to \oh_Y^{\oplus 4}\xra{ev} I_{p/Y}(H).\]
Finally, the stability and reflexivity of $E_C$ follow from \cite[Lemma 5.1]{bayer2020desingularization}.
\end{proof}
%By definition, $\bR_{\oh_Y(-H)}I_{p/Y}$ sits in the exact triangle
%$$\bR_{\oh_Y(-H)}I_{p/Y}\rightarrow I_{p/Y}\rightarrow\mathrm{RHom}(I_{p/Y},\oh_Y(-H))^{\vee}\otimes\oh_Y(-H).$$
%Applying the rotation functor $\bR$ to this triangle, we get
%$$\mathrm{pr}_Y(\CC_p)[-1]\rightarrow\bL_{\oh_Y}(I_{p/Y}(H))\rightarrow\bL_{\oh_Y}(\mathrm{RHom}(I_{p/Y},\oh_Y(-H))^{\vee}\otimes\oh_Y)\cong 0.$$
%Thus $\mathrm{pr}_Y(\CC_p)[-1]\cong\bL_{\oh_Y}(I_{p/Y}(H))$ and $E_C\cong\mathrm{pr}_Y(\CC_p)[-2]\cong\bL_{\oh_Y}(I_{p/Y}(H))[-1]$, which means that $E_C$ is the kernel of the evaluation map
%$E_C\to \oh_Y^{\oplus 4}\xra{ev} I_{p/Y}(H).$

By Lemma \ref{projection_object_cubics} and \cite[Theorem 6.1(ii), Theorem 8.7]{bayer2020desingularization},  $\mathrm{pr}_Y(I_{C/Y}(2H))$ is $\sigma_Y$-stable for every Serre-invariant stability condition $\sigma_Y$ on $\Ku(Y)$. Then the functor $\pr_Y$ induces a morphism $$\pi \colon M_3(Y)\rightarrow M_{\sigma_Y}^Y(2,1).$$ 
Moreover, $\pi$ is surjective by \cite[Theorem 6.1(ii)]{bayer2020desingularization}, and $M_{\sigma_Y}^Y(2,1)$ is smooth and projective by \cite[Theorem 7.1]{bayer2020desingularization}. 

Since the subvariety $Z_Y\subset Z$ is a two-step contraction of  $\alpha|_{M_3(Y)}$, to show $M_{\sigma_Y}^Y(2,1)\cong Z_Y$, it suffices to prove that $\pi$ contracts the same locus as $\alpha|_{M_3(Y)}$, which is equivalent to the following theorem.

\begin{theorem}
\label{SS_variety_equal_BBF_moduli_space}
For twisted cubics $C$ and $C'$ on $Y$, $\mathrm{pr}_Y(I_{C/Y}(2H))\cong\mathrm{pr}_Y(I_{C'/Y}(2H))$ if and only if $\alpha(C)=\alpha(C')$, where $\alpha 
\colon M_3(Y)\rightarrow Z_Y$. Thus, we have $Z_Y\cong M_{\sigma_Y}^Y(2,1)$. 
\end{theorem}

\begin{proof}
The argument is very similar to the proof of \cite[Proposition 2]{AL17}, but the situation here is simpler. 
If $\alpha(C)=\alpha(C')=p\in\mu(Y)$, from the construction, $C$ and $C'$ are both non-CM twisted cubics with the embedded point $p$. Thus, by Lemma~\ref{projection_object_cubics}(2), we have $\mathrm{pr}_Y(I_{C/Y}(2H))\cong\mathrm{pr}_Y(I_{C'/Y}(2H))$. If $\alpha(C)=\alpha(C')\notin\mu(Y)$, then $C$ and $C'$ are both aCM twisted cubics and they are in the same fiber of the $\mathbb{P}^2$-bundle map $\alpha_1$. This implies that they are in the same linear system, i.e.~$I_{C/S}\cong I_{C'/S}$. Then $\mathrm{pr}_Y(I_{C/Y}(2H))\cong\mathrm{pr}_Y(I_{C'/Y}(2H))$.

Conversely, we show that if  $\pr_Y(I_{C/Y}(2H))\cong\pr_Y(I_{C'/Y}(2H))$, then $\alpha(C)=\alpha(C')$.
\begin{enumerate}
    \item If $C$ and $C'$ are both aCM, we need to show that $C$ and $C'$ are contained in the same cubic surface and in the same linear system. Let $S:=\langle C \rangle \cap Y$ and $S':=\langle C' \rangle \cap Y$. By assumption, we see that $E_C\cong E_{C'}$. Using the same method as in  \cite[Proposition 8.1]{GLZ2021conics}, we know that $\Hom_Y(I_{C/S}, I_{C'/S'})\neq 0$. Then $I_{C/S}\cong I_{C'/S'}$ since they are Gieseker-stable. Hence, we get $\alpha(C)=\alpha(C')$.
    
    \item If $C$ and $C'$ are not aCM, with the embedded points $p$ and $p'$, due to $$\pr_Y(I_{C/Y}(2H))\cong\pr_Y(I_{C'/Y}(2H)),$$ we obtain $E_C\cong E_{C'}$. Since $E_C$ is only non-locally free at $p$ and $E_{C'}$ is only non-locally free at $p'$, we get $p=p'$. Hence, $\alpha(C)=\alpha(C')$. 
    
    \item If $C$ is aCM and $C'$ is not aCM, we aim to prove that their projection objects in $\Ku(Y)$ can not be isomorphic. This is obvious. By Lemma~\ref{projection_object_cubics}, $E_C$ is locally free, while $E_{C'}$ is non-locally free at the embedded point of $C'$. Hence, they can not be isomorphic.\qedhere 
\end{enumerate}
\end{proof}

%\begin{corollary}
%\label{subvariety}
%The moduli space $\mathcal{M}_{\tau}(\Ku(Y),2\lambda_1+\lambda_2)$ is a subvariety of $\mathcal{M}_{\sigma}(\Ku(X),2\Lambda_1+\Lambda_2)$.
%\end{corollary}
Next, we show that the functor $\pr_X\circ j_*\colon \Ku(Y)\rightarrow\Ku(X)$ induces a Lagrangian embedding  $$p_j\colon M_{\sigma_Y}^Y(2,1)\to M_{\sigma_X}^X(2,1).$$

For a twisted cubic curve $C\subset X$ contained in a cubic surface $S\subset X$, let $F_C$ be the kernel of the evaluation map
\[ev\colon H^0(X, I_{C/S}(2H))\otimes \oh_X\twoheadrightarrow I_{C/S}(2H)\]
and $F'_C:=\pr_X(F_C)\cong\bR_{\oh_X(-H)}F_C\in\Ku(X)$ be the projection object of $F_C$. By definition, we see 
\[F'_C\cong\pr_X(I_{C/S}(2H))[-1].\]

\begin{proposition} \label{push-cubic}
Let $X$ be a cubic fourfold and $j\colon  Y\hookrightarrow X$ be a smooth hyperplane section. If $C$ is a twisted cubic on $Y$, then we have $\pr_X(j_*E_C)\cong F'_C.$
\end{proposition}

\begin{proof}
By definition, we have $E_C=\pr_Y(I_{C/Y}(2H))[-1]$. Applying $\pr_Y$ to the exact sequence
\[0\to \oh_Y(H)\to I_{C/Y}(2H)\to I_{C/S}(2H)\to 0,\]
we see $E_C\cong \pr_Y(I_{C/S}(2H))[-1]$. Since $F'_C=\pr_X(I_{C/S}(2H))[-1]$, the result follows from Proposition \ref{prop-pushforward}.
\end{proof}

Combined with Proposition \ref{push-cubic}, \cite[Theorem 1.2]{li2018twisted} and \cite[Theorem A.4]{FGLZ24}, we obtain the following result.

\begin{theorem} \label{8fold-embed}
Let $X$ be a cubic fourfold not containing a plane and  $j\colon  Y\hookrightarrow X$ be a smooth hyperplane section. Then the functor $\pr_X\circ j_*$ induces a Lagrangian embedding 
\[p_j\colon M_{\sigma_Y}^Y(2,1)\hookrightarrow M_{\sigma_X}^X(2,1).\]
\end{theorem}

In \cite{shinder2017geometry}, the authors show that the natural embedding $i'\colon Z_Y\hookrightarrow Z$ realizes $Z_Y$ as a Lagrangian subvariety of $Z$. Moreover, we have a commutative diagram
% https://q.uiver.app/#q=WzAsNCxbMCwwLCJNKFkpIl0sWzAsMSwiTShYKSJdLFsxLDEsIloiXSxbMSwwLCJaX1kiXSxbMywyLCJpJyIsMCx7InN0eWxlIjp7InRhaWwiOnsibmFtZSI6Imhvb2siLCJzaWRlIjoidG9wIn19fV0sWzEsMiwiXFxhbHBoYSIsMl0sWzAsM10sWzAsMSwiIiwyLHsic3R5bGUiOnsidGFpbCI6eyJuYW1lIjoiaG9vayIsInNpZGUiOiJ0b3AifX19XV0=
\[\begin{tikzcd}
	{M_3(Y)} & {Z_Y} \\
	{M_3(X)} & Z
	\arrow["{i'}", hook, from=1-2, to=2-2]
	\arrow["\alpha"', from=2-1, to=2-2]
	\arrow[from=1-1, to=1-2]
	\arrow[hook, from=1-1, to=2-1]
\end{tikzcd}\]
where $M_3(Y)\hookrightarrow M_3(X)$ is the natural inclusion. Actually, the result of \cite{shinder2017geometry} can also be deduced via Bridgeland moduli spaces. In fact, by Theorem~\ref{8fold-embed}, after the identification $M_{\sigma_Y}^Y(2,1)\cong Z_Y$ and $M_{\sigma_X}^X(2,1)\cong Z$, we have an embedding $Z_Y\hookrightarrow Z$, which realizes $Z_Y$ as a Lagrangian subvariety of $Z$. Moreover, we can show that the two embeddings $p_j$ and $i'$ are compatible.

\begin{corollary}\label{commupatible_LLSvS}
The embedding in Theorem \ref{8fold-embed} is compatible with the one in \cite{shinder2017geometry}, which means that we have a commutative diagram
% https://q.uiver.app/#q=WzAsOCxbMCwwLCJaX1kiXSxbMCwxLCJaX1giXSxbMSwxLCJcXG1hdGhjYWx7TX1fe1xcc2lnbWF9XlgoMiwxKSJdLFsxLDAsIlxcbWF0aGNhbHtNfV97XFx0YXV9XlkoMiwxKSJdLFszLDAsIlxcYWxwaGEoQykiXSxbMywxLCJcXGFscGhhKEMpIl0sWzQsMSwiRidfQyJdLFs0LDAsIlxccHJfWShJX3tDL1l9KDJIKSkiXSxbMywyLCJwX2oiLDAseyJzdHlsZSI6eyJ0YWlsIjp7Im5hbWUiOiJob29rIiwic2lkZSI6InRvcCJ9fX1dLFswLDEsImknIiwyLHsic3R5bGUiOnsidGFpbCI6eyJuYW1lIjoiaG9vayIsInNpZGUiOiJ0b3AifX19XSxbMSwyLCJcXGNvbmciLDJdLFswLDMsIlxcY29uZyIsMl0sWzQsNSwiIiwwLHsic3R5bGUiOnsidGFpbCI6eyJuYW1lIjoibWFwcyB0byJ9fX1dLFs1LDYsIiIsMCx7InN0eWxlIjp7InRhaWwiOnsibmFtZSI6Im1hcHMgdG8ifX19XSxbNCw3LCIiLDIseyJzdHlsZSI6eyJ0YWlsIjp7Im5hbWUiOiJtYXBzIHRvIn19fV0sWzcsNiwiIiwyLHsic3R5bGUiOnsidGFpbCI6eyJuYW1lIjoibWFwcyB0byJ9fX1dXQ==
\[\begin{tikzcd}
	{Z_Y} & {M_{\sigma_Y}^Y(2,1)} && {\alpha(C)} & {[\pr_Y(I_{C/S}(2H))=E_C[1]]} \\
	{Z} & {M_{\sigma_X}^X(2,1)} && {\alpha(C)} & {[\pr_X(I_{C/S}(2H))=F'_C[1]]}
	\arrow["{p_j}", hook, from=1-2, to=2-2]
	\arrow["{i'}"', hook, from=1-1, to=2-1]
	\arrow["\cong"', from=2-1, to=2-2]
	\arrow["\cong"', from=1-1, to=1-2]
	\arrow[maps to, from=1-4, to=2-4]
	\arrow[maps to, from=2-4, to=2-5]
	\arrow[maps to, from=1-4, to=1-5]
	\arrow[maps to, from=1-5, to=2-5]
\end{tikzcd}\]
\end{corollary}

\subsection{Double dual EPW sextic}\label{subsec_double_dual_EPW_sextic}
Let $X$ be a smooth GM variety. By \cite[Theorem 3.10]{debarre2015gushel}, $X$ is determined by its Lagrangian data $(V_6, V_5, A(X))$, where $V_6$ is a $6$-dimensional vector space, $V_5\subset V_6$ is a hyperplane, and $A(X)\subset \wedge^3 V_6$ is a Lagrangian subspace (cf.~\cite[Definition 3.4]{debarre2015gushel}). Given a Lagrangian data, for each integer $k$, we can define subschemes 
\[Y^{\geq k}_{A(X)}:=\{[v]\in\mathbb{P}(V_6) \mid \mathrm{dim}(A(X)\cap(v\wedge\bigwedge^2V_6))\geq k\}\]
and
\[Y^{\geq k}_{A(X)^{\perp}}=\{V_5\in\mathrm{Gr}(5,V_6) \mid \mathrm{dim}(A(X)\cap\bigwedge^3V_5)\geq k\},\]
as in \cite[Section 2]{o2006irreducible}.

Furthermore, $Y^{\geq 1}_{A(X)}$ and $Y^{\geq 1}_{A(X)^{\perp}}$ admit natural double coverings $\tilde{Y}^{\geq 1}_{A(X)}$ and $\tilde{Y}^{\geq 1}_{A(X)^{\perp}}$, which are called \emph{double Eisenbud--Popsecu--Walter (EPW) sextic} and \emph{double dual EPW sextic} associated with $X$. By \cite{o2006dual}, when $X$ is general, $\widetilde{Y}_{A(X)}$ and $\widetilde{Y}_{A(X)^{\perp}}$ are $4$-dimensional hyperk\"ahler manifolds.

Similarly, $Y^{\geq 2}_{A(X)}$ and $Y^{\geq 2}_{A(X)^{\perp}}$ admit natural double coverings $\widetilde{Y}^{\geq 2}_{A(X)}$ and $\widetilde{Y}^{\geq 2}_{A(X)^{\perp}}$, which are called \emph{double EPW surface} and \emph{double dual EPW surface} associated with $X$. When $X$ is general, both surfaces are smooth.

We start with a result in \cite{JLLZ2021gushelmukai}.

\begin{theorem}[{\cite{JLLZ2021gushelmukai}}]\label{mod_215_epwsurface}
Let $Y$ be a general GM threefold and $\sigma_Y$ be a Serre-invariant stability condition on $\Ku(Y)$. Then we have $M_{\sigma_Y}^Y(1,0)\cong\widetilde{Y}^{\geq 2}_{A(Y)^{\perp}}$ and $M_{\sigma_Y}^Y(0,1)\cong \widetilde{Y}^{\geq 2}_{A(Y)}$.
\end{theorem}

\begin{proof}
From \cite{Log12} and \cite{Debarre2024quadrics}, we know that $\widetilde{Y}^{\geq 2}_{A(Y)^{\perp}}\cong \cC_m(Y)$, where $\cC_m(Y)$ is the minimal model of the Fano surface of conics on $Y$. At the same time, $\widetilde{Y}^{\geq 2}_{A(Y)}\cong \cC_m(Y_L)$, where $Y_L$ is a line transform (period dual) of $Y$. Then the result follows from \cite[Theorem 7.12, Corollary 9.5, Theorem 8.9]{JLLZ2021gushelmukai}.
\end{proof}

More precisely, let $\cC(Y)$ be the Fano surface of conics on $Y$. It is proved in \cite{JLLZ2021gushelmukai} that $\pr_Y$ induces a surjective morphism $\cC(Y)\twoheadrightarrow M_{\sigma_Y}^Y(1,0)$ such that it contracts the same locus as the morphism $\cC(Y)\to \widetilde{Y}^{\geq 2}_{A(Y)^{\perp}}$ (cf.~\cite{Log12}, \cite{Debarre2024quadrics}). Hence, we obtain an isomorphism $M_{\sigma_Y}^Y(1,0)\cong\widetilde{Y}^{\geq 2}_{A(Y)^{\perp}}$ such that the following diagram commutes:
% https://q.uiver.app/#q=WzAsMyxbMCwwLCJcXGNDKFkpIl0sWzAsMSwiXFx3aWRldGlsZGV7WX1ee1xcZ2VxIDJ9X3tBKFkpXntcXHBlcnB9fSJdLFsxLDEsIk1fe1xcc2lnbWFfM31eWSgxLDApIl0sWzAsMV0sWzEsMiwiXFxjb25nIiwyXSxbMCwyXV0=
\[\begin{tikzcd}
	{\cC(Y)} \\
	{\widetilde{Y}^{\geq 2}_{A(Y)^{\perp}}} & {M_{\sigma_Y}^Y(1,0)}
	\arrow[from=1-1, to=2-1]
	\arrow["\cong"', from=2-1, to=2-2]
	\arrow[from=1-1, to=2-2]
\end{tikzcd}\]
where the morphism $\cC(Y)\to M_{\sigma_Y}^Y(1,0)$ is given by $[C]\mapsto [\pr_Y(I_{C/Y})]$.

Analogously, for a very general GM fourfold $X$, it is proved in \cite[Theorem 1.1]{GLZ2021conics} that we have a commutative diagram
% https://q.uiver.app/#q=WzAsMyxbMCwwLCJcXGNDKFgpIl0sWzAsMSwiXFx3aWRldGlsZGV7WX1fe0EoWSlee1xccGVycH19Il0sWzEsMSwiTV97XFxzaWdtYV80fV5YKDEsMCkiXSxbMCwxXSxbMSwyLCJcXGNvbmciLDJdLFswLDJdXQ==
\[\begin{tikzcd}
	{\cC(X)} \\
	{\widetilde{Y}_{A(Y)^{\perp}}} & {M_{\sigma_X}^X(1,0)}
	\arrow[from=1-1, to=2-1]
	\arrow["\cong"', from=2-1, to=2-2]
	\arrow[from=1-1, to=2-2]
\end{tikzcd}\]
where $\cC(X)$ is the Hilbert scheme of conics on $X$ and the morphism $\cC(X)\to \widetilde{Y}_{A(Y)^{\perp}}$ is constructed in \cite{iliev2011fano}. At the same time, the morphism $\cC(X)\to M_{\sigma_X}^X(1,0)$ is given by $[C]\mapsto [\pr_X(I_{C})]$. Note that the computation of $\Ext^1_X(\pr_X(I_{C}), \pr_X(I_{C}))=\CC^4$ holds for general $X$, so the same argument as in the proof of Theorem \ref{thm-stability-proj-general} implies that the morphism $\cC(X)\to M_{\sigma_X}^X(1,0)$ and hence the above diagram can be also constructed for \emph{general} $X$.

Let $X$ be a general GM fourfold and $j\colon Y\hookrightarrow X$ be a general hyperplane section. According to \cite[Theorem 5.3]{FGLZ24}, we have an induced morphism $q_j\colon  M_{\sigma_Y}^Y(1,0)\to M_{\sigma_X}^X(1,0)$, whose image is Lagrangian. Moreover, by \cite[Lemma 6.1]{GLZ2021conics} and Proposition \ref{prop-pushforward}, for any conic $C\subset Y$, we have
\[\pr_X(j_*\pr_Y(I_{C/Y}))\cong\pr_X(I_{C}).\]
Putting together the above arguments yields the following result.

\begin{corollary}\label{compatible_epw}
Let $X$ be a general GM fourfold and $j\colon Y\hookrightarrow X$ be a general hyperplane section. Then the functor $\pr_X \circ j_*$ induces a morphism
\[q_j\colon  M_{\sigma_Y}^Y(1,0)\to M_{\sigma_X}^X(1,0),\]
which is compatible with the one in \cite[Section 5.1]{iliev2011fano}. This means that there is a commutative diagram
\[\begin{tikzcd}
	{\widetilde{Y}^{\geq 2}_{A(Y)^{\perp}}} & {M_{\sigma_Y}^Y(1,0)} \\
	{\widetilde{Y}_{A(X)^{\perp}}} & {M_{\sigma_X}^X(1,0)}
	\arrow["{q_j}", from=1-2, to=2-2]
	\arrow["{i''}"', from=1-1, to=2-1]
	\arrow["\cong"', from=2-1, to=2-2]
	\arrow["\cong"', from=1-1, to=1-2]
\end{tikzcd}\]
\end{corollary}

\subsection{Double EPW sextic}\label{subsec_double_EPW_sextic}
Let $X$ be a very general GM fourfold. By \cite[Theorem 1.1]{GLZ2021conics}, we have an isomorphism  $M_{\sigma_X}^X(0,1)\cong\widetilde{Y}_{A(X)}$, which is the double EPW sextic associated with $X$. On the other hand, for a general hyperplane section $j\colon Y\hookrightarrow X$, the moduli space $M^Y_{\sigma_Y}(0,1)$ is a smooth projective surface according to \cite[Theorem 1.3]{ppzEnriques2023}.

By virtue of Theorem \ref{mod_215_epwsurface}, we have an isomorphism $M^Y_{\sigma_Y}(0,1)\cong \widetilde{Y}^{\geq 2}_{A(Y)}$. Using \cite[Theorem 5.3]{FGLZ24}, one can  construct a finite unramified morphism
\[\widetilde{Y}^{\geq 2}_{A(Y)}\to \widetilde{Y}_{A(X)}\]
whose image is Lagrangian. We expect that this morphism coincides with the one in \cite{iliev2011fano}. This alignment will be attained if $M_{\sigma_X}^X(0,1)$ also has a description in terms of a moduli space of sheaves, as in \cite[Theorem 8.9]{JLLZ2021gushelmukai}. We will revisit this point in the future.

\section{Twisted cubics on Gushel--Mukai threefolds}\label{appendix-B}

In this section, we study twisted cubics on GM threefolds. These results are only needed when we construct the Lagrangian covering family for double EPW cubes in Section \ref{subsec-covering-cube}.

Recall that there are two classes of GM threefolds: ordinary and special. An ordinary GM threefold is given by $$Y=\Gr(2,5)\cap \PP(V_8)\cap Q\subset \PP^9,$$ where $Q$ is a quadric hypersurface and $\PP(V_8)\cong \PP^7$. A special GM threefold is a double cover $\pi\colon Y\to Y_5$, where $Y_5$ is a codimension $3$ linear section of $\Gr(2,5)$. In both cases, $Y$ is an intersection of quadrics in the projective space, see \cite[Theorem 1.2]{debarre2015gushel}.

As in Section \ref{sec-very-general-GM}, we have semi-orthogonal decompositions
\[\D^b(Y)=\langle \Ku(Y), \oh_Y, \cU_Y^{\vee}\rangle\]
and
\[\D^b(Y)=\langle \cA_Y, \cU_Y, \oh_Y\rangle.\]
Moreover, according to \cite[Lemma 3.7]{JLLZ2021gushelmukai}, there is an equivalence given by 
\[\Xi_Y:=\bL_{\oh_Y}\circ (-\otimes \oh_Y(H))\colon \cA_Y\to \Ku(Y).\]

\subsection{Basic properties}

First, we demonstrate some basic properties of twisted cubics. Let $Y$ be a GM threefold.

\begin{lemma} \label{lem-cubic-pure}
We have 

\begin{enumerate}
    \item $\Hilb_Y^{3t+m}=\varnothing$ for $m<1$, and

    \item $\oh_C$ is a pure sheaf for any $[C]\in \Hilb_Y^{3t+1}$.
\end{enumerate}

\end{lemma}

\begin{proof}
Since $Y$ does not contain any plane and is an intersection of quadrics, the result follows from the same argument as Lemma \ref{lem-pure-cubic}.
\end{proof}

\begin{lemma} \label{lem-cubic-rhom}
Let $[C]\in \Hilb_Y^{3t+1}$. Then $\RHom_Y(\oh_Y,\oh_C)=\CC$ and $\RHom_Y(\oh_Y, I_{C/Y})=0.$
\end{lemma}

\begin{proof}
Note that $\RHom_Y(\oh_Y,\oh_C)=\CC$ if and only if  $\RHom_Y(\oh_Y, I_{C/Y})=0.$
Moreover, since $\chi(\oh_C)=1$,
$\RHom_Y(\oh_Y,\oh_C)=\CC$ if and only if $H^1(\oh_C)=0$. Thus, it suffices to show $H^1(\oh_C)=0$. This follows from Lemma \ref{lem-cubic-pure}(2) and \cite[Corollary 1.38(3)]{sa14}.
\end{proof}

\begin{lemma} \label{lem-cubic-U}
For any $[C]\in \Hilb_Y^{3t+1}$, either $\RHom_Y(\cU_Y, I_{C/Y})=0$ or $\RHom_Y(\cU_Y, I_{C/Y})=\CC\oplus \CC[-1]$. 
\end{lemma}

\begin{proof}
Similar to the proof of Lemma \ref{lem-homo-cubic}, we know that $\Hom_Y(\cU_Y, I_{C/Y}[k])=0$ for $k\notin \{0,1\}$. Since $\chi(\cU_Y, I_{C/Y})=0$, we only need to compute $\Hom_Y(\cU_Y, I_{C/Y})$.

Assume that $\Hom_Y(\cU_Y, I_{C/Y})=\CC^n$ for $n\geq 0$. When $Y$ is ordinary, $C$ is contained in $\Gr(2,5-n)$. Then $n\leq 1$ because $C$ is not contained in any plane. When $Y$ is special, denote by $\cV$ is the restriction of the tautological sub-bundle of $\Gr(2,5)$ via $Y_5\hookrightarrow \Gr(2,5)$. If $n\geq 2$, then $C$ is contained in $\pi^{-1}(D')$, where $\pi\colon Y\to Y_5$ is the defining double cover and $D'$ is the zero locus of two linearly independent sections of $\cV$. Since $Y_5$ is a linear section of $\Gr(2,5)$ and does not contain any plane, $D'=\Gr(2,3)\cap Y_5$ is contained in a line. Then $\pi^{-1}(D')$ is of degree $2$, which is impossible. Thus, we obtain $n\leq 1$.

Hence, in both cases, we see $\Hom_Y(\cU_Y, I_{C/Y})=0$ or $\CC$ and the result follows.
\end{proof}

\begin{remark}
When $\Hom_Y(\cU_Y, I_{C/Y})=\CC$, $C$ is contained in the zero locus $D$ of a non-zero section of $\cU^{\vee}_Y$. Moreover, $D$ is a degree $4$ genus $1$ curve and the residue curve of $C$ in $D$ is a line. See also Lemma \ref{lem-rho-cubic}(3).
\end{remark}

\begin{lemma}\label{lem-cubic-UC}
Let $[C]\in \Hilb_Y^{3t+1}$. Then $\RHom_Y(I_{C/Y}, \cU_Y)=\CC[-1]$.
\end{lemma}

\begin{proof}
It suffices to show $H^0(\cU_Y|_C)=0$, which can be deduced from \cite[Lemma B.3.3]{KPS}. %When $X$ is ordinary, this can be deduced from \cite[Lemma B.3.3]{KPS}. When $X$ is special, this follows from \cite[Lemma 2.9]{zhang2020bridgeland}.
\end{proof}

\begin{lemma}\label{lem-ICH}
Let $[C]\in \Hilb_Y^{3t+1}$. Then $H^1(\oh_C(H))=0$ and $\RHom_Y(\oh_Y, I_{C/Y}(H))=\CC^4$.
\end{lemma}

\begin{proof}
The result follows from the same argument as Lemma \ref{lem-cubic-cohomology-1}.
\end{proof}

\subsection{Projection objects and stability}
In this subsection, we compute the projection objects and establish their stability. 

Recall that for any line $L\subset Y$, $\cU_Y^{\vee}|_L\cong \oh_L\oplus \oh_L(1)$ and $\cQ_Y|_L\cong \oh_L\oplus \oh_L\oplus \oh_L(1)$. Firstly, we verify the projection objects under $\pr_Y$ associated with twisted cubics.

\begin{lemma} \label{lem-cubic-mutation}
Let $[C]\in \Hilb_Y^{3t+1}$. 
\begin{enumerate}
    \item If $\RHom_Y(\cU_Y, I_{C/Y})=0$, then $\pr_Y(I_{C/Y}(H))=\bL_{\oh_Y}(I_{C/Y}(H))=\cone(\oh_Y^{\oplus 4}\to I_{C/Y}(H))$.

    \item If $\RHom_Y(\cU_Y, I_{C/Y})\neq 0$, $\pr_Y(I_{C/Y}(H))=\pr_Y(\oh_L(-H))$. We have a non-splitting exact triangle 
\begin{equation}\label{eq-pr-cubic}
    \oh_L(-H)\to \pr_Y(I_{C/Y}(H))\to \cQ^{\vee}_Y[1],
\end{equation}
corresponding to the unique non-zero element (up to scalar) of $\Hom_Y(\cQ^{\vee}_Y[1], \oh_L(-H)[1])=\CC$.
\end{enumerate}

\end{lemma}

\begin{proof}
By definition, we have $\pr_Y(I_{C/Y}(H))=\bL_{\oh_Y}\bL_{\cU^{\vee}_Y}(I_{C/Y}(H))$. If $\RHom_Y(\cU_Y, I_{C/Y})=0$, then we have $\pr_Y(I_{C/Y}(H))=\bL_{\oh_Y}(I_{C/Y}(H))$ and (1) follows from Lemma \ref{lem-ICH}.

Now we assume that $\RHom_Y(\cU_Y, I_{C/Y})\neq 0$. By Lemma \ref{lem-cubic-U}, we have $\RHom_Y(\cU_Y, I_{C/Y})=\CC\oplus \CC[-1]$. Let $s\colon \cU^{\vee}_Y\to I_{C/Y}(H)$ be the non-zero map. Then $\ker(s)=\oh_Y$ and $\im(s)=I_{D/Y}(H)$, where $D$ is the zero locus of a section of $\cU^{\vee}_Y$. Therefore, $\mathrm{coker}(s)=\oh_L(-H)$, where $L$ is the residue component of $C$ inside $D$. Hence, taking cohomology objects to the triangle defining $\bL_{\cU^{\vee}_Y}(I_{C/Y}(H))$, we have $$\cH^{-1}(\bL_{\cU^{\vee}_Y}(I_{C/Y}(H)))\cong \oh_Y$$ and an exact sequence
\[0\to \oh_L(-H)\to \cH^{0}(\bL_{\cU^{\vee}_Y}(I_{C/Y}(H)))\to \cU^{\vee}_Y\to 0.\]
Applying $\bL_{\oh_Y}$ to $\bL_{\cU^{\vee}_Y}(I_{C/Y}(H))$, we get the desired triangle \eqref{eq-pr-cubic}. Note that \eqref{eq-pr-cubic} is non-splitting, otherwise, $\Hom_Y(\cQ^{\vee}_Y[1], \pr_Y(I_{C/Y}(H)))\neq 0$ and this contradicts the fact that $\pr_Y(I_{C/Y}(H))\in \Ku(Y)$. Finally, it is easy to see that $\pr_Y(\oh_L(-H))$ also fits into \eqref{eq-pr-cubic}. Then $\pr_Y(I_{C/Y}(H))=\pr_Y(\oh_L(-H))$ follows from $\Hom_Y(\cQ^{\vee}_Y[1], \oh_L(-H)[1])=\CC$.
\end{proof}

\begin{remark}\label{rmk-cubic-pr}
In the case (1) of Lemma \ref{lem-cubic-mutation}, we know that $I_{C/Y}\in \cA_Y$ and $\pr_Y(I_{C/Y}(H))=\Xi_Y(I_{C/Y})$. Since $\Xi_Y\colon \cA_Y\to \Ku(Y)$ is an equivalence, $\pr_Y(I_{C/Y}(H))\cong \pr_Y(I_{C'/Y}(H))$ if and only if $C=C'$.

In the case (2) of Lemma \ref{lem-cubic-mutation}, $\pr_Y(I_{C/Y}(H))$ is uniquely determined by its residue line $L$. Conversely, for any line $L\subset Y$, since $\Hom_Y(\cU_Y, I_{L/Y})=\CC^2$, $L$ has a $\PP^1=\PP(\Hom_Y(\cU_Y, I_{L/Y}))$-family of  residue twisted cubics. Every pair of different twisted cubics $C$ and $C'$ in this family admits the isomorphism  $\pr_Y(I_{C/Y}(H))\cong \pr_Y(I_{C'/Y}(H))$.
\end{remark}

Now we are going to prove the stability of projection objects. We start with a computation of $\Ext$-groups.

\begin{lemma} \label{lem-cubic-ext1}
Let $[C]\in \Hilb_Y^{3t+1}$. If $\RHom_Y(\cU_Y, I_{C/Y})\neq 0$, then 
\[\RHom_Y(\pr_Y(I_{C/Y}(H)), \pr_Y(I_{C/Y}(H)))=\CC\oplus \CC^{3+\delta}[-1]\oplus \CC^{\delta}[-2]\]
for $0\leq \delta\leq 1$. If  $[L]\in \Hilb^{t+1}_Y$ is a smooth point for the residue line $L$ corresponding to $C$, then $\delta=0$.
\end{lemma}

\begin{proof}
We know that $\Hom_Y(\oh_L,\oh_L[1])=\Hom(I_{L/Y},\oh_L)=H^0(N_{L/Y})$ for any line $L\subset Y$. By \cite[Lemma 4.2.1]{iskovskikh1999fano}, we have $N_{L/Y}\cong \oh_L\oplus \oh_L(-1)$ or $\oh_L(1)\oplus \oh_L(-2)$. Hence, 3$\Hom_Y(\oh_L,\oh_L[1])=\CC$ or~$\CC^2$. Moreover,  $\Hom_Y(\oh_L,\oh_L[1])=\CC$ if and only if $[L]\in \Hilb^{t+1}_Y$ is a smooth point.

Since $\pr_Y(I_{C/Y}(H))\in \Ku(Y)$, we see 
\[\RHom_Y(\pr_Y(I_{C/Y}(H)), \pr_Y(I_{C/Y}(H)))=\RHom_Y(\oh_L(-H), \pr_Y(I_{C/Y}(H))).\]
Note that $\RHom_Y(\oh_L(-H), \cQ^{\vee}_Y[1])=\CC^2[-1]$. Then the result follows from applying the functor $\Hom_Y(\oh_L(-H),-)$ to the exact triangle \eqref{eq-pr-cubic}.
\end{proof}

In the following, we will employ tilt-stability and the wall-crossing argument. We adapt the notion introduced in \cite[Section 2]{feyzbakhsh2023new}.

\begin{proposition} \label{lem-cubic-stability}
Let $[C]\in \Hilb_Y^{3t+1}$. Then $\pr_Y(I_{C/Y}(H))$ is $\sigma_Y$-stable for any Serre-invariant stability condition $\sigma_Y$ on $\Ku(Y)$.  
\end{proposition}

\begin{proof}
When $\RHom_Y(\cU_Y, I_{C/Y})\neq 0$, the $\sigma_Y$-stability of $\pr_Y(I_{C/Y}(H))$ follows from Lemma \ref{lem-cubic-ext1} and Lemma \ref{lem-ext1=4}.

Now we assume that $\RHom_Y(\cU_Y, I_{C/Y})=0$. In this case, we have $I_{C/Y}\in \cA_Y$. It is enough to prove that $I_{C/Y}$ is stable with respect to stability conditions constructed in Theorem \ref{blms-induce}. Since $\ch^{-1}_1(I_{C/Y})=1$, we see $I_{C/Y}$ is $\nu_{-1,w}$-stable for any $w>\frac{1}{2}$. 

Assume that there is a wall for $I_{C/Y}$ connecting points $\Pi(I_{C/Y})$ and $\Pi(\oh_Y(-H)[1])=(-1,\frac{1}{2})$, and let $A\to I_{C/Y}\to B$ be the destabilizing sequence. By \cite[Proposition 3.1]{liu-ruan:cast-bound}, $A$ is a torsion-free sheaf, thus $\ch_0(A)\geq 1$. According to  \cite[Proposition 4.16]{bayer2020desingularization}, we have $\ch_0(A)\leq 2$. Since $$\ch_1^{-1+\epsilon}(I_{C/Y})> \ch_1^{-1+\epsilon}(A)> 0$$ for any $0<\epsilon\ll 1$ sufficiently small, we only have $\ch_{\leq 1}(A)=(2,-H)$ and $\ch_{\leq 1}(B)=(-1,H)$. By definition, $\Pi(\oh_Y(-H)[1])$ lies on the line connecting points $\Pi(B)$ and $\Pi(I_{C/Y})$, then  $\Pi(B)=\Pi(\oh_Y(-H)[1])$ and $\ch_{\leq 2}(B)=\ch_{\leq 2}(\oh_Y(-H)[1])$. A standard argument shows that $B\cong \oh_Y(-H)[1]$. However, this contradicts Lemma \ref{lem-cubic-rhom} by Serre duality. Thus, $I_{C/Y}$ is $\nu_{-1+\epsilon,w}$-stable for any $w>\frac{(-1+\epsilon)^2}{2}$ and any $0<\epsilon\ll 1$ sufficiently small. By \cite[Proposition 4.1]{FeyzbakhshPertusi2021stab}, $I_{C/Y}$ is $\sigma^0_{\alpha,-1+\epsilon}$-stable for any $\alpha>0$. Then the result follows from Theorem \ref{blms-induce}.
\end{proof}

\begin{lemma}\label{lem-ext1=4}
If $E\in\Ku(Y)$ satisfies either  $\mathrm{ext}^1_Y(E,E)\leq 3$, or $\RHom_Y(E,E)=\CC\oplus \CC^4[-1]\oplus \CC[-2]$, then $E$ is stable with respect to every Serre-invariant stability condition $\sigma_Y$ on $\Ku(Y)$.
\end{lemma}

\begin{proof}
When $\ext^1(E,E)\leq 3$, the result follows from \cite[Lemma 4.12(1)]{FGLZ24}. Now we assume the later case holds. By \cite[Lemma 3.6(3), Lemma 3.5]{FGLZ24}, if $E$ is strictly $\sigma_Y$-semistable, then $E$ has two Jordan--H\"older factors $A,B$ with $\ext^1_Y(A,A)=\ext^1_Y(B,B)=2$. Since $A$ and $B$ have the same phase, we know that up to sign, $[A]=[B]=\lambda_1$ or $\lambda_2$. However, this implies $\chi(E, E)=-4$, leading to a contradiction. Thus, we only need to show that $E$ is $\sigma_Y$-semistable. 

If $E$ is not $\sigma_Y$-semistable, using \cite[Lemma 3.6(3), Lemma 3.5]{FGLZ24} again, $E$ has two Harder--Narasimhan factors $A,B$ with an exact triangle $A\to E\to B$ such that $$\phi_{\sigma_Y}(A)>\phi_{\sigma_Y}(B), \quad\ext^1_Y(A,A)=\ext^1_Y(B,B)=2.$$ 
Hence,  we have $\RHom_Y(A,A)=\RHom_Y(B,B)=\CC\oplus \CC^2[-1]$ with $\{[A],[B]\}\subset \{\lambda_1,\lambda_2,-\lambda_1,-\lambda_2\}$. As $\Hom_Y(A,B)=\Hom_Y(B,A[2])=0$, applying \cite[Lemma 3.12]{FGLZ24} to $B[-1]\to A\to E$ and using $\RHom_Y(E,E)=\CC\oplus \CC^4[-1]\oplus \CC[-2]$, we deduce that $\RHom_Y(B,A)=\CC[-1]$. This contradicts $\chi(B, A)=0$.
\end{proof}

\subsection{Moduli spaces}
In this subsection, we aim to describe the Bridgeland moduli space $M_{\sigma_Y}^Y(1,-1)$.

\begin{proposition}
Let $\sigma_Y$ be a Serre-invariant stability condition on $\Ku(Y)$. Then we have a closed immersion
\[p_Y'\colon \Hilb^{t+1}_Y\hookrightarrow M_{\sigma_Y}^Y(1,-1), \quad [L]\mapsto [\cone(\cQ^{\vee}_Y\to \oh_L(-H))].\]
\end{proposition}

\begin{proof}
    By Lemma \ref{lem-cubic-mutation}, for those $[C]\in \Hilb^{3t+1}_Y$ such that $\RHom_Y(\cU_Y, I_{C/Y})\neq 0$, the object $\pr_Y(I_{C/Y}(H))$ is uniquely determined by its residue line. Conversely, by Remark \ref{rmk-cubic-pr}, every line has a $\PP^1$-family of residue twisted cubics with the same projection object in $\Ku(Y)$. Hence,  using Proposition \ref{lem-cubic-stability}, we obtain an injective morphism
\[p'_Y\colon \Hilb^{t+1}_Y\hookrightarrow M^Y_{\sigma_Y}(1,-1).\]
Since both moduli spaces are proper, to show that $p_Y'$ is a closed immersion, it suffices to prove that the tangent map of $p_Y'$ is injective. To this end, for any $[L]\in \Hilb^{t+1}_Y$, there is a natural identification $T_{[L]}\Hilb^{t+1}_Y=\Ext^1_Y(\oh_L, \oh_L)$. Let $C$ be a residue twisted cubic of $L$. Applying $\Hom_Y(\oh_L(-H), -)$ to \eqref{eq-pr-cubic}, we get a morphism
\[d\colon \Ext^1_Y(\oh_L, \oh_L)\to \Ext^1_Y(\oh_L(-H), \pr_Y(I_{C/Y}(H))).\]
Since $$\Ext^1_Y(\oh_L(-H), \pr_Y(I_{C/Y}(H)))=\Ext^1_Y(\pr_Y(I_{C/Y}(H)), \pr_Y(I_{C/Y}(H)))=T_{[\pr_Y(I_{C/Y}(H))]}M_{\sigma_Y}^Y(1,-1),$$
the map $d$ can be identified with the tangent map of $p_Y'$ at the point $[L]$. Then it is injective because $\Hom_Y(\oh_L(-H), \cQ^{\vee}_Y[1])=0$.
\end{proof}

By Proposition \ref{lem-cubic-stability}, the functor $\pr_Y$ induces a morphism
\[p_Y\colon \Hilb_Y^{3t+1}\to M_{\sigma_Y}^Y(1,-1).\]
We define $\cS:=p_Y^{-1}(\Hilb^{t+1}_Y)$. It is a closed subscheme of $\Hilb^{3t+1}_Y$ parameterizing those $[C]\in \Hilb_Y^{3t+1}$ such that $\Hom_Y(\cU_Y, I_{C/Y})\neq 0$.

\begin{theorem} \label{thm-cubic-3fold}
Let $\sigma_Y$ be a Serre-invariant stability condition on $\Ku(Y)$. Then the functor $\pr_Y$ induces a morphism
\[p_Y\colon \Hilb_Y^{3t+1}\to M_{\sigma_Y}^Y(1,-1),\quad [C]\mapsto [\pr_Y(I_{C/Y}(H))]\]
such that $p_Y$ is an open immersion outside $\cS$. Moreover, $p_Y(\cS)\cong \Hilb^{t+1}_Y$ and $p_Y|_{\cS}\colon \cS\to \Hilb^{t+1}_Y$ is a $\PP^1$-bundle.
\end{theorem}

\begin{proof}
By Proposition \ref{lem-cubic-stability}, the functor $\pr_Y$ induces a morphism
\[p_Y\colon \Hilb_Y^{3t+1}\to M_{\sigma_Y}^Y(1,-1).\]
According to Remark \ref{rmk-cubic-pr}, $p_Y$ is injective outside $\cS$. To show that $p_Y|_{\Hilb_Y^{3t+1}\setminus\cS}$ is an open immersion, it is enough to show the tangent map at any point outside $\cS$ is an isomorphism.

To this end, let $[C]\in \Hilb_Y^{3t+1}\setminus\cS$. By Serre duality and Lemma \ref{lem-cubic-rhom}, we have $$\RHom_Y(I_{C/Y}(H), \oh_Y)=0.$$
Hence, applying $\Hom_Y(I_{C/Y}(H),-)$ to the exact triangle defining $\bL_{\oh_Y}(I_{C/Y}(H))$, we get an isomorphism
\[\Ext^1_Y(I_{C/Y}(H), I_{C/Y}(H))\xra{\cong }\Ext^1_Y(I_{C/Y}(H),\bL_{\oh_Y}(I_{C/Y}(H))),\]
which can be identified with the tangent map of $p_Y$ at the point $[C]$ because of the identification
\[\Ext^1_Y(I_{C/Y}(H),\bL_{\oh_Y}(I_{C/Y}(H)))=\Ext^1_Y(\bL_{\oh_Y}(I_{C/Y}(H)), \bL_{\oh_Y}(I_{C/Y}(H))).\]
Hence, $p_Y|_{\Hilb_Y^{3t+1}\setminus\cS}$ is an open immersion. The last statement of the theorem follows from the construction of $\cS$ and Remark \ref{rmk-cubic-pr}.
\end{proof}

Finally, we give a smoothness criterion for $\Hilb_Y^{3t+1}$.

\begin{corollary} \label{cor-hilb-smooth}
If $\Hilb_Y^{t+1}$ and $\Hilb_Y^{3t+1}\setminus \cS\cong p_Y(\Hilb_Y^{3t+1}\setminus \cS)$ are both smooth, then $\Hilb_Y^{3t+1}$ is  smooth. In particular, if $Y$ is general, $\Hilb_Y^{3t+1}$ is a smooth irreducible threefold.
\end{corollary}

\begin{proof}
When $Y$ is general, $\Hilb^{t+1}_Y$ is a smooth irreducible curve by \cite[Theorem 4.2.7]{iskovskikh1999fano} and 
 $M_{\sigma_Y}^Y(1,-1)$ is smooth of dimension $3$ by \cite[Proposition 5.1]{FGLZ24}. Hence, according to Theorem \ref{thm-cubic-3fold}, we deduce that $\Hilb_Y^{3t+1}\setminus \cS\cong p_Y(\Hilb_Y^{3t+1}\setminus \cS)$ is smooth of dimension $3$. Thus, we only need to prove the first statement of the corollary. By assumptions, to show that $\Hilb_Y^{3t+1}$ is smooth, it remains to prove $\Ext^1_Y(I_{C/Y}, I_{C/Y})=\CC^3$ for all $[C]\in \cS$.
 
 To this end, assume $[C]\in \cS$. According to Lemma \ref{lem-cubic-U}, we have a non-zero map $s\colon \cU^{\vee}_Y\to I_{C/Y}(H)$ with $\im(s)=I_{D/Y}(H)$, $\ker(s)=\oh_Y$ and $\mathrm{coker}(s)=\oh_L(-H)$, where $L$ is the residue line of $C$ in $D$. Since $\RHom_Y(I_{C/Y}(H),\oh_Y)=0$ by Lemma \ref{lem-cubic-rhom}, we see
 \[\RHom_Y(I_{C/Y}(H),\cone(s))=\RHom_Y(I_{C/Y}(H),\oh_L(-H)).\]
Applying $\Hom_Y(-,\oh_L(-H))$ to $\oh_Y\to \cU^{\vee}_Y\to I_{D/Y}(H)$, we have
\[\RHom_Y(I_{D/Y}(H),\oh_L(-H))=\RHom_Y(\cU^{\vee}_Y,\oh_L(-H))=\CC[-1].\]
Since $\Hilb^{t+1}_Y$ is a smooth curve,  $\RHom_Y(\oh_L,\oh_L)=\CC\oplus \CC[-1]$. Thus, after  applying the functor $\Hom_Y(-,\oh_L(-H))$ to the sequence $I_{D/Y}(H)\to I_{C/Y}(H)\to \oh_L(-H)$, we obtain
\[\RHom_Y(I_{C/Y}(H), \oh_L(-H))=\CC\oplus \CC^2[-1].\]

Finally, as $\RHom_Y(I_{C/Y}(H), \cone(s))=\RHom_Y(I_{C/Y}(H), \oh_L(-H))=\CC\oplus \CC^2[-1]$, we deduce that $\Ext^1_Y(I_{C/Y}, I_{C/Y})=\CC^3$. Indeed, this follows from Lemma \ref{lem-cubic-UC} and applying $\Hom_Y(I_{C/Y}(H), -)$ to the triangle $\cU^{\vee}_Y\xra{s}I_{C/Y}(H)\to \cone(s).$ Thus, we conclude that $\Hilb^{3t+1}_Y$ is smooth of dimension $3$.

Let $U$ be the open subscheme of $\Hilb^{3t+1}_Y$ parameterizing irreducible smooth cubics. It is proved in \cite[Lemma 7.3, Theorem 7.4]{lt} that $U$ is irreducible when $Y$ is general. Thus, to prove the irreducibility of $\Hilb^{3t+1}_Y$, we only need to show $\dim \Hilb^{3t+1}_Y\setminus U<3$. A non-reduced twisted cubic is determined by its reduction. Hence,  the locus of such cubics is parameterized by a subscheme of $\Hilb^{t+1}_Y\cup \Hilb^{2t+1}_Y$, whose dimension is smaller than $3$ since $\dim \Hilb^{t+1}_Y=1$ and $\dim \Hilb^{2t+1}_Y=2$. A reducible twisted cubic is a union of a line and a conic, intersecting at a single point. Then the locus of reducible cubics is parameterized by the image of morphism $$\mathcal{I}\to \Hilb^{3t+1}_Y, \quad(C,L)\mapsto C\cup L,$$ 
where $\mathcal{I}\subset \Hilb_Y^{t+1}\times \Hilb_Y^{2t+1}$ is the incidence variety of dimension smaller than $3$. This proves the irreducibility part of the corollary.
\end{proof}

\end{appendix}

\bibliography{cube}

\bibliographystyle{alpha}

\end{document}